\documentclass[11pt]{article}

\usepackage[margin=1in]{geometry}
\usepackage{mathtools}
\usepackage{calc}
\usepackage{amsmath,amsfonts,amssymb,amsthm}
\usepackage{bm}
\usepackage{enumitem,multirow}
\usepackage{caption}
\usepackage{subcaption}
\usepackage[dvipsnames]{xcolor}
\usepackage{multirow}
\usepackage{graphicx}
\usepackage{makecell}
\usepackage{booktabs}
\usepackage{array}
\usepackage{url}
\usepackage{algorithm}
\usepackage{algorithmic}
\usepackage{dsfont}
\usepackage[hidelinks]{hyperref}
\usepackage{nicefrac}

\hypersetup{
    colorlinks,
    citecolor=blue,
    filecolor=black,
    linkcolor=black,
    urlcolor=black
}

\newcommand{\fb}{\mathbf{f}}
\newcommand{\gb}{\mathbf{g}}

\newcommand{\rb}{\mathbf{r}}

\newcommand{\ub}{\mathbf{u}}
\newcommand{\vb}{\mathbf{v}}

\newcommand{\yb}{\mathbf{y}}

\newcommand{\ba}{\bm{a}}
\newcommand{\bb}{\bm{b}}

\newcommand{\be}{\bm{e}}
\newcommand{\bbf}{\bm{f}}
\newcommand{\bg}{\bm{g}}

\newcommand{\bell}{\bm{\ell}}

\newcommand{\br}{\bm{r}}

\newcommand{\bu}{\bm{u}}
\newcommand{\bv}{\bm{v}}

\newcommand{\bz}{\bm{z}}

\newcommand{\Ab}{\mathbf{A}}
\newcommand{\Bb}{\mathbf{B}}

\newcommand{\Fb}{\mathbf{F}}
\newcommand{\Gb}{\mathbf{G}}

\newcommand{\Ob}{\mathbf{O}}

\newcommand{\Qb}{\mathbf{Q}}

\newcommand{\Ub}{\mathbf{U}}
\newcommand{\Vb}{\mathbf{V}}
\newcommand{\Wb}{\mathbf{W}}
\newcommand{\Xb}{\mathbf{X}}

\newcommand{\Zb}{\mathbf{Z}}

\newcommand{\bB}{\mathbf{B}}

\newcommand{\bD}{\mathbf{D}}
\newcommand{\bE}{\mathbf{E}}
\newcommand{\bF}{\mathbf{F}}
\newcommand{\bG}{\mathbf{G}}

\newcommand{\bL}{\mathbf{L}}

\newcommand{\bO}{\mathbf{O}}

\newcommand{\bQ}{\mathbf{Q}}
\newcommand{\bR}{\mathbf{R}}
\newcommand{\bS}{\mathbf{S}}

\newcommand{\bU}{\mathbf{U}}
\newcommand{\bV}{\mathbf{V}}
\newcommand{\bW}{\mathbf{W}}
\newcommand{\bX}{\mathbf{X}}
\newcommand{\bY}{\mathbf{Y}}
\newcommand{\bZ}{\mathbf{Z}}

\newcommand{\cE}{\mathcal{E}}

\newcommand{\cI}{\mathcal{I}}

\newcommand{\cN}{\mathcal{N}}

\newcommand{\cS}{{\mathcal{S}}}
\newcommand{\cT}{{\mathcal{T}}}

\newcommand{\cX}{\mathcal{X}}

\newcommand{\EE}{\mathbb{E}}

\newcommand{\RR}{\mathbb{R}}

\newcommand{\bgamma}{\bm{\gamma}}

\newcommand{\blambda}{\bm{\lambda}}
\newcommand{\bmu}{\bm{\mu}}
\newcommand{\bnu}{\bm{\nu}}

\newcommand{\brho}{\bm{\varrho}}

\newcommand{\bphi}{\bm{\phi}}
\newcommand{\bvarphi}{\bm{\varphi}}

\newcommand{\bpsi}{\bm{\psi}}

\newcommand{\bGamma}{\bm{\Gamma}}
\newcommand{\bDelta}{\bm{\Delta}}
\newcommand{\bTheta}{\bm{\Theta}}
\newcommand{\bLambda}{\bm{\Lambda}}
\newcommand{\bXi}{\bm{\Xi}}

\newcommand{\bSigma}{\bm{\Sigma}}
\newcommand{\bUpsilon}{\bm{\Upsilon}}
\newcommand{\bPhi}{\bm{\Phi}}
\newcommand{\bPsi}{\bm{\Psi}}
\newcommand{\bOmega}{\bm{\Omega}}

\DeclareMathOperator{\Var}{{\rm Var}}					   
\DeclareMathOperator{\Cov}{\rm Cov}						   
\DeclareMathOperator{\ind}{\mathds{1}}   				   

\newcommand{\diag}{{\rm diag}}

\newcommand{\F}{\mathrm{F}}

\newcommand{\rbr}[1]{\left(#1\right)}

\DeclareMathOperator*\uplim{lim\,sup}

\ifx\theorem\undefined
\newtheorem{theorem}{Theorem}
\fi
\ifx\example\undefined
\newtheorem{example}[theorem]{Example}
\fi
\ifx\property\undefined
\newtheorem{property}{Property}
\fi
\ifx\lemma\undefined
\newtheorem{lemma}[theorem]{Lemma}
\fi
\ifx\proposition\undefined
\newtheorem{proposition}[theorem]{Proposition}
\fi
\ifx\corollary\undefined
\newtheorem{corollary}[theorem]{Corollary}
\fi
\ifx\conjecture\undefined
\newtheorem{conjecture}[theorem]{Conjecture}
\fi
\ifx\fact\undefined
\newtheorem{fact}[theorem]{Fact}
\fi
\ifx\claim\undefined
\newtheorem{claim}[theorem]{Claim}
\fi
\ifx\assumption\undefined
\newtheorem{assumption}[theorem]{Assumption}
\fi

\theoremstyle{definition}

\ifx\definition\undefined
\newtheorem{definition}[theorem]{Definition}
\fi

\ifx\remark\undefined
\newtheorem{remark}[theorem]{Remark}
\fi

\numberwithin{equation}{section}
\numberwithin{theorem}{section}

\makeatletter
\DeclareRobustCommand\widecheck[1]{{\mathpalette\@widecheck{#1}}}
\def\@widecheck#1#2{%
    \setbox\z@\hbox{\m@th$#1#2$}%
    \setbox\tw@\hbox{\m@th$#1%
      \widehat{%
          \vrule\@width\z@\@height\ht\z@
          \vrule\@height\z@\@width\wd\z@}$}%
    \dp\tw@-\ht\z@
    \@tempdima\ht\z@ \advance\@tempdima2\ht\tw@ \divide\@tempdima\thr@@
    \setbox\tw@\hbox{%
      \raise\@tempdima\hbox{\scalebox{1}[-1]{\lower\@tempdima\box
\tw@}}}%
    {\ooalign{\box\tw@ \cr \box\z@}}}
\makeatother

\makeatletter
\let\save@mathaccent\mathaccent
\newcommand*\if@single[3]{%
  \setbox0\hbox{${\mathaccent"0362{#1}}^H$}%
  \setbox2\hbox{${\mathaccent"0362{\kern0pt#1}}^H$}%
  \ifdim\ht0=\ht2 #3\else #2\fi
  }
\newcommand*\rel@kern[1]{\kern#1\dimexpr\macc@kerna}
\newcommand*\widebar[1]{\@ifnextchar^{{\wide@bar{#1}{0}}}{\wide@bar{#1}{1}}}
\newcommand*\wide@bar[2]{\if@single{#1}{\wide@bar@{#1}{#2}{1}}{\wide@bar@{#1}{#2}{2}}}
\newcommand*\wide@bar@[3]{%
  \begingroup
  \def\mathaccent##1##2{%
    \let\mathaccent\save@mathaccent
    \if#32 \let\macc@nucleus\first@char \fi
    \setbox\z@\hbox{$\macc@style{\macc@nucleus}_{}$}%
    \setbox\tw@\hbox{$\macc@style{\macc@nucleus}{}_{}$}%
    \dimen@\wd\tw@
    \advance\dimen@-\wd\z@
    \divide\dimen@ 3
    \@tempdima\wd\tw@
    \advance\@tempdima-\scriptspace
    \divide\@tempdima 10
    \advance\dimen@-\@tempdima
    \ifdim\dimen@>\z@ \dimen@0pt\fi
    \rel@kern{0.6}\kern-\dimen@
    \if#31
      \overline{\rel@kern{-0.6}\kern\dimen@\macc@nucleus\rel@kern{0.4}\kern\dimen@}%
      \advance\dimen@0.4\dimexpr\macc@kerna
      \let\final@kern#2%
      \ifdim\dimen@<\z@ \let\final@kern1\fi
      \if\final@kern1 \kern-\dimen@\fi
    \else
      \overline{\rel@kern{-0.6}\kern\dimen@#1}%
    \fi
  }%
  \macc@depth\@ne
  \let\math@bgroup\@empty \let\math@egroup\macc@set@skewchar
  \mathsurround\z@ \frozen@everymath{\mathgroup\macc@group\relax}%
  \macc@set@skewchar\relax
  \let\mathaccentV\macc@nested@a
  \if#31
    \macc@nested@a\relax111{#1}%
  \else
    \def\gobble@till@marker##1\endmarker{}%
    \futurelet\first@char\gobble@till@marker#1\endmarker
    \ifcat\noexpand\first@char A\else
      \def\first@char{}%
    \fi
    \macc@nested@a\relax111{\first@char}%
  \fi
  \endgroup
}
\makeatother

\newcommand{\bff}{\mathbf{f}}

\def\bB{\mathbf{B}}

\def\bX{\mathbf{X}}
\def\be{\mathbf{e}}
\def\bu{\mathbf{u}}
\def\bv{\mathbf{v}}
\def\br{\mathbf{r}}
\def\bg{\mathbf{g}}
\def\bff{\mathbf{f}}
\def\bW{\mathbf{W}}
\def\bmu{\pmb{\mu}}
\def\bnu{\pmb{\nu}}
\def\bvarphi{\pmb{\varphi}}
\def\bDelta{\pmb{\Delta}}

\def\bPhi{\pmb{\Phi}}
\def\bTheta{\pmb{\Theta}}
\def\bXi{\pmb{\Xi}}
\def\bF{\mathbf{F}}
\def\bO{\mathbf{O}}
\def\bU{\mathbf{U}}
\def\bSigma{\mathbf{\Sigma}}
\def\bLambda{\mathbf{\Lambda}}
\def\blambda{\pmb{\lambda}}

\newcommand{\bbE}{\mathbb{E}}
\newcommand{\bbR}{\mathbb{R}}
\def\bbE{\mathbb{E}}

\def\cN{\mathcal{N}}
\def\calN{\mathcal{N}}

\DeclareMathOperator{\supp}{supp}

\DeclareMathOperator{\Tr}{Tr}
\DeclareMathOperator{\MSE}{MSE}

\newcommand{\toWtwo}{\mathrel{\overset{W_2}{\to}}}
\newcommand{\defeq}{\mathrel{\overset{\mathrm{def}}{=}}}

\newcommand{\pca}{\mathrm{pca}}
\newcommand{\tkappa}{\tilde{\kappa}}
\newcommand{\hkappa}{\hat{\kappa}}

\newcommand{\Id}{\mathrm{Id}}
\newcommand{\eps}{\varepsilon}
\newcommand{\all}{\mathrm{all}}

\def\tkappa{\tilde{\kappa}}
\def\hkappa{\hat{\kappa}}

\newcommand{\makebtau}[1]{\mathbf{#1}^{(\tau)}}

\newcommand{\taua}{a^{(\tau)}}

\newcommand{\taubA}{\makebtau{A}}

\newcommand{\maketau}[1]{{#1}^{(\tau)}}

\newcommand{\tbpsi}{\maketau{\bpsi}}
\newcommand{\taubmu}{\maketau{\bmu}}
\newcommand{\taubnu}{\maketau{\bnu}}

\newcommand{\taubSigma}{\maketau{\bSigma}}

\newcommand{\taubOmega}{\maketau{\bOmega}}
\newcommand{\taubDelta}{\maketau{\bDelta}}
\newcommand{\taubGamma}{\maketau{\bGamma}}
\newcommand{\taubPsi}{\maketau{\bPsi}}
\newcommand{\taubPhi}{\maketau{\bPhi}}

\newcommand{\taubphi}{\maketau{\bphi}}
\newcommand{\taubpsi}{\maketau{\bpsi}}

\newcommand\independent{\protect\mathpalette{\protect\independenT}{\perp}}
\def\independenT#1#2{\mathrel{\rlap{$#1#2$}\mkern2mu{#1#2}}}

\def\Uniform{\textnormal{Uniform}}
\def\Beta{\textnormal{Beta}}

\def\toWtwo{\overset{W_2}{\longrightarrow}}

\def\limn{\lim_{n\to\infty}}

\def\limmn{\lim_{m,n\to\infty}}
\def\limtau{\lim_{\tau\to\infty}}
\def\limtaun{\lim_{\tau\to\infty}\lim_{n\to\infty}}

\def\defeq{\overset{\mathrm{def}}{=}}

\title{Approximate Message Passing for orthogonally invariant ensembles:
Multivariate non-linearities and spectral initialization}

\author{Xinyi Zhong\thanks{These authors contributed equally.\newline
Yale University, Department of Statistics and Data Science.\newline
\texttt{xinyi.zhong@yale.edu}, \texttt{tianhao.wang@yale.edu}, 
\texttt{zhou.fan@yale.edu}}
\and Tianhao Wang\footnotemark[1] \and Zhou Fan}

\date{}

\begin{document}
\maketitle

\begin{abstract}
We study a class of Approximate Message Passing (AMP) algorithms for symmetric
and rectangular spiked random matrix models with orthogonally invariant noise.
The AMP iterates have fixed dimension $K \geq 1$, a multivariate non-linearity
is applied in each AMP iteration, and
the algorithm is spectrally initialized with $K$ super-critical sample
eigenvectors. We derive the forms of the Onsager debiasing coefficients and
corresponding AMP state evolution, which depend on the free cumulants of the
noise spectral distribution. This extends previous results for such
models with $K=1$ and an independent initialization.

Applying this approach to Bayesian principal components
analysis, we introduce a Bayes-OAMP algorithm that uses as
its non-linearity the posterior mean conditional on all preceding AMP iterates. We
describe a practical implementation of this algorithm, where all debiasing
and state evolution parameters are estimated from the observed data, and we
illustrate the accuracy and stability of this approach in simulations.
\end{abstract}

\section{Introduction}

In recent years, Approximate Message Passing (AMP) algorithms have been used in
an increasingly diverse range of applications. These algorithms were originally
derived as approximations to message passing and belief propagation algorithms
for densely connected graphical models~\cite{kabashima2003cdma,donoho2009message,donoho2010messageI,donoho2010messageII}.
They have since been successfully
adapted to perform both optimization and Bayesian inference in many problems
arising in high-dimensional statistics and machine learning, and we refer to
\cite{feng2021unifying} for a recent review.

By design, AMP algorithms are closely tailored to distributional assumptions
for the data matrices to which they are applied. These algorithms exhibit
fast rates of convergence for typical realizations of such random data
\cite{maleki2010approximate}, and can achieve near-optimal estimation risk
in many contexts of Bayesian inference
\cite{krzakala2012probabilistic,deshpande2014information,dia2016mutual,deshpande2017asymptotic,barbier2019optimal,barbier2020mutual}. Furthermore, iterates of AMP
admit an exact asymptotic distributional characterization, known as its
``state evolution'', that is simpler than that of alternative first-order
procedures. Thus AMP has also served as a broadly useful theoretical tool for
analyzing the asymptotic behavior of statistical methods
\cite{bayati2011lasso,maleki2013asymptotic,donoho2013information,donoho2016high,su2017false,sur2019likelihood,bu2020algorithmic}
as well as probabilistic models
\cite{bolthausen2018morita,ding2019capacity,fan2021replica,celentano2021local}.

The most common examples of AMP algorithms are tailored to data
matrices with i.i.d.\ entries, and a line of work
\cite{bolthausen2014iterative,bayati2011dynamics,javanmard2013state,bayati2015universality,berthier2020state,chen2021universality} has rigorously
established the validity of their state evolutions in this context. This state
evolution may no longer correctly describe the iterates for non-i.i.d.\ data, where such AMP algorithms
may also exhibit divergent behavior \cite{vila2015adaptive,rangan2019convergence}. More recently, several AMP
algorithms have been developed for a broader class of random matrices that
are orthogonally or unitarily invariant in law, but that can have arbitrary
spectral distribution. These include the Orthogonal AMP
\cite{takeuchi2017rigorous,ma2017orthogonal},
Vector AMP \cite{rangan2019vector,schniter2016vector}, Convolutional AMP
\cite{takeuchi2019unified,takeuchi2020convolutional,takeuchi2020bayes}, and
Memory AMP \cite{liu2021memory} procedures for linear models and
generalized linear models with orthogonally invariant designs, as well as
a general class of AMP procedures for symmetric and rectangular
orthogonally invariant ensembles studied in
\cite{opper2016theory,ccakmak2020dynamical,fan2020approximate}. Broadly, these
algorithms adapt to the more general spectral laws that may arise in such
models, by either using divergence-free non-linearities or applying
Onsager corrections that are tailored to these spectral laws.

\subsection{Contributions}

Motivated by applications to statistical principal components analysis (PCA),
our work extends the class of AMP algorithms for orthogonally
invariant random matrix ensembles studied in
\cite{opper2016theory,ccakmak2020dynamical,fan2020approximate},
in two directions:

\begin{enumerate}
\item We extend the AMP procedures from vector-valued iterates
$\bu_t \in \RR^n$ to matrix-valued iterates $\bU_t \in \RR^{n \times K}$, for an
arbitrary fixed dimension $K \geq 1$. We consider such AMP
algorithms that apply multivariate non-linearities in
every iteration, and derive the forms of the Onsager corrections and state
evolutions for algorithms of this type.

Importantly, the non-linearities need not be separable
across the $K$ dimensions. This generalization is particularly useful for PCA,
where (empirically, in many domains of application) there
is often joint structure across coordinates of multiple PCs,
and a multivariate non-linearity should be used to regularize
estimates towards this structure \cite{zhong2020empirical}.

\item In a model of data consisting of low-rank signal plus additive noise, we
develop a method of spectral initialization for the general AMP algorithms of
\cite{fan2020approximate}, using the sample eigenvectors or singular vectors
of the data. This is analogous
to the work of \cite{montanari2021estimation} that developed this extension for
AMP algorithms when the noise has i.i.d.\ entries.

Such an extension eliminates the need for an informative initialization that is
independent of the data, which is typically unavailable in practice. Our
analysis shows that the AMP Onsager correction and state evolution must treat the
spectral initialization separately, and they take different forms from 
the descriptions of \cite{fan2020approximate}.
\end{enumerate}

The first generalization above is a more direct extension of the previous
analyses in \cite{fan2020approximate}, and we describe these results in
Section \ref{sec: AMP with indinit}.
The second generalization to a spectral initialization constitutes the larger
technical contribution of our work, and we
describe it for symmetric and rectangular matrices in
Sections \ref{sec:sym:pcaInit} and in Appendices~\ref{sec: main rec} respectively. Our proof is different from that
of \cite{montanari2021estimation}, and instead follows a strategy
introduced in \cite{mondelli2021approximate} of approximating the spectral initialization
by a sequence of {$\tau$} linear AMP steps that converge to the sample eigenvectors,
as ${\tau} \to \infty$.

Recent independent work of \cite{mondelli2021pca} has used this approach to derive also the
forms of spectrally-initialized AMP
algorithms corresponding to the ``single-iterate posterior mean'' PCA procedure
described in \cite{fan2020approximate}; this constitutes an important case of our current
results. Our results expand upon \cite{mondelli2021pca}, analyzing instead the general AMP
algorithms in \cite{fan2020approximate} whose non-linearities may be functions of
all preceding AMP iterates, and in the above context of multivariate iterates
with dimension $K \geq 1$.
At a technical level, we avoid the restrictive assumption imposed
in \cite{mondelli2021pca} that all free cumulants of the noise spectral law be positive, and
we use an alternative strategy for analyzing the convergence of the linear AMP
iterations that leads to a different and explicit assumption of
sufficiently large signal strength.

Finally, as an application of these results, we propose in Section
\ref{sec:numerical} a Bayes-OAMP\footnote{This is different from the 
algorithm called OAMP in \cite{ma2017orthogonal}. Throughout the paper,
we refer to our algorithm specifically as Bayes-OAMP to avoid potential confusion.}
algorithm for PCA with a Bayesian prior for the
PCs. This algorithm differs from the single-iterate posterior mean procedure that 
was analyzed in \cite{fan2020approximate}, computing instead the Bayes posterior
mean based on the multivariate Gaussian joint law of all preceding AMP iterates,
as described by the above state evolution. We demonstrate in 
Section \ref{sec:numerical} that this can yield a sizeable improvement
in estimation accuracy for rotationally invariant noise ensembles
in settings of weak signal strength.

The development and rigorous characterization of Bayes-optimal
estimation procedures for this PCA problem is an interesting open question.
Following the initial posting of our work, \cite{barbier2023fundamental} has
obtained the first results in this direction, deriving a conjectural form of the
Bayes-optimal error in a symmetric rank-1 spiked model using the 
replica method, for certain examples of rotationally-invariant noise matrices
defined by polynomial potentials. The authors of \cite{barbier2023fundamental}
suggested also two AMP methods that numerically attain the conjectured
Bayes-optimal error in these examples, one of which is an extension of
Bayes-OAMP (dubbed ``AMP-AP'') that alternates between posterior mean and
identity nonlinearities. Our current results establish the rigorous state
evolution of both Bayes-OAMP and AMP-AP under a spectral initialization and
sufficiently large signal strength, constituting a first step towards an
analytic characterization of the estimation errors attained by these procedures.

\paragraph{Notational conventions.}
For random variables $X$ and $Y$, $X \independent Y$ denotes that they are
independent. $\|\cdot\|$ denotes the $\ell_2$-norm for vectors and the
$\ell_2 \to \ell_2$ operator norm for matrices. $\|\cdot\|_\F$ is the
Frobenius norm for matrices.
$\diag(v) \in \RR^{K \times K}$ is the diagonal matrix with $v \in \RR^K$ on its
diagonal, and we write $\diag(v) \in \RR^{K \times K'}$ to indicate this matrix right-padded by
$K'-K$ columns of 0. We adopt the convention $M^0=\Id$ for the $0^\text{th}$
power of any square matrix $M$. For a block matrix $M \in \RR^{tK \times tK}$
and $\kappa \in \RR^{K \times K}$, $M \odot \kappa \in \RR^{tK \times tK}$ and
$\kappa \odot M \in \RR^{tK \times tK}$ denote the block-wise right- and left-
multiplication by $\kappa$.

For a matrix $\Ub \in \RR^{n\times K}$ and random vector $U \in \RR^K$,
we write $\Ub\overset{W_2}{\to} U$ for the Wasserstein-2 convergence of the
empirical distribution of rows of $\Ub$, as $n \to \infty$. Letting $u_i$ be the
$i^\text{th}$ row of $\Ub$, this means
\begin{align*}
\lim_{n\to\infty} \frac{1}{n}\sum_{i=1}^n g(u_i)=\EE[g(U)]
\end{align*}
for any continuous function $g:\RR^K \to \RR$ such that $|g(u)|
\leq C(1+\|u\|^2)$ for a constant $C>0$. We write
$\langle \Ub \rangle=n^{-1}\sum_{i=1}^n u_i \in \RR^K$
for the empirical average of rows of $\Ub$.

\section{Symmetric AMP with independent initialization}\label{sec: AMP with indinit}

In this section, we first describe extensions of the AMP algorithms and
state evolution characterizations of \cite{fan2020approximate} from vector-valued to
matrix-valued iterates, for orthogonally invariant matrices and an
independent initialization. 
We will then discuss signal-plus-noise models and
spectral initializations in Section~\ref{sec:sym:pcaInit}.
We focus on the setting of symmetric matrices in the main text for ease of presentation, 
and corresponding results for rectangular matrices are given in 
Appendices \ref{sec:indrect} and \ref{sec: main rec}.

Let $\Wb \in \RR^{n \times n}$ be a symmetric matrix, with
eigen-decomposition $\Wb = \Ob^\top \bLambda \Ob$ where
$\bLambda=\diag(\blambda)$. We
will assume that $\bO$ is a Haar-distributed orthogonal basis of eigenvectors,
so $\bW$ is orthogonally invariant in law.

For fixed dimensions $J \geq 0$ and $K \geq 1$, consider a possible additional
matrix $\bE \in \RR^{n \times J}$ of ``side information'',
and a sequence of Lipschitz functions $u_2,u_3,\ldots$ where each
$u_{t+1}:\RR^{tK+J} \to \RR^K$. (We may set $J=0$ if there is no such side
information.) We consider an AMP algorithm with initialization
$\bU_1 \in \RR^{n \times K}$ independent of $\bW$, having the iterates
\begin{align}
\Zb_t &=  \bW \bU_t - \bU_1 b_{t1}^\top-\bU_2 b_{t2}^\top-\ldots-\bU_t
b_{tt}^\top\label{eq:AMPz}\\
\bU_{t+1} &= u_{t+1}(\Zb_1, \ldots, \Zb_t,\bE).\label{eq:AMPu}
\end{align}
Here $b_{ts} \in \RR^{K \times K}$ is a matrix-valued Onsager debiasing coefficient
for each $s=1,\ldots,t$, and $u_{t+1}(\cdot)$ is applied row-wise to
$(\bZ_1,\ldots,\bZ_t,\bE) \in \RR^{n \times (tK+J)}$ to yield each next iterate
$\bU_{t+1} \in \RR^{n \times K}$.

\paragraph{Debiasing coefficients.}
Let $\partial_s u_{t+1}(Z_1,\ldots,Z_t,E) \in \RR^{K \times K}$ denote the
Jacobian of $u_{t+1}$ in its vector argument $Z_s$, which exists
Lebesgue-a.e.\ since $u_{t+1}(\cdot)$ is Lipschitz
\cite[Theorem 2.2.1]{ziemer2012weakly}. Denote
\[\langle \partial_s \bU_{t+1} \rangle
=\frac{1}{n}\sum_{i=1}^n \partial_s u_{t+1}(z_{1,i},\ldots,z_{t,i},e_i)\]
where $z_{t,i} \in \RR^K$ and $e_i \in \RR^J$ are the 
$i^\text{th}$ rows of $\bZ_t$ and $\bE$.
For each $T \geq 1$, define the $TK\times TK$ block-lower-triangular matrix
\begin{equation}\label{eq:phi}
\bphi_T={\footnotesize\begin{pmatrix} 0 & 0 & \cdots & 0 & 0 \\
\langle \partial_1 \bU_2 \rangle & 0 & \cdots & 0 & 0 \\
\langle \partial_1 \bU_3 \rangle & \langle \partial_2 \bU_3 \rangle
& \cdots & 0 & 0 \\
\vdots & \vdots & \ddots & \vdots & \vdots \\
\langle \partial_1 \bU_T \rangle & \langle \partial_2 \bU_T \rangle
& \cdots & \langle \partial_{T-1} \bU_T \rangle & 0 \end{pmatrix}}.
\end{equation}

Let $\Lambda$ be a random variable on $\RR$ with compact support, which will be
the limit eigenvalue distribution of $\bW$ as $n \to \infty$.
Let $\{\kappa_j\}_{j \geq 1}$ be the free cumulants of $\Lambda$---see
e.g.\ \cite[Section 2.3]{fan2020approximate} for definitions.
Applying the convention $\bphi_t^0=\Id$,
we take the debiasing coefficient matrices $\{b_{ts}\}$ in (\ref{eq:AMPz}) up to
iteration $T$ to be the blocks of
\begin{align*}
\bb_T =\sum_{j=0}^\infty {\kappa_{j+1}} \bphi_T^j \defeq
{\footnotesize\begin{pmatrix}
b_{11} & & & \\
b_{21} & b_{22} & & \\
\vdots & \vdots & \ddots & \\
b_{T1} & b_{T2} & \cdots & b_{TT}
\end{pmatrix}} \in \bbR^{TK\times TK}.
\end{align*}
This may be interpreted as the $R$-transform of $\Lambda$ applied to $\bphi_T$
(cf.\ Section \ref{sec:sympreliminaries}).
Note that this is a finite sum, because $\bphi_T^j=0$ for all $j \geq T$. We
have $b_{tt}=\kappa_1 \Id$
for every $t \geq 1$, which vanishes if $\Lambda$ has mean $\kappa_1=0$.

\paragraph{State Evolution.} This choice of debiasing coefficients leads to
the empirical distribution of rows of $(\bZ_1,\ldots,\bZ_T)$ having an
asymptotically mean-zero multivariate Gaussian limit, as $n \to \infty$ for
any fixed iteration $T$. The limit Gaussian law is described by its covariance
matrix $\bSigma_T \in \RR^{TK \times TK}$, which may be defined recursively via
the following state evolution:

Let $(U_1,E)$ be an initial random vector with $U_1 \in \RR^K$ and $E \in
\RR^J$, representing the limit empirical distribution of rows of
$(\bU_1,\bE)$. Inductively for $t=1,2,3,\ldots$, having defined the joint law of
$(U_1,\ldots,U_t,Z_1,\dots,Z_{t-1},E)$, define the $tK \times tK$ matrices
\begin{align}
\bDelta_t&={\footnotesize\begin{pmatrix} \bbE[  U_1 U_1^\top] & \bbE[  U_1  U_2^\top]
& \cdots & \bbE[  U_1  U_t^\top] \\
\bbE[  U_2 U_1^\top] & \bbE[  U_2 U_2^\top] & \cdots & \bbE[  U_2  U_t^\top]\\
\vdots & \vdots & \ddots & \vdots \\ \bbE[  U_t U_1^\top] & \bbE[  U_t U_2^\top]
& \cdots & \bbE[  U_t  U_t^\top] \end{pmatrix}},\label{eq:Delta}\\
\bPhi_t&={\footnotesize\begin{pmatrix} 0 & 0 & \cdots & 0 & 0 \\
\bbE[ \partial_1  u_2(Z_1,E)] & 0 & \cdots & 0 & 0 \\
\bbE[ \partial_1  u_3(Z_1,Z_2,E)] & \bbE[ \partial_2  u_3(Z_1,Z_2,E)]
& \cdots & 0 & 0 \\
\vdots & \vdots & \ddots & \vdots & \vdots \\
\bbE[ \partial_1  u_t(Z_1,\ldots,Z_{t-1},E)] & \bbE[ \partial_2
u_t(Z_1,\ldots,Z_{t-1},E)]
& \cdots & \bbE[ \partial_{t-1}  u_t(Z_1,\ldots,Z_{t-1},E)] & 0
\end{pmatrix}}.\label{eq:Phi}
\end{align}
Here, $\bPhi_t$ corresponds to the large-$n$ limit of $\bphi_t$ defined in (\ref{eq:phi}). 
Then define the covariance $\bSigma_t$ by
\begin{equation}\label{eq:Sigma}
\bSigma_t = \sum_{j=0}^\infty \bTheta^{(j)}[\bPhi_t, \kappa_{j+2}\bDelta_t]
\quad \text{ where } \quad
\bTheta^{(j)}[\bPhi, \kappa \bDelta] = \sum_{i=0}^j \bPhi^i (\kappa\bDelta)
(\bPhi^\top)^{j-i}.
\end{equation}
Define the next joint law of $(U_1,\ldots,U_{t+1},Z_1,\ldots,Z_t,E)$ by
\begin{equation}\label{eq:SEindinit}
(Z_1,\ldots,Z_t) \sim \cN(0,\bSigma_t) \independent (U_1,E), \quad
U_{s+1}=u_{s+1}(Z_1,\ldots,Z_s,E) \text{ for each } s=1,\ldots,t.
\end{equation}
Under these inductive definitions,
it may be checked that the upper-left $(t-1) \times (t-1)$ blocks of
$\bSigma_t$ coincide with $\bSigma_{t-1}$.

This state evolution characterizes the iterates of the AMP algorithm
(\ref{eq:AMPz}--\ref{eq:AMPu}), under the following assumptions.

\begin{assumption}\label{assump:symW}
The matrix $\bW=\Ob^\top \diag(\blambda)\Ob$ and random variable $\Lambda$ satisfy
\begin{enumerate}[label=(\alph*)]
\item $\Ob$ is random and Haar-distributed over the orthogonal group.
\item $\blambda$ is independent of $\Ob$, and its empirical distribution
converges weakly a.s.\ to $\Lambda$ as $n \to \infty$.
\item $\Lambda$ has compact support $\supp(\Lambda)$. Denoting
$(\lambda_-,\lambda_+)=(\min \operatorname{supp}(\Lambda),
\max \operatorname{supp}(\Lambda))$, we have
$\min(\blambda) \to \lambda_-$ and
$\max(\blambda) \to \lambda_+$ a.s.\ as $n \to \infty$.
\end{enumerate}
\end{assumption}

\begin{assumption}\label{assump:symindinit}
The AMP initialization $\bU_1$, functions $u_2,u_3,\ldots$, and random vectors
$(U_1,E)$ satisfy
\begin{enumerate}[label=(\alph*)]
\item $(\bU_1,\bE) \in \RR^{n \times (K+J)}$ is independent of $\Ob$, and
$(\bU_1,\bE) \toWtwo (U_1,E)$ a.s.\ as $n \to \infty$.
\item Each $u_{t+1}(\cdot)$ is Lipschitz in all arguments.
For each $s=1,\ldots,t$, $\partial_s u_{t+1}(Z_1,\ldots,Z_t,E)$ exists and
is continuous on a set of probability 1 under the law of $(Z_1,\ldots,Z_t,E)$
defined by (\ref{eq:SEindinit}).
\end{enumerate}
\end{assumption}

\begin{theorem}\label{thm:symindinit}
Suppose Assumptions \ref{assump:symW} and \ref{assump:symindinit} hold. For any
$T \geq 1$, consider the AMP algorithm (\ref{eq:AMPz}--\ref{eq:AMPu}) up to
iteration $T$, and define $(U_1,\ldots,U_{T+1},Z_1,\ldots,Z_T,E)$ by the state
evolution (\ref{eq:SEindinit}). Then almost surely as $n \to \infty$,
\[(\bU_1,\ldots,\bU_{T+1},\bZ_1,\ldots,\bZ_T,\bE) \toWtwo
(U_1,\ldots,U_{T+1},Z_1,\ldots,Z_T,E).\]
\end{theorem}

The proof of Theorem \ref{thm:symindinit} is an extension of that of
\cite[Theorem 4.3 and Corollary 4.4]{fan2020approximate}. Compared with
\cite[Corollary 4.4]{fan2020approximate}, Theorem
\ref{thm:symindinit} considers matrix-valued iterates having
dimension $n \times K$, relaxes the needed convergence $(\bU_1,\bE) \to
(U_1,E)$ from Wasserstein-$p$ for all orders $p \geq 1$ to only Wasserstein-2,
and relaxes the continuous-differentiability requirement for each function
$u_{t+1}(\cdot)$ to the weaker condition of
Assumption \ref{assump:symindinit}(b). We describe the modifications of the
proofs of \cite{fan2020approximate} needed to establish Theorem
\ref{thm:symindinit} in Appendix \ref{appendix:indInitproof}.

\begin{remark}\label{remark:indinitestimatekappa_sym}
We have defined $b_{ts}$ in (\ref{eq:AMPz}) using the free cumulants of the 
limit spectral distributions. 
Theorem \ref{thm:symindinit} then also holds for any AMP algorithm where 
$b_{ts}$ are replaced by $b_{ts}'$ such that $\|b_{ts}-b_{ts}'\| \to 0$ a.s.\ as 
$n \to \infty$. 
In particular, they hold if $b_{ts}$ are defined with $\{\kappa_j\}$ 
replaced by consistent estimates of these limit free cumulants.
\end{remark}

\section{Spectral initialization for the symmetric spiked model}\label{sec:sym:pcaInit}

We now develop versions of the preceding AMP algorithms for ``spiked''
signal-plus-noise models, with spectral initialization. Consider a rank-$K'$
symmetric spiked model
\begin{align}\label{eq:sym:rankkModel}
\bX= \sum_{k=1}^{K'} \frac{\theta_k}{n} \bu_*^k {\bu_*^k}^\top + \bW \in
\bbR^{n\times n},
\end{align}
where $\bu_*^1,\ldots,\bu_*^{K'}$ are $K'$ orthogonal ``signal'' eigenvectors,
$\theta_1,\ldots,\theta_{K'}$ are non-zero signal eigenvalues,
and the noise matrix $\bW=\Ob^\top \bLambda \Ob$ is symmetric and
rotationally-invariant in law. We distinguish
this rank $K'$ from the dimension $K$ of the AMP iterates, in anticipation of
applications where $K'$ may be modeled as large, and the practitioner
may wish to apply AMP to estimate small subsets of $K$ signal eigenvectors at a
time. We develop a version of AMP where any $K$ super-critical sample
eigenvectors of $\bX$ may be chosen as the spectral initialization.

In the model (\ref{eq:sym:rankkModel}), we fix the normalization
\begin{equation}\label{eq:unormalization}
\|\bu_*^k\|^2=n, \qquad {\bu_*^j}^\top \bu_*^k=0
\qquad \text{ for all } j \neq k \in \{1,\ldots,K'\}.
\end{equation}
We order the signal components such that the
first $K$ will correspond to the spectral initialization, and the
remaining $K'-K$ are ordered arbitrarily. (Thus $\theta_1,\ldots,\theta_{K'}$
are not sorted, and they may have arbitrary signs.)
Supposing that $K_+$ values $\theta_1,\ldots,\theta_{K'}$ are positive and
$K_-=K'-K_+$ values are negative, we denote by
\begin{equation}\label{eq:lambdaX}
\lambda_1(\Xb),\ldots,\lambda_{K'}(\Xb)
\end{equation}
the largest $K_+$ and smallest $K_-$ sample eigenvalues of $\Xb$, sorted in the same
order as $\theta_1,\ldots,\theta_{K'}$. We denote the
associated sample eigenvectors of $\bX$ by $\bff_\pca^1,\ldots,\bff_\pca^{K'}$,
with the normalization and sign convention
\begin{equation}\label{eq:fpca}
\|\bff_\pca^k\|^2=n, \qquad {\bff_\pca^k}^\top \bu_*^k \geq 0,
\qquad {\bff_\pca^j}^\top \bff_\pca^k=0
\qquad \text{for all } j \neq k \in \{1,\ldots,K'\}.
\end{equation}

\subsection{Preliminaries on sample eigenvectors and the
$R$-transform}\label{sec:sympreliminaries}

\paragraph{Eigenvectors and spectral phase transition.}
For the signal-plus-noise model~\eqref{eq:sym:rankkModel},
the quantitative behavior of the leading (positive and negative) sample
eigenvalues/eigenvectors of $\bX$ and associated phenomena of spectral phase
transitions were studied in
\cite{benaych2011eigenvalues}, extending the work of
\cite{baik2005phase,baik2006eigenvalues,paul2007asymptotics} for models with
i.i.d.\ noise. These depend on the Cauchy-transform of the limit spectral
distribution of $\bW$, and we briefly review these results here.

Let $\Lambda$ be the limit spectral distribution of $\bW$. Denote the Cauchy-transform of $\Lambda$ by
\[G(z) = \bbE[(z - \Lambda)^{-1}] \qquad \text{ for }
z \in (\lambda_+,\infty) \cup (-\infty,\lambda_-)\]
where $\lambda_\pm$ are the endpoints of support of $\Lambda$ as defined in
Assumption \ref{assump:symW}(c).
$G(z)$ is strictly decreasing and positive on $(\lambda_+,\infty)$ and strictly
decreasing and negative on $(-\infty,\lambda_-)$, and hence admits a functional
inverse $G^{-1}(z)$ on $(0,G(\lambda_+)) \cup (G(\lambda_-),0)$ where
$G(\lambda_{\pm})=\lim_{z \to \lambda_{\pm}} G(z)$. For $k \in \{1,\ldots,K'\}$ such that
$1/\theta_k$ belongs to this domain of $G^{-1}(z)$, define
\begin{align}\label{eq:sym:pcasol}
\lambda_{\pca,k} = G^{-1}(1/\theta_k),
\qquad
\mu_{\pca,k}^2  = \frac{-1}{\theta_k^2G'(\lambda_{k,\pca})}.
\end{align}
The following theorem summarizes several results of
\cite[Theorems~2.1 and 2.2]{benaych2011eigenvalues}, which establish the
first-order behavior of the leading sample eigenvalues and eigenvectors.

\begin{theorem}[\cite{benaych2011eigenvalues}]\label{thm:symspike}
Suppose $\bW$ satisfies Assumption \ref{assump:symW}, and
$\theta_1,\ldots,\theta_{K'}$ are distinct and fixed as $n \to \infty$. Then
for each $k \in \{1,\ldots,K'\}$ where $\theta_k>1/G(\lambda_+)\geq 0$ or $\theta_k<1/G(\lambda_-)\leq 0$, almost surely
\begin{align*}
\lim_{n \to \infty}
\lambda_k(\bX)= \lambda_{\pca, k}, \quad
\lim_{n\to\infty} \left( \frac{{\bff_\pca^k}^\top \bu_*^k}{n}
\right)^2 = \mu_{\pca,k}^2,
\quad \lim_{n\to\infty} \left(\frac{{\bff_\pca^k}^\top \bu_*^j}{n}
\right)^2 = 0 \text{ for all } j \in \{1,\ldots,K'\} \setminus \{k\}.
\end{align*}
For each other $k \in \{1,\ldots,K'\}$,
$\lim_{n \to \infty} \lambda_k(\Xb) \in \{\lambda_-,\lambda_+\}$.
\end{theorem}

Thus, for a ``super-critical'' signal eigenvalue $\theta_k$ exceeding the
positive and negative
phase transition thresholds $1/G(\lambda_+)$ and $1/G(\lambda_-)$, the
corresponding sample eigenvalue $\lambda_k(\bX)$ converges to a deterministic
value $\lambda_{\pca,k}$ outside the interval $[\lambda_-,\lambda_+]$,
the sample eigenvector $\bff_\pca^k$ achieves asymptotically non-vanishing
alignment with its corresponding signal vector $\bu_*^k$, 
and it has asymptotically 0 alignment
with the other signal vectors $\bu_*^j$. For a ``sub-critical''
signal eigenvalue $\theta_k$ below these phase transition thresholds,
$\lambda_k(\bX)$ converges to the spectral edges $\lambda_\pm$ of the noise
spectral distribution. Note that
$G(\lambda_+)$ and $G(\lambda_-)$ may be infinite, in which case the phase
transition thresholds are 0, and all signal eigenvalues are super-critical.
Whether this occurs depends on the rates of decay of the distribution of
$\Lambda$ at its spectral edges $\lambda_\pm$, and
this is discussed further in \cite[Proposition 2.4]{benaych2011eigenvalues}.

\paragraph{$R$-transform.} The above first-order limits may be re-expressed via
the $R$-transform and free cumulants of $\Lambda$, which linearize free
addition of independent random matrices. Define
\[R(z) = G^{-1}(z) - \frac{1}{z}\]
on the same domain as $G^{-1}$. Differentiating on both sides yields 
\begin{align*}
R'(z) = \frac{1}{G'(G^{-1}(z))} + \frac{1}{z^2}. 
\end{align*}
For $z$ small enough, $R(z)$ and its derivative
have the convergent series expansions defined through the free cumulants $\{\kappa_j\}_{j\geq 1}$ of $\Lambda$, 
\begin{equation}\label{eq:Rtransformseries}
R(z) = \sum_{j = 0}^{\infty} \kappa_{j+1} z^j,\qquad
R'(z) = \sum_{j = 0}^{\infty} (j+1) \kappa_{j+2} z^{j},
\end{equation}
see e.g.\ \cite[Theorem 17]{mingo2017free}. Thus for sufficiently
large $|\theta_k|$, the limits $\lambda_{\pca,k}$ and $\mu_{\pca,k}^2$ in
\eqref{eq:sym:pcasol} also have the convergent series forms
\begin{align}\label{eq:sigmasq_series_expansion}
\lambda_{\pca,k}&=\theta_k + R(1/\theta_k)
= \theta_k + \sum_{j=0}^\infty \frac{\kappa_{j+1}}{\theta_k^j},
&
1 - \mu_{\pca,k}^2 &= \frac{1}{\theta_k^2}R'(1/\theta_k)=
\sum_{j=0}^\infty (j+1)\frac{\kappa_{j+2}}{\theta_k^{j+2}}.
\end{align}

\subsection{AMP algorithm}

Isolating the first $K \leq K'$ of the (unsorted) signal components for the
spectral initialization, denote
\[S=\diag(\theta_1,\ldots,\theta_K)\in\RR^{K\times K},
\qquad S'=\diag(\theta_1,\ldots,\theta_{K'})\in\RR^{K'\times K'},\]
\[\bF_\pca=(\bff_\pca^1,\ldots,\bff_\pca^K)\in\RR^{n\times K},
\qquad \bU_*'=(\bu_*^1,\ldots,\bu_*^{K'})\in\RR^{n\times K'}.\]
Here we assume that $\theta_1, \ldots, \theta_K$ are known for
simplicity. The following procedure is also applicable when $\theta_1, \ldots, 
\theta_K$ are unknown, as one can consistently estimate these values using 
Theorem \ref{thm:symspike}, and we discuss this estimation in
Section~\ref{subsec:implementation}.
We consider an AMP algorithm with iterates of dimension $n \times K$,
initialized spectrally at
\begin{equation}\label{eq:symPCAinit}
\bU_0=\bF_\pca\,S^{-1}, \qquad \bF_0=\bF_\pca.
\end{equation}
For a sequence of Lipschitz functions $u_1,u_2,\ldots$ where $u_t:\RR^{tK} \to
\RR^K$, this algorithm then iteratively computes for $t \geq 1$
\begin{align}
\bU_t&=u_t(\bF_0,\bF_1,\ldots,\bF_{t-1})\label{eq:AMPPCAu},\\
\bF_t&=\bX \bU_t-\bU_0 b_{t0}^\top-\bU_1 b_{t1}^\top-\ldots-\bU_t b_{tt}^\top.
\label{eq:AMPPCAf}
\end{align}
Thus each $u_t(\cdot)$ may depend on the preceding iterates
$\bF_0,\ldots,\bF_{t-1}$, including the spectral initialization.
An additional matrix of side information may be incorporated into each
$u_t(\cdot)$ as in (\ref{eq:AMPu}), but we omit this here for simplicity. To
ease notation in the analysis, we have shifted the initialization index from
1 to 0.

\paragraph{Debiasing coefficients.}

For each fixed $s \geq 1$, define diagonal matrices
$\tilde\kappa_s,\hat\kappa_s \in \RR^{K \times K}$ by the matrix series
\begin{align}\label{eq:sym:kappaseries}
	\tilde\kappa_s = \sum_{j=0}^{\infty} \kappa_{j+s}S^{-j},\qquad 
	\hat\kappa_s = \sum_{j=0}^{\infty} (j+1) \kappa_{j+s}S^{-j}.
\end{align}
For each $T\geq 1$, define the block-lower-triangular matrix
\begin{align*}
\bphi_T = \Big(\langle \partial_s \bU_r \rangle \Big)_{r,s \in \{0,\ldots,T\}}
\in \RR^{(T+1)K \times (T+1)K}
\end{align*}
with {row blocks} indexed by $r$ and {column blocks} by $s$, where
$\langle \partial_s \bU_r \rangle=n^{-1}\sum_{i=1}^n
\partial_s u_r(f_{0,i},\ldots,f_{r-1,i})$,
$\partial_s u_r \in \RR^{K \times K}$ denotes the Jacobian of
$u_r(f_0,\ldots,f_{r-1})$ in the argument $f_s$, and $\partial_s u_r=0$ for $s
\geq r$. Define the matrix series
\begin{align*}
	\bb_T = \sum_{j=0}^\infty 
	\kappa_{j+1}\bphi_T^j, \qquad
\widetilde\bb_T = \sum_{j=0}^\infty 
	\bphi_T^j \odot \tilde\kappa_{j+1},
\end{align*}
where we recall our notation
$M \odot \tilde\kappa_s$ and $\tilde\kappa_s \odot M$ for the right-
and left- multiplication of each block of $M$.
Indexing blocks by $\{0,\ldots,T\}$ and writing $[t,s]$ to denote the $K \times
K$ submatrix corresponding to row block $t$ and column block $s$, we set the
debiasing coefficients of (\ref{eq:AMPPCAf}) up to iteration $T$ as
\begin{align}\label{eq: sym debias coef spec}
    b_{ts} = \begin{cases}
    \widetilde\bb_T[t,s] & \textnormal{if }s=0,\\
    \bb_T[t,s] & \textnormal{otherwise.}
    \end{cases}
\end{align}

\paragraph{State Evolution.}
The state of this algorithm up to iteration $T$ is characterized by
\[\bmu_T=\begin{pmatrix} \mu_0 \\ \vdots \\ \mu_T \end{pmatrix}
\in \RR^{(T+1)K \times K'}, \qquad
\bSigma_T=\begin{pmatrix} \sigma_{00} & \cdots & \sigma_{0T} \\
\vdots & \ddots & \vdots \\ \sigma_{T0} & \cdots & \sigma_{TT} \end{pmatrix}
\in \RR^{(T+1)K \times (T+1)K}\]
and a corresponding joint law for random vectors
$(U_*',U_0,\ldots,U_{T+1},F_0,\ldots,F_T)$, where $U_*',U_t,F_t$
represent the limit distributions of $\bU_*',\bU_t,\bF_t$, the matrix
$\bSigma_T$ is the covariance
of $(F_0,\ldots,F_T)$ conditional on $U_*'$, and $\bmu_T$ relates the conditional
mean of $(F_0,\ldots,F_T)$ to $U_*'$. These are defined recursively as follows:

Let $U_*' \in \RR^{K'}$ be a random vector satisfying $\EE[U_*'{U_*'}^\top]=\Id$,
representing the limit distribution of rows of $\bU_*'$ under the
normalization (\ref{eq:unormalization}). Set
$\mu_\pca=\diag(\mu_{\pca,1},\ldots,\mu_{\pca,K}) \in \RR^{K \times K'}$, where
the last $K'-K$ columns are 0. Then
$\Id-\mu_\pca\mu_\pca^\top=\diag(1-\mu_{\pca,1}^2,\ldots,1-\mu_{\pca,K}^2) \in
\RR^{K \times K}$. We initialize
\[\bmu_0=\mu_0=\mu_\pca,
\quad \bSigma_0=\sigma_{00}=\Id-\mu_\pca\mu_\pca^\top.\]
Inductively, having defined $(\bmu_{t-1},\bSigma_{t-1})$, 
we define a joint law for $(U_*',U_0,\ldots,U_t,F_0,\ldots,F_{t-1})$ by
\begin{align}
\begin{aligned}
(F_0,\ldots,F_{t-1}) \mid U_*' &\sim
\cN\big(\bmu_{t-1} \cdot U_*',\;\bSigma_{t-1}\big),\\
U_0=S^{-1}F_0, &\quad
U_s = u_s(F_0,\ldots,F_{s-1}) \textnormal{ for } s = 1,\ldots,t.
\end{aligned}
\label{eq:limitDistPCAAMP}
\end{align}
We then define the next mean transformation $\bmu_t$ to have the blocks
\begin{equation}\label{eq:symmu}
\mu_s=\EE[U_s{U_*'}^\top] \cdot S' \text{ for each } s=0,\ldots,t.
\end{equation}
For $s=0$, this may be checked to coincide with the above initialization
$\mu_\pca$.

We decompose the second moment matrix of $(U_0,\ldots,U_t)$ into four parts,
\begin{align}
\bDelta_t\defeq &\underbrace{{\footnotesize\begin{pmatrix}
		0 & 0 & \cdots & 0 \\
		0 & \bbE[U_1U_1^\top] & \cdots & \bbE[U_1U_t^\top] \\
		\vdots & \vdots & \ddots & \vdots\\
		0 & \bbE[U_tU_1^\top] & \cdots & \bbE[U_tU_t^\top] \\
	\end{pmatrix}}}_{\widebar{\bDelta}_t} + \underbrace{{\footnotesize\begin{pmatrix}
		0 & \bbE[U_0U_1^\top] & \cdots & \bbE[U_0U_t^\top] \\
		0 & 0 & \cdots & 0 \\
		\vdots & \vdots & \ddots & \vdots\\
		0 & 0 & \cdots & 0\\
	\end{pmatrix}}}_{\widetilde\bDelta_t}\notag\\
&\hspace{2in} + \underbrace{{\footnotesize\begin{pmatrix}
		0 & 0 & \cdots & 0 \\
		\bbE[U_1U_0^\top] & 0 & \cdots & 0 \\
		\vdots & \vdots & \ddots & \vdots\\
		\bbE[U_tU_0^\top] & 0 & \cdots & 0\\
	\end{pmatrix}}}_{\widetilde\bDelta_t^\top}
+ \underbrace{{\footnotesize\begin{pmatrix}
		\bbE[U_0U_0^\top] & 0 & \cdots & 0 \\
		0 & 0 & \cdots & 0 \\
		\vdots & \vdots & \ddots & \vdots\\
		0 & 0 & \cdots & 0 \\
	\end{pmatrix}}}_{\widehat\bDelta_t},\label{eq: sym Delta_t}
\end{align}
set
\[\bDelta_t^{(j)}=\kappa_{j+2}\widebar\bDelta_t +
\tilde\kappa_{j+2}\odot\widetilde\bDelta_t +
\widetilde\bDelta_t^\top\odot\tilde\kappa_{j+2} +
\hat\kappa_{j+2}\odot\widehat\bDelta_t,\]
and define analogously to~\eqref{eq:phi}
\begin{align*}
	\bPhi_t = \Big(\EE[\partial_s u_r(F_0,\ldots,F_{r-1})]\Big)_{r,s \in
	\{0,\ldots,t\}} \in \RR^{(t+1)K \times (t+1)K}.
\end{align*}
Then, recalling the function $\bTheta^{(j)}[\cdot,\cdot]$ from (\ref{eq:Sigma}),
we define the next covariance matrix $\bSigma_t$ by
\begin{align}\label{def: sym Sigma_T}
\bSigma_t = \sum_{j=0}^\infty \bTheta^{(j)}[\bPhi_t,\bDelta_t^{(j)}].
\end{align}
It may be checked from these definitions that the first $t$ blocks of $\bmu_t$
coincide with $\bmu_{t-1}$, and the upper-left $t \times t$ blocks of
$\bSigma_t$ coincide with $\bSigma_{t-1}$.

Our main result in the context of model (\ref{eq:sym:rankkModel})
shows that this state evolution provides a
rigorous characterization of the AMP algorithm
(\ref{eq:AMPPCAu}--\ref{eq:AMPPCAf}) with the spectral initialization
(\ref{eq:symPCAinit}), under the following assumptions.

\begin{assumption}\label{assump:symAMP}
\begin{enumerate}[label=(\alph*)]
\item $\bU_*'=(\bu_*^1,\ldots,\bu_*^{K'})$
is independent of $\bO$, satisfies (\ref{eq:unormalization}), 
and $\bU_*' \toWtwo U_*'$ a.s.\ as $n \to \infty$ where
$\EE[U_*'{U_*'}^\top]=\Id$.
\item Each $u_{t+1}(\cdot)$ is Lipschitz in all arguments. For each
$s=0,\ldots,t$ and all $(\bmu,\bSigma)$ in a sufficiently small open neighborhood of
$(\bmu_t,\bSigma_t)$ defined by (\ref{eq:symmu}) and (\ref{def: sym Sigma_T}),
$\partial_s u_{t+1}(F_0,\ldots,F_t)$ exists and is continuous on
a set of probability 1 under the marginal law of $(F_0,\ldots,F_t)$ defined by
$(F_0,\ldots,F_t) \mid U_*' \sim \cN(\bmu \cdot U_*',\bSigma)$.
\item $\theta_1,\ldots,\theta_{K'}$ are distinct.
For each $k \in \{1,\ldots,K\}$, either
$\theta_k>G(1/\lambda_+) \geq 0$ or $\theta_k<G(1/\lambda_-) \leq 0$,
and there exists some constant $\iota\in(0,1)$ such that
\begin{align}\label{eq:largetheta}
\frac{\max(|\lambda_+|,|\lambda_-|)}{|\theta_k|}
+\sum_{j=1}^\infty \frac{|\kappa_j|}{|\theta_k|^j\cdot\iota^{j-1}}<1.
\end{align}
\end{enumerate}
\end{assumption}

Assumption \ref{assump:symAMP}(c) requires $\theta_1,\ldots,\theta_K$ for the
first $K$ selected signals to be ``super-critical'', as described by
Theorem \ref{thm:symspike}. Furthermore, (\ref{eq:largetheta}) requires each
$|\theta_k|$ to exceed some constant depending only on the law of $\Lambda$. (We
believe that this additional requirement may be an artifact of the proof
technique, and we do not optimize the value of this constant.) Note that this
requirement (\ref{eq:largetheta}) is sufficient to imply that the series
(\ref{eq:Rtransformseries}) for $R(z)$ and $R'(z)$
are absolutely convergent at $z=1/\theta_k$, 
and hence the series (\ref{eq:sym:kappaseries}) defining
$\tilde{\kappa}_s,\hat{\kappa}_s$ are also absolutely convergent.

We remark that we assume for simplicity the distinctness of
$\theta_1,\ldots,\theta_{K'}$. If these signal values are not all distinct, then
the AMP state evolution parameters may not concentrate around deterministic
values, and
this has been discussed and analyzed in \cite{montanari2021estimation} when the
noise matrix is Gaussian.
The study of similar phenomena in our model with rotationally-invariant noise
may also be of interest, but we will not pursue this direction in the current paper.

\begin{theorem}\label{thm:sym}
Consider the symmetric spiked model (\ref{eq:sym:rankkModel}), where
Assumptions \ref{assump:symW} and \ref{assump:symAMP}
hold. For any $T \geq 1$, consider the spectrally initialized AMP algorithm
(\ref{eq:symPCAinit}--\ref{eq:AMPPCAf}) up to iteration $T$, and define
$(U_*',U_0,\ldots,U_{T+1},F_0,\ldots,F_T)$ by \eqref{eq:limitDistPCAAMP}.
Then almost surely as $n\to\infty$,
\begin{align*}
(\bU_*',\bU_0,\ldots,\bU_{T+1}, \bF_0,\ldots,\bF_T) \toWtwo (U_*', U_0,\ldots,
U_{T+1},F_0,\ldots,F_T).
\end{align*}
\end{theorem}

\begin{remark}\label{remark:PCAinitestimatekappa_sym}
As in Remark \ref{remark:indinitestimatekappa_sym},
we have defined $b_{ts}$ in (\ref{eq: sym debias coef spec})
using the free cumulants $\{\kappa_j\}$ of the
limit noise spectral distribution, as well as the true signal values
$\theta_1,\ldots,\theta_K$. Theorem~\ref{thm:sym}
then also holds when $b_{ts}$ are replaced by $b_{ts}'$ such that
$\|b_{ts}-b_{ts}'\| \to 0$ a.s., and in
particular if $\{\kappa_j\}$, $\{\tilde\kappa_j\}$, $\{\hat\kappa_j\}$, and
$\theta_1,\ldots,\theta_K$ are
replaced by consistent estimates of these quantities.
\end{remark}

\subsection{Proof idea}\label{sec: sym proof sketch}
We provide the proof of Theorem \ref{thm:sym} in Appendix
\ref{appendix:symproof}, and describe here the main idea.
We construct an auxiliary AMP algorithm consisting of two phases. For some
$\tau \geq 1$, we index the iterates in the first phase from $-\tau$ to 0. We
consider an initialization $\bU_{-\tau}^{(\tau)}$
that is independent of $\bW$, and
apply linear AMP iterations with $\bU_t^{(\tau)}= \bF_{t-1}^{(\tau)}S^{-1}$ for
$t=-\tau+1,\ldots,0$. This has the effect of implementing a version of the
power method to compute the sample eigenvectors $\bF_\pca$ of $\bX$. A
main step of the proof is to show that as $\tau \to \infty$, the final iterate
$\bF_0^{(\tau)}$ obtained from these linear updates approaches the spectral
initialization $\bF_0=\bF_\pca$.
Here, the assumption in \eqref{eq:largetheta} guarantees the desired
convergence of $\bF_0^{(\tau)}$ and the corresponding state evolution.
Our analysis directly characterizes the alignment between 
$\bF_0^{(\tau)}$ and $\bF_\pca$ (which then implies the convergence of the state
evolution) under general spectral distributions. This is different
from \cite{mondelli2021pca}, which studied directly the asymptotic
state evolution, and instead established the optimal signal-to-noise
condition for convergence of the linear AMP iterations under a more
restrictive class of spectral distributions having positive free cumulants.

In the second phase, starting from $\bF_0^{(\tau)}$, we then apply the same
functions $u_t(\cdot)$ as in the spectrally-initialized
AMP algorithm for $t=1,2,\ldots$.
The combined auxiliary AMP algorithm can thus be summarized as follows:
\begin{align*}
	\bF_t^{(\tau)} = \Xb\bU_t^{(\tau)}-\sum_{s=-\tau}^t \bU_s^{(\tau)}b_{ts}^{(\tau)\top}, \qquad \bU_{t+1}^{(\tau)} = u_{t+1}(\bF_{-\tau}^{(\tau)},\ldots,\bF_t^{(\tau)})
\end{align*}
where
\begin{equation}\label{eq:auxAMP}
	u_{t+1}(\bF_{-\tau}^{(\tau)},\ldots,\bF_t^{(\tau)}) = \begin{cases}
		\bF_t^{(\tau)}S^{-1}
		& \text{ for } -\tau+1 \leq t <0,\\
		u_{t+1}(\bF_0^{(\tau)},\ldots,\bF_t^{(\tau)}) & \text{ for } 0\leq t \leq T.
	\end{cases}
\end{equation}
This AMP algorithm may be analyzed using Theorem \ref{thm:symindinit} with side
information $\bE=\bU_*'$, to yield a state evolution for iterates $t \geq 0$
characterized by some $(\bmu_t^{(\tau)},\bSigma_t^{(\tau)})$, which we describe
explicitly in Corollary \ref{cor:auxAMPSE}.
Then we prove that in the limit $\tau \to \infty$, the iterates of this
auxiliary AMP algorithm for $t \geq 0$
will be close to those of the spectrally-initialized AMP algorithm, in the sense that
\begin{align*}
    \lim_{\tau\to\infty} \limsup_{n\to\infty} \frac{\|\bF_t^{(\tau)} -
\bF_t\|_\F}{\sqrt{n}} = 0, \qquad \lim_{\tau \to\infty} \limsup_{n \to\infty} \frac{\|\bU_t^{(\tau)} - \bU_t\|_\F}{\sqrt{n}} = 0.
\end{align*}
We also prove that as $\tau \to \infty$, the state evolution characterizing
(\ref{eq:auxAMP}) converges to the state evolution described by 
(\ref{eq:symmu}) and (\ref{def: sym Sigma_T}), i.e.\ for each fixed $t \geq 1$,
\begin{align*}
    \limtau \bmu_t^{(\tau)} = \bmu_t\quad \textnormal{and}\quad \limtau \bSigma_t^{(\tau)} = \bSigma_t.
\end{align*}
Combining the above two facts, we then can characterize the behavior of the
AMP algorithm with spectral initialization.

\section{Orthogonal AMP for Bayesian PCA}\label{sec:numerical}

Finally, we discuss in this section an application 
to estimating the signal vectors $\bu_*^k$ in the
preceding symmetric signal-plus-noise model in Eq.~\eqref{eq:sym:rankkModel}. 
Applications for the rectangular model are deferred to 
Appendix~\ref{sec:numerical_rec}.

The distributions of $U_*'\in \RR^{K'}$
for the row-wise limits of $\bU_*'$ may be interpreted as Bayesian
``priors'' for these rows. Assuming these prior distributions are known,
we describe a Bayes-OAMP method in Section \ref{subsec:BayesOAMP} that uses
Bayes posterior-mean denoisers as the non-linearities in the preceding AMP
algorithms. We suggest ways of estimating the Onsager debiasing coefficients
and state evolution parameters in Section \ref{subsec:implementation}, and
illustrate the method in simulation in Section \ref{subsec:simulations}.
In practice, the distributions of $U_*'$ are typically also unknown, but may 
be estimated from the data $\bX$ using empirical Bayes ideas. We
will not consider this additional complexity here, but refer readers
to \cite{zhong2020empirical} for an example of this approach.

Let us describe the Bayes-OAMP method in a setting where the AMP dimension
$K$ and signal rank $K'$ satisfy $K \leq K'$, reflecting applications where $K'$
may be large, and one may wish to estimate smaller
subsets of dependent PCs at a time. In this setting, we consider the following
additional assumption for the laws of $U_*'$.

\begin{assumption}\label{assump:BayesOAMP}
The last $K'-K$ coordinates of $U_*'$ have mean 0, and are independent of
the first $K$ coordinates.
\end{assumption}

In applications, the data $\bX$ is often centered to have row and column
means approximately 0 before applying PCA, leading to PCs also having mean
approximately 0.
The components $\bu_*^k$ should ideally be grouped into
small subsets of dependent signals, with the signals within each subset estimated
together to maximally leverage their joint structure. The above assumption of
exact independence of one such subset $(1,\ldots,K)$ from the remaining signals
$(K+1,\ldots,K')$ is a modeling approximation for this grouping.

Assumption \ref{assump:BayesOAMP} ensures that when AMP is initialized
with the first $K$ signal components, its state evolution will not depend on the
remaining $K'-K$ components: In the symmetric setting of
Section \ref{sec:sym:pcaInit}, let
$U_* \in \RR^K$ be the first $K$ coordinates of $U_*'$. Then it is inductively
verified using Assumption \ref{assump:BayesOAMP} that each mean transformation
$\bmu_t \in \RR^{(t+1)K \times K'}$
has last $K'-K$ columns equal to 0, and the law of each vector $F_t$ and $U_t$
depends on $U_*'$ only via $U_*$. We may then write the
conditional law of $(F_0,\ldots,F_t)$ more simply as
\[(F_0,\ldots,F_t) \mid U_* \sim \cN(\bmu_t \cdot U_*,\,\bSigma_t)\]
where, with slight abuse of notation, we write $\bmu_t \in \RR^{(t+1)K \times
K}$ for its first $K$ columns. 
\subsection{Bayes-OAMP}\label{subsec:BayesOAMP}

In the symmetric spiked model (\ref{eq:sym:rankkModel}), under Assumption
\ref{assump:BayesOAMP}, we consider an AMP algorithm which estimates the first
$K$ components $\bU_* \in \RR^{n \times K}$ of $\bU_*'$, using
as its non-linearities the posterior-mean denoising functions
\[u_{t+1}(f_0,\ldots,f_t)=\EE[U_* \mid (F_0,\ldots,F_t)=(f_0,\ldots,f_t)]\]
computed under the above conditional law
$(F_0,\ldots,F_t) \mid U_* \sim \cN(\bmu_t \cdot U_*,\,\bSigma_t)$.
Explicitly, writing as shorthand $\bbf_t=(f_0,\ldots,f_t) \in \RR^{(t+1)K}$,
\begin{equation}\label{eq:posteriormean}
u_{t+1}(\bbf_t)=\frac{\EE[U_* \exp(-\frac{1}{2}(\bbf_t-\bmu_t \cdot U_*)^\top
\bSigma_t^{-1}(\bbf_t-\bmu_t \cdot U_*))]}
{\EE[\exp(-\frac{1}{2}(\bbf_t-\bmu_t \cdot U_*)^\top
\bSigma_t^{-1}(\bbf_t-\bmu_t \cdot U_*))]} \in \RR^K.
\end{equation}
We will refer to the AMP algorithm (\ref{eq:symPCAinit}--\ref{eq:AMPPCAf}) using
this choice of non-linearity as \emph{Bayes-OAMP}. 

If this posterior mean function $u_{t+1}(\cdot)$ is Lipschitz (which holds
e.g.\ if $U_*$ has bounded support or log-concave density) and the signal
strengths $|\theta_1|,\ldots,|\theta_K|$ are sufficiently large, then
Theorem \ref{thm:sym} implies that the mean-squared-error risk of the estimate
$\bU_t$ has the asymptotic limit
\[\lim_{n \to \infty} \frac{1}{n}\|\bU_t-\bU_*\|_F^2
\defeq \MSE(\bU_t)=\EE\big[\|U_*-\EE[U_* \mid F_0,\ldots,F_t]\|^2\big]
=\Tr \Cov[U_* \mid F_0,\ldots,F_t].\]
Thus these errors satisfy
\begin{enumerate}
\item $\MSE(\bU_{t+1}) \leq \MSE(\bU_t)$, by the property
$\Cov[U_* \mid F_0,\ldots,F_{t+1}] \preceq \Cov[U_* \mid F_0,\ldots,F_t]$. This
implies also by the martingale
property of $U_1,U_2,\ldots$ that the algorithm is convergent in the sense
\[\lim_{n \to \infty}
\frac{1}{n}\|\bU_{t+1}-\bU_t\|_F^2
=\EE\big[\|U_{t+1}-U_t\|^2\big]=\MSE(\bU_t)-\MSE(\bU_{t+1}) \overset{t \to
\infty} \to 0.\]
\item $\MSE(\bU_t) \leq \MSE(\bF_\pca)$ for any iterate $t \geq 1$, because
\[\MSE(\bF_\pca)=\MSE(\bF_0)=\EE[\|U_*-F_0\|_2^2] \geq \EE[\|U_*-\EE[U_* \mid
F_0,\ldots,F_t]\|_2^2]=\MSE(\bU_t).\]
\end{enumerate}

\begin{remark}
This algorithm differs from the Bayes-AMP algorithms
of \cite{rangan2012iterative,montanari2021estimation} for Gaussian noise and
the single-iterate Bayes-OAMP algorithm that was studied in
\cite[Section 3]{fan2020approximate}, which use instead
$u_{t+1}(f_0,\ldots,f_t)=\EE[U_* \mid F_t=f_t]$ based only on the single
preceding iterate $F_t$.

When $\bW$ is symmetric Gaussian noise and $\Lambda$ is distributed as Wigner's
semicircle law, these two approaches coincide: Indeed, in this case
$\kappa_2=1$ and $\kappa_j=0$ for all $j \geq 3$, so (\ref{def: sym Sigma_T})
reduces to $\bSigma_t=\bDelta_t=\EE[UU^\top]$ where $U=(U_0,\ldots,U_t)$,
coinciding with the state evolution shown in \cite{montanari2021estimation}.
For any $t \geq 1$ and $0 \leq s \leq t$, from the martingale identity
\[\EE[U_sU_t^\top]=\EE[U_s \EE[U_*^\top \mid F_0,\ldots,F_t]]
=\EE[\EE[U_sU_*^\top \mid F_0,\ldots,F_t]]=\EE[U_sU_*^\top]\]
and the definition of $\bmu_t$ in (\ref{eq:symmu}),
we observe that $\bSigma_t(\be_t \odot S)=\bmu_t$ and $\sigma_{tt}S=\mu_t$,
where $S=\diag(\theta_1,\ldots,\theta_K) \in \RR^{K \times K}$,
$\be_t \odot S \in \RR^{(t+1)K \times K}$ has first $t$ row
blocks equal to 0 and last row block equal to $S$, and
$\sigma_{tt},\mu_t \in \RR^{K \times K}$ denote the last blocks of
$\bSigma_t$ and $\bmu_t$.
Then $\bSigma_t^{-1}\bmu_t=\be_t \odot S$ and $\sigma_{tt}^{-1}\mu_t=S$.
So (\ref{eq:posteriormean}) is equivalent to the single-iterate posterior mean,
\begin{align*}
u_{t+1}(\bbf_t)&=\frac{\EE[U_*\exp(f_t^\top SU_*-\frac{1}{2}U_*^\top \mu_t^\top
SU_*)]}{\EE[\exp(f_t^\top SU_*-\frac{1}{2}U_*^\top \mu_t^\top SU_*)]}\\
&=\frac{\EE[U_*\exp(-\frac{1}{2}(f_t-\mu_t U_*)^\top \sigma_{tt}^{-1}(f_t-\mu_t
U_*))]}{\EE[\exp(-\frac{1}{2}(f_t-\mu_t U_*)^\top \sigma_{tt}^{-1}(f_t-\mu_t
U_*))]}=\EE[U_* \mid F_t=f_t].
\end{align*}

Outside of this setting where $\Lambda$ has Wigner semicircle law, 
these two approaches are different. The state evolution of Bayes-OAMP is more
involved, and we will not pursue a theoretical characterization
of its fixed point in this work, as was done for single-iterate Bayes-OAMP in
\cite{fan2020approximate}. We observe in simulation that the above
monotonicity property of $\MSE(\bU_t)$ need not hold for single-iterate
Bayes-OAMP in settings of small signal strength, as was observed also in
\cite{mondelli2021pca}, and that Bayes-OAMP can substantially improve over the
single-iterate approach in such settings.
\end{remark}

\subsection{Estimating the debiasing corrections and state evolution}\label{subsec:implementation}
Numerical implementations of the Bayes-OAMP algorithms require estimating the
debiasing coefficients and state evolution parameters that describe the
conditional laws of $(F_0,\ldots,F_t)$. We describe here
one approach for this estimation for the symmetric model.

We estimate the law
of $\Lambda$ by the observed empirical eigenvalue distribution of $\bX$, with
largest $K_+$ positive eigenvalues and smallest $K_-$ negative eigenvalues
removed. This estimate is weakly
consistent as $n \to \infty$ by Weyl's eigenvalue interlacing inequality.
We compute the empirical moments of this law,
and estimate the free cumulants $\{\kappa_s\}$ through the non-crossing
moment-cumulant relations, see e.g.\ \cite[Section 2.3]{fan2020approximate}.
(Alternative methods for estimating the free cumulants of
$\Lambda$ based on power iteration have also been discussed in
\cite[Proposition 1]{liu2021memory} and
\cite[Section 3]{venkataramanan2022estimation}.)
For each signal value $\theta_k$, based on (\ref{eq:sym:pcasol}) and
(\ref{eq:sigmasq_series_expansion}), we then estimate
\[\theta_k \text{ by } \frac{1}{G(\lambda_{\pca,k})},
\quad \mu_{\pca,k}^2 \text{ by } \frac{-1}{\theta_k^2 G'(\lambda_{\pca,k})},
\quad R(1/\theta_k) \text{ by } \lambda_{\pca,k}-\theta_k,
\quad R'(1/\theta_k) \text{ by } \theta_k^2(1-\mu_{\pca,k}^2)\]
where $\lambda_{\pca,k}$ is the observed eigenvalue of $\bX$, and the integrals
defining $G(\cdot)$ and $G'(\cdot)$ are computed using the above empirical
estimate for the law of $\Lambda$. In particular, this provides the estimate of
$S=\diag(\theta_1,\ldots,\theta_K)$.
Comparing \eqref{eq:Rtransformseries} and \eqref{eq:sym:kappaseries},
we then estimate
\[\tilde{\kappa}_1 \text{ by }\diag(R(1/\theta_1),\ldots,R(1/\theta_K)),
\qquad \hat{\kappa}_2 \text{ by } \diag(R'(1/\theta_1),\ldots,R'(1/\theta_K)),\]
and estimate the remaining matrices
$\tilde{\kappa}_s$ and $\hat{\kappa}_s$ using the following
recursions derived from (\ref{eq:sym:kappaseries}),
\begin{align*}
\tilde{\kappa}_s&= \sum_{j = 0}^\infty \kappa_{s+j} S^{-j} = \kappa_s I +
\sum_{j=0}^\infty \kappa_{s+1+j} S^{-j}\cdot S^{-1} = \kappa_s I +
\tilde{\kappa}_{s+1} S^{-1}
\Rightarrow \tilde{\kappa}_{s+1}=(\tilde{\kappa}_s-\kappa_s I)S,\\
\hat{\kappa}_s &= \sum_{j = 0}^\infty (j+1) \kappa_{s+j} S^{-j} =
\tilde{\kappa}_s + \sum_{j = 0}^\infty (j+1)\cdot \kappa_{s+1+j} S^{-j} \cdot S^{-1} = \tilde{\kappa}_s + \hat{\kappa}_{s+1} S^{-1}
\Rightarrow \hat{\kappa}_{s+1}=(\hat{\kappa}_s-\tilde{\kappa}_s)S.
\end{align*}

Next, explicitly differentiating (\ref{eq:posteriormean}) to compute the
matrices $\langle \partial_s \bU_t \rangle$ constituting $\bphi_t$,
and combining with the above, we obtain consistent 
estimates of the debiasing coefficients $b_{ts}$.
For the state evolution, we note that numerically evaluating the expectations
defining $(\bmu_t,\bSigma_t)$ may be prohibitive. We instead approximate the
second-moment matrices $\EE[U_sU_t^\top]$ by the empirical averages $n^{-1}
\bU_s^\top \bU_t$. We approximate $\bPhi_t$ also by the empirical average
$\bphi_t$, and combine with the above estimates of
$\kappa_s,\tilde{\kappa}_s,\hat{\kappa}_s$ according to the definition
(\ref{def: sym Sigma_T}) to obtain a consistent estimate of $\bSigma_t$.
We use the martingale identity $\EE[U_tU_*^\top]=\EE[U_tU_t^\top]$ for $t \geq 1$
to approximate $\EE[U_tU_*^\top]$ by the empirical average
$n^{-1}\bU_t^\top \bU_t$, and use this to consistently estimate $\bmu_t$.

We remark that when the AMP algorithm approaches convergence, the covariance
matrix $\bSigma_t$ becomes nearly singular, leading to potential
instabilities in evaluating the posterior mean function (\ref{eq:posteriormean})
and its Jacobian. We use a simple early stopping rule of terminating the
iterations when the smallest eigenvalue of our estimate for $\bSigma_t$ falls
below a small threshold. Alternatively, a small ridge regularization may be used
when computing $\bSigma_t^{-1}$.

\subsection{Simulation studies}\label{subsec:simulations}

We compare the estimation accuracy of three AMP algorithms
under various noise spectral distributions for $\bW$:
\begin{enumerate}
\item \textbf{Bayes-OAMP}, as described and implemented in Sections
\ref{subsec:BayesOAMP} and \ref{subsec:implementation} above.
\item \textbf{Single-iterate Bayes-OAMP}, using instead the single-iterate
posterior mean non-linearities\newline
$u_{t+1}(f_0,\ldots,f_t)=\EE[U_* \mid F_t=f_t]$, with debiasing coefficients and
state evolution parameters estimated as in Section \ref{subsec:implementation}.
\item \textbf{Gaussian Bayes-AMP}, using the debiasing coefficients and state
evolution described in \cite{montanari2021estimation} for i.i.d.\ noise. We also
estimate the signal strengths $\theta_k$ and initial spectral alignments
$\mu_{\pca,k}^2$ from (\ref{eq:sym:pcasol}) using functions $G(\cdot)$ that
correspond to the semicircle law for symmetric i.i.d.\ noise.
\end{enumerate}

\paragraph{Symmetric model.} We consider the symmetric model
(\ref{eq:sym:rankkModel}) with $n=4000$, in the settings
\begin{itemize}
\item (Semicircle) $\bW \sim \operatorname{GOE}(n)$ has i.i.d.\ $\cN(0,1/n)$
entries above the diagonal, and the limit spectral law $\Lambda$
is the semicircle distribution supported on $[-2,2]$.
\item (Uniform) $\bW=\bO^\top \bLambda \bO$ where $\bLambda$ has
i.i.d.\ $\Uniform[-\sqrt{3},\sqrt{3}]$ diagonal entries, and $\bO$ is uniformly
random.
\item (Centered Beta) $\bW=\bO^\top \bLambda \bO$ where $\bLambda$ has
i.i.d.\ $\sqrt{80/3}\cdot(\Beta(3,1) - 3/4)$ diagonal entries, and $\bO$ is uniformly random.
\end{itemize}

\begin{figure}
\centering
\includegraphics[width=0.3\textwidth]{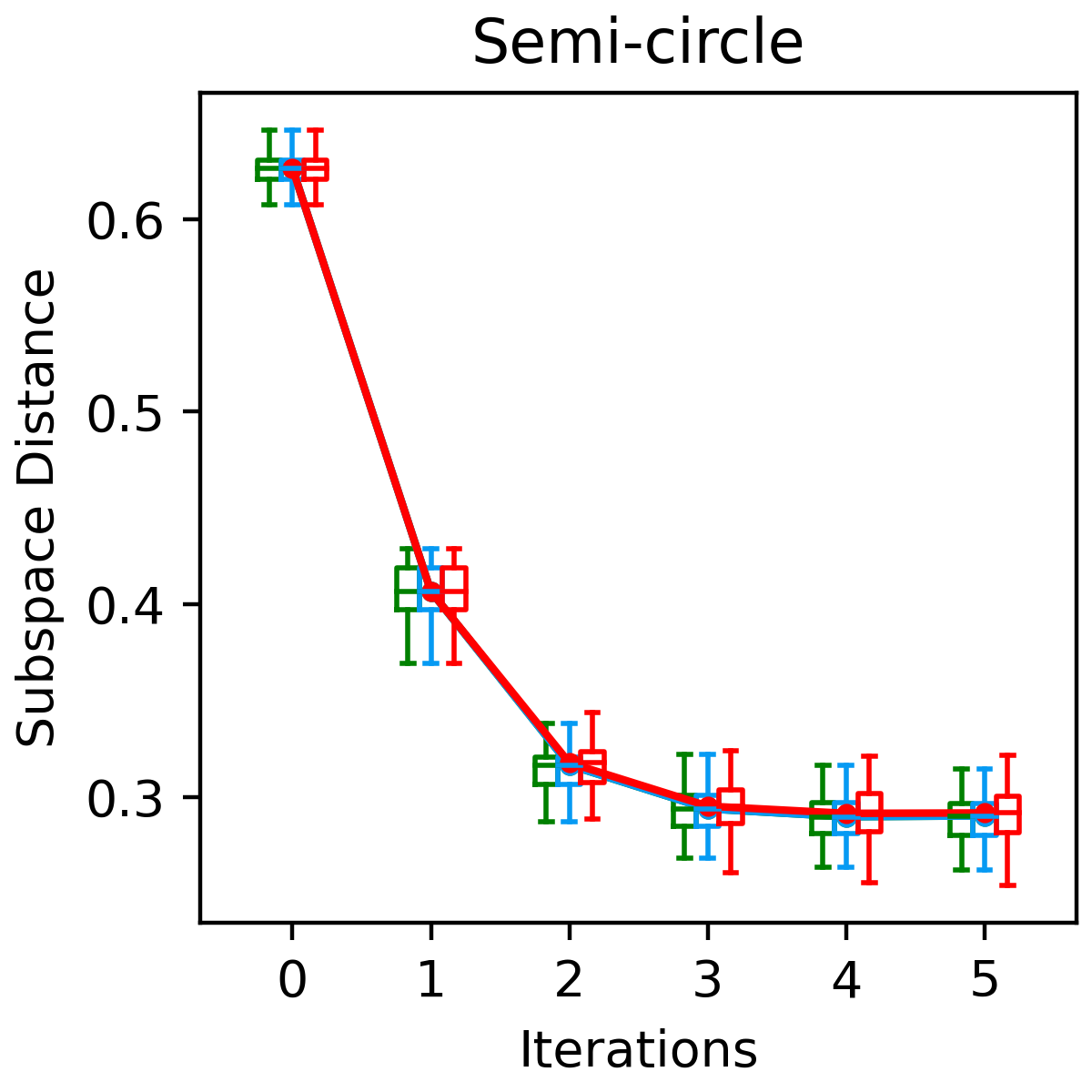}
\includegraphics[width=0.3\textwidth]{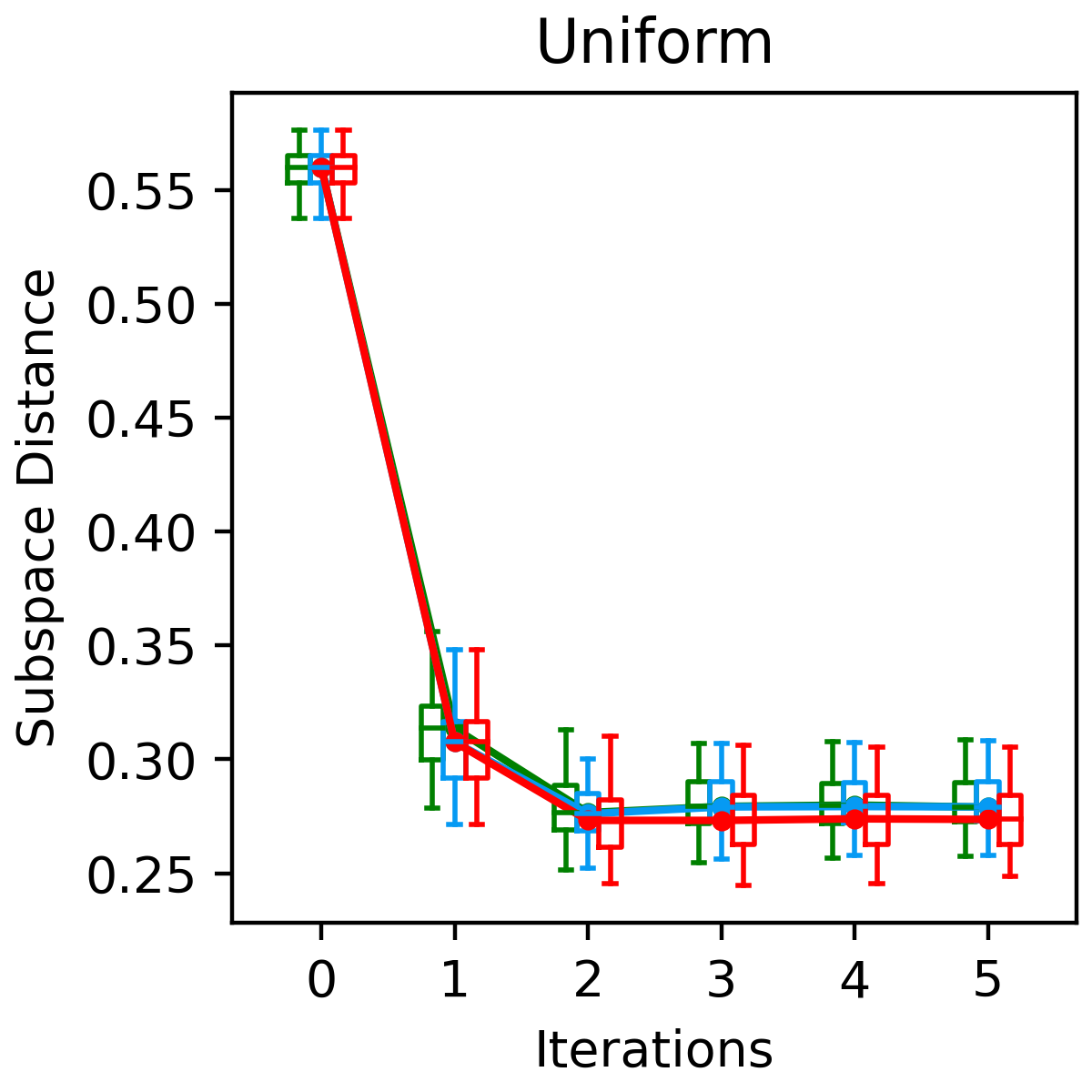}
\includegraphics[width=0.3\textwidth]{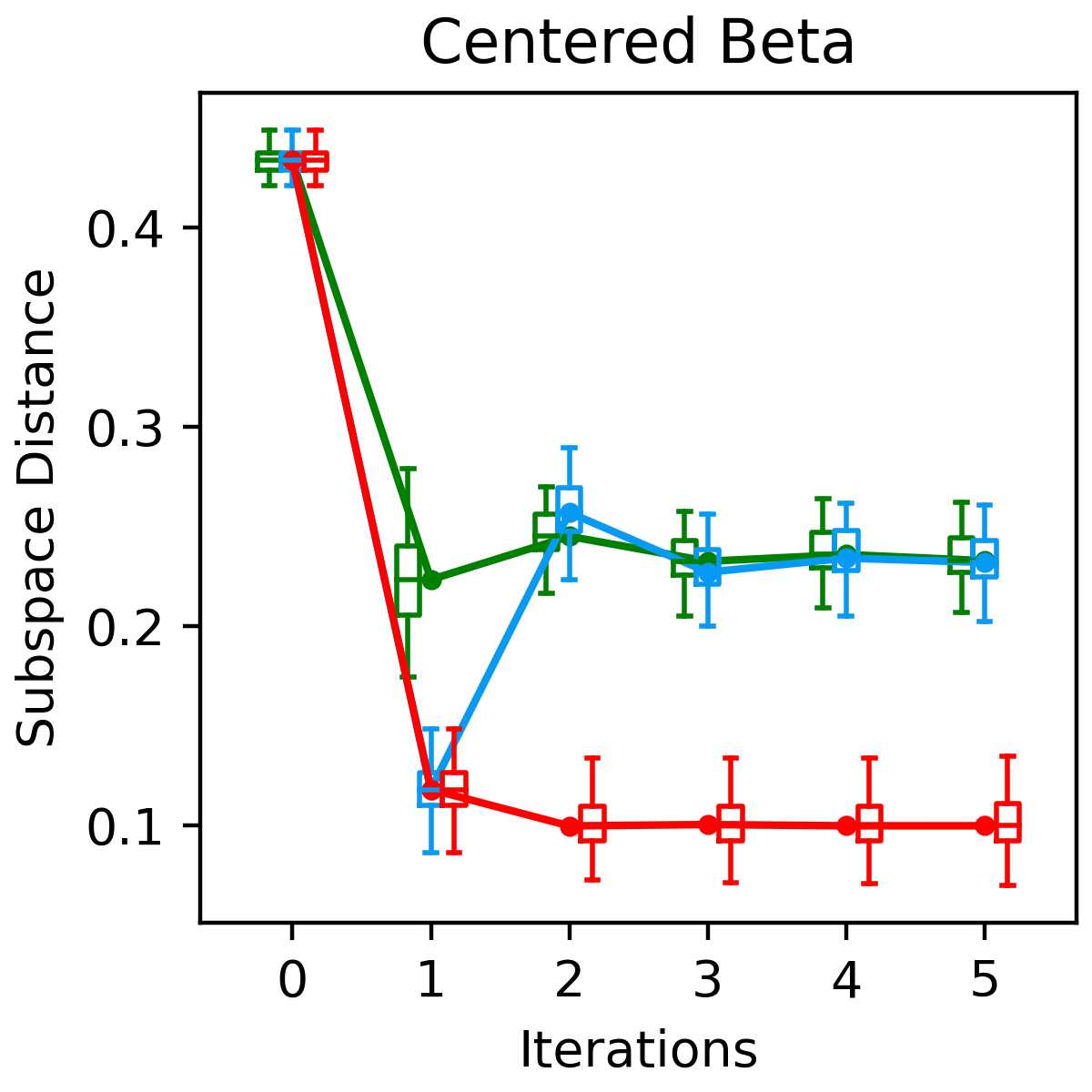}\\
\includegraphics[width = \textwidth]{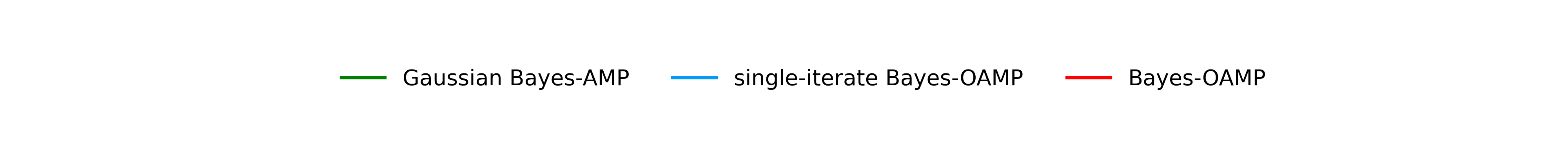}
\caption{Estimation errors for AMP iterates $\bU_t$ in the symmetric spiked
model with $n=4000$, rank-2 signal, and signal prior $U_* \sim
\frac{1}{2}\delta_{(0,1)}+\frac{1}{4}\delta_{(\sqrt{2},-1)}+\frac{1}{4}\delta_{(-\sqrt{2},-1)}$.
Boxes indicate the $\{25,50,75\}$-percentiles across 50 random trials, 
and whiskers indicate $1.5 \times \text{inter-quartile range}$.
Iteration 0 corresponds to the spectral initialization $\bU_0$.
The noise spectral distributions are (left) the semicircle law,
(middle) $\Uniform[-\sqrt{3},\sqrt{3}]$, and (right) standarized $\Beta(3,1)$.}
\label{fig:sym:algocombat}
\end{figure}

In all three settings, $\Lambda$ is normalized to have mean $\kappa_1 = 0$ and
variance $\kappa_2 = 1$.
Figure~\ref{fig:sym:algocombat} compares estimation error across
AMP iterations, for an example with a rank-2 signal where
$K'=K=2$, the two signal strengths are $(\theta_1,\theta_2)=(2,1.6)$, and 
the elements of the two signal vectors are drawn i.i.d. from the discrete 
three-point prior $U_*$ defined by
\begin{equation}\label{eq:threepointprior}
\frac{1}{2}\delta_{(0,1)}+\frac{1}{4}\delta_{(\sqrt{2},-1)}+\frac{1}{4}\delta_{(-\sqrt{2},-1)}.
\end{equation}
The error is defined by the subspace distance $\|\Pi_{\bU_*}-\Pi_{\bU_t}\|$
where $\Pi_\bU \in \RR^{n \times n}$ is the orthogonal
projector onto the two-dimensional column span of $\bU$. 
From the left panel, for GOE noise, we observe that the three algorithms yield
comparable per-iteration error, and there is negligible
degradation in accuracy from estimating the spectral free cumulants and the more
complex state parameters in Bayes-OAMP.
In the middle panel, for Uniform noise spectrum,
Gaussian Bayes-AMP remains reasonably robust to the misspecification of the
noise distribution. In the right panel, for Centered Beta noise spectrum,
we observe a significant improvement of Bayes-OAMP over the other two
approaches, which both exhibit non-monotonic error across iterations.

\begin{figure}
\centering
\begin{subfigure}[b]{\textwidth}
\caption{Gaussian Bayes-AMP}\label{fig:sim:marginalggamp}
\includegraphics[width=0.3\textwidth]{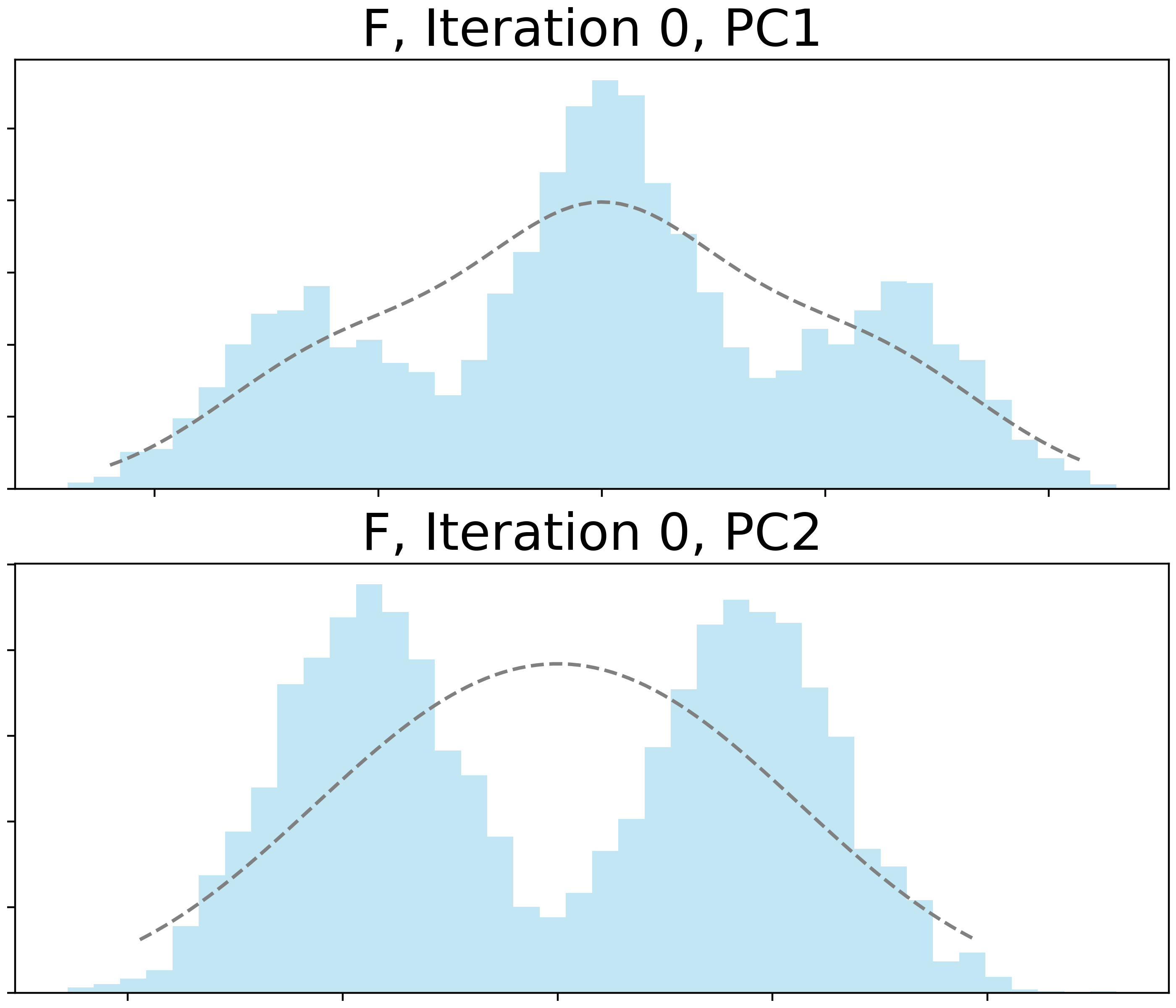}
\hfill
\includegraphics[width=0.3\textwidth]{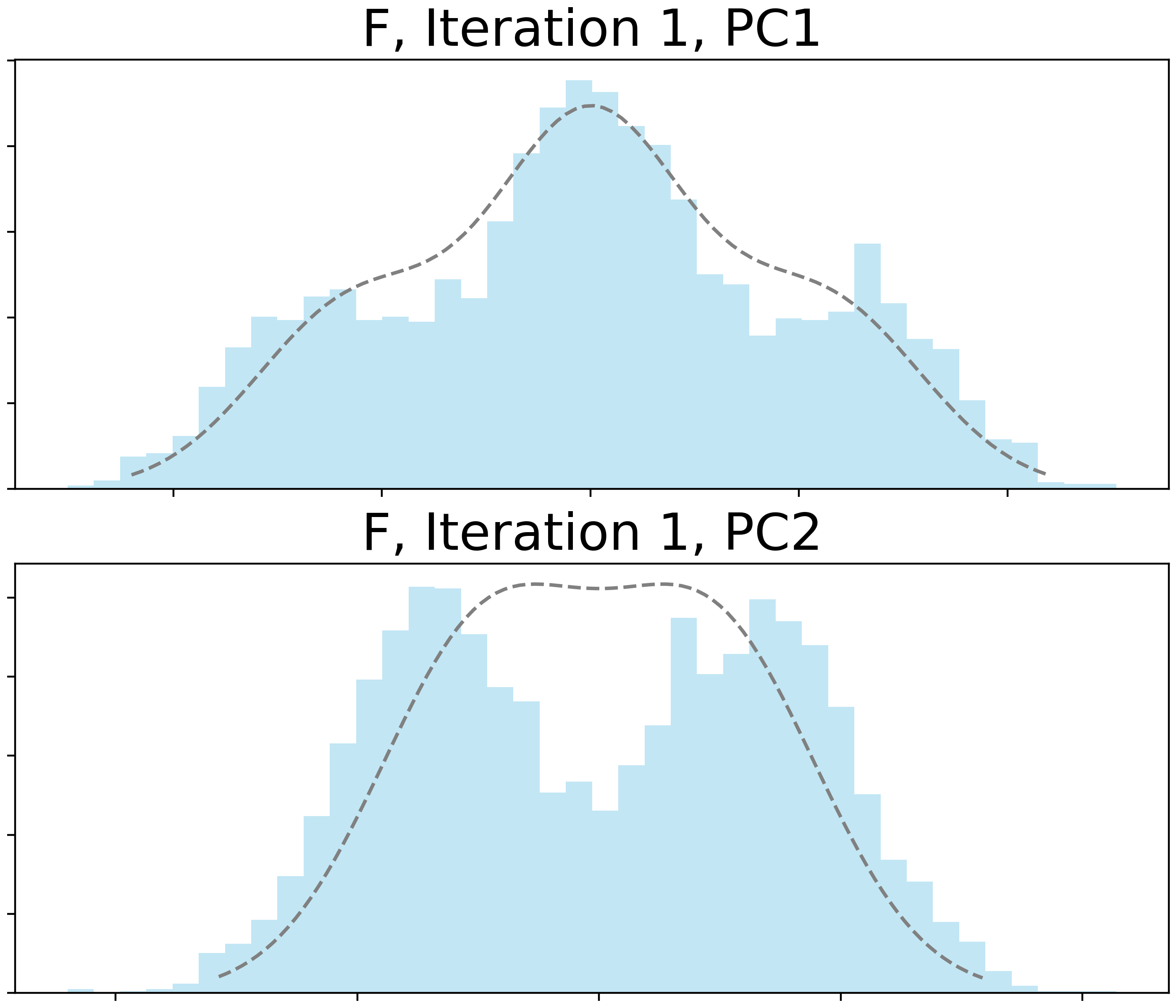}
\hfill
\includegraphics[width=0.3\textwidth]{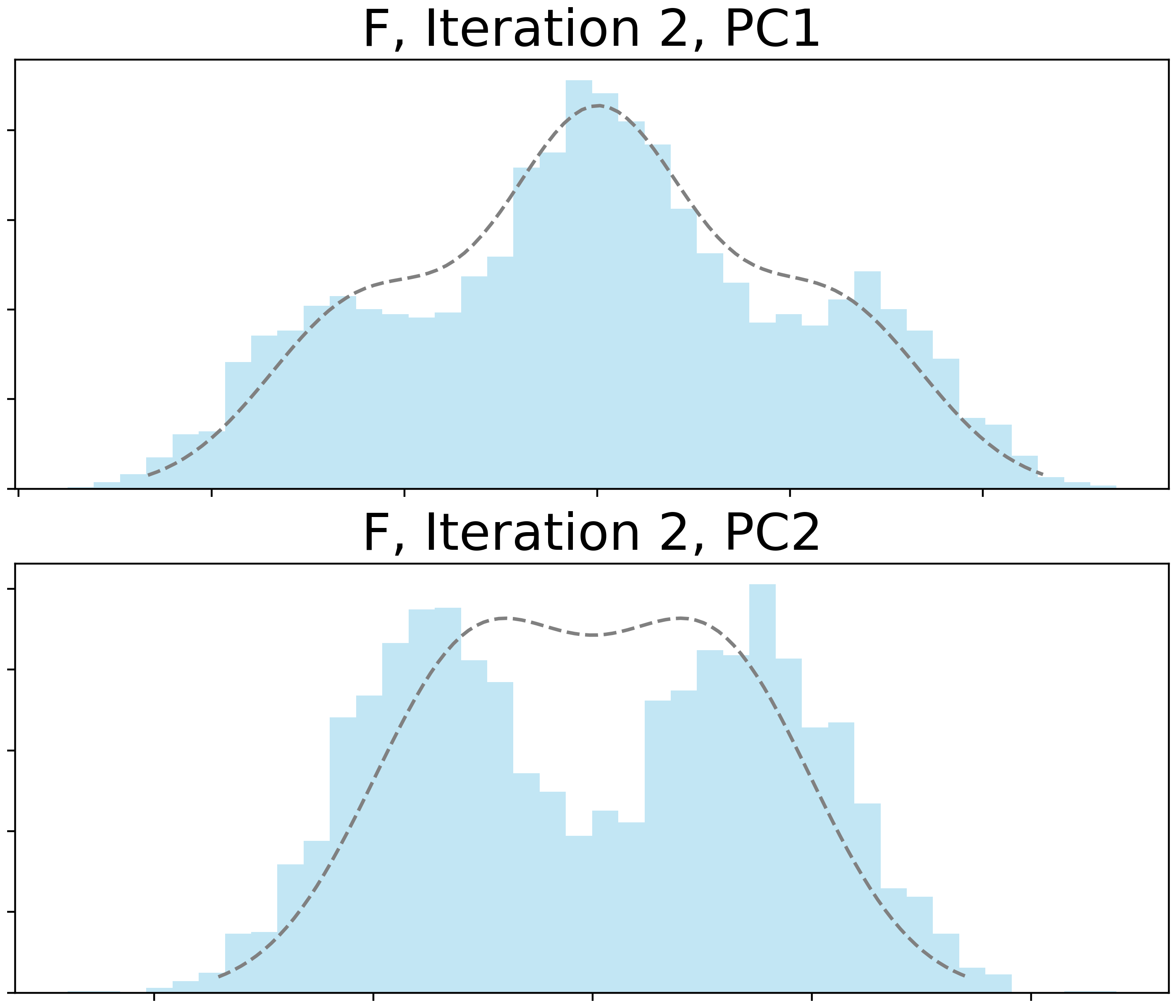}
\end{subfigure}
\\\vspace{\baselineskip}
\begin{subfigure}[b]{\textwidth}
\caption{Bayes-OAMP} \label{fig:sim:marginalrriamp}
\includegraphics[width=0.3\textwidth]{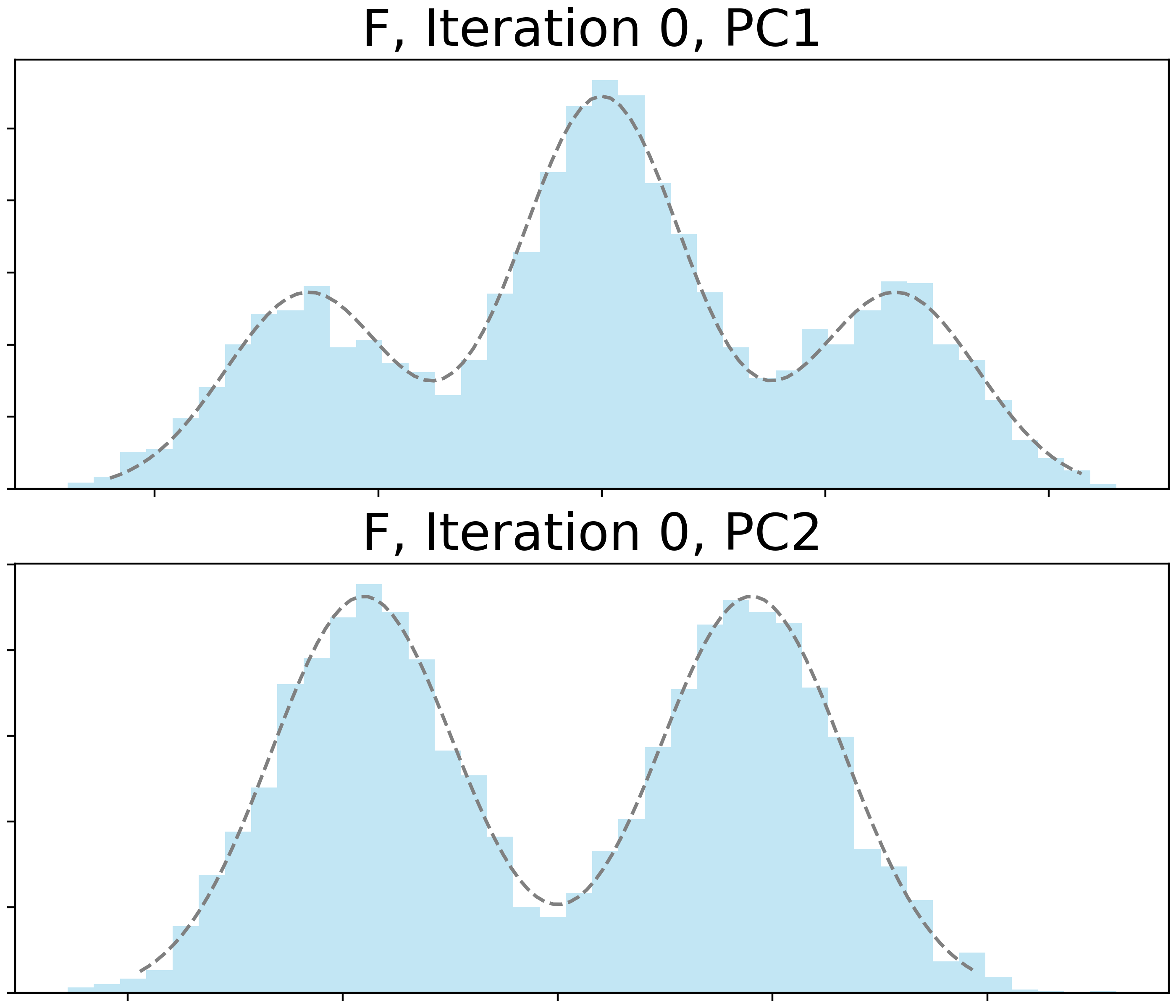}
\hfill
\includegraphics[width=0.3\textwidth]{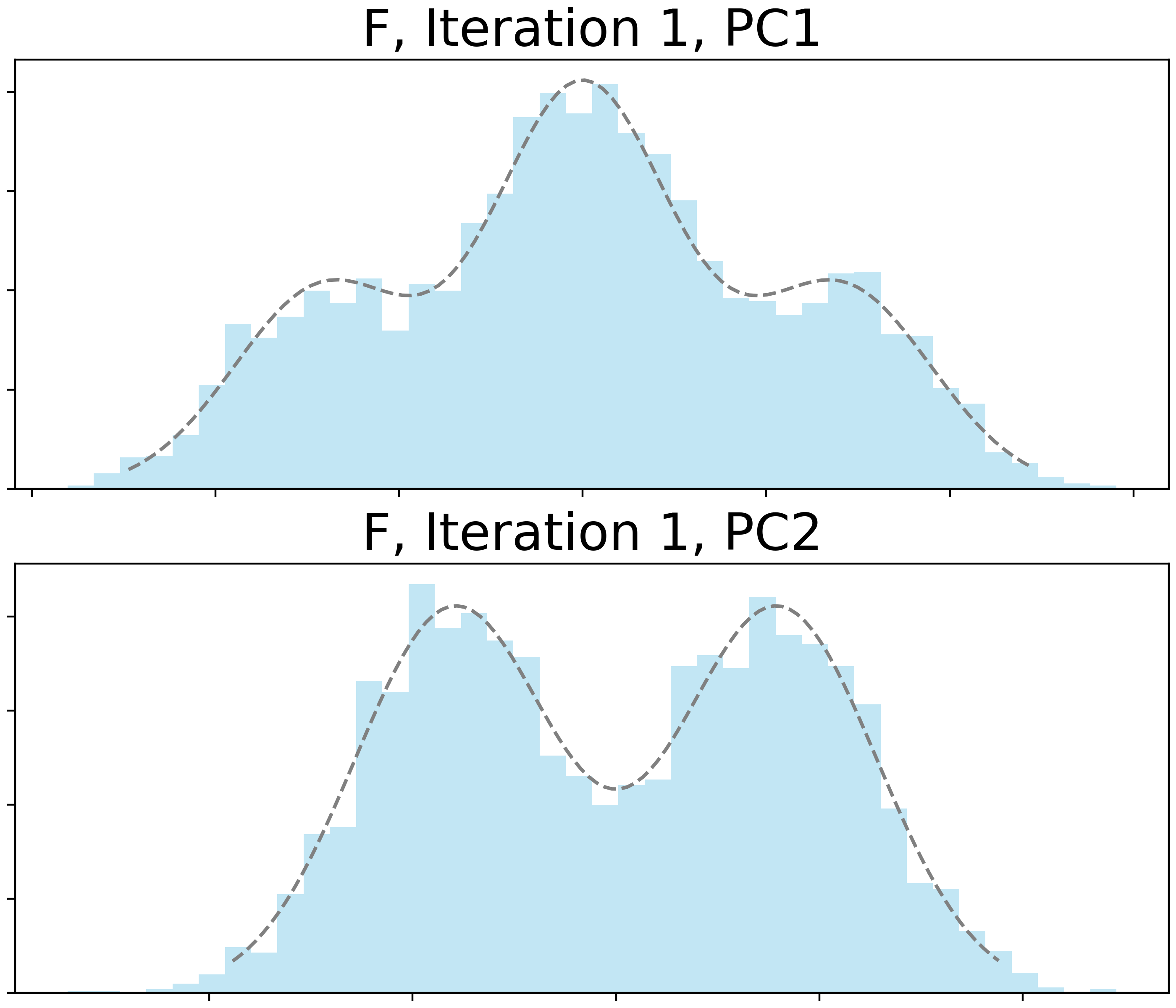}
\hfill
\includegraphics[width=0.3\textwidth]{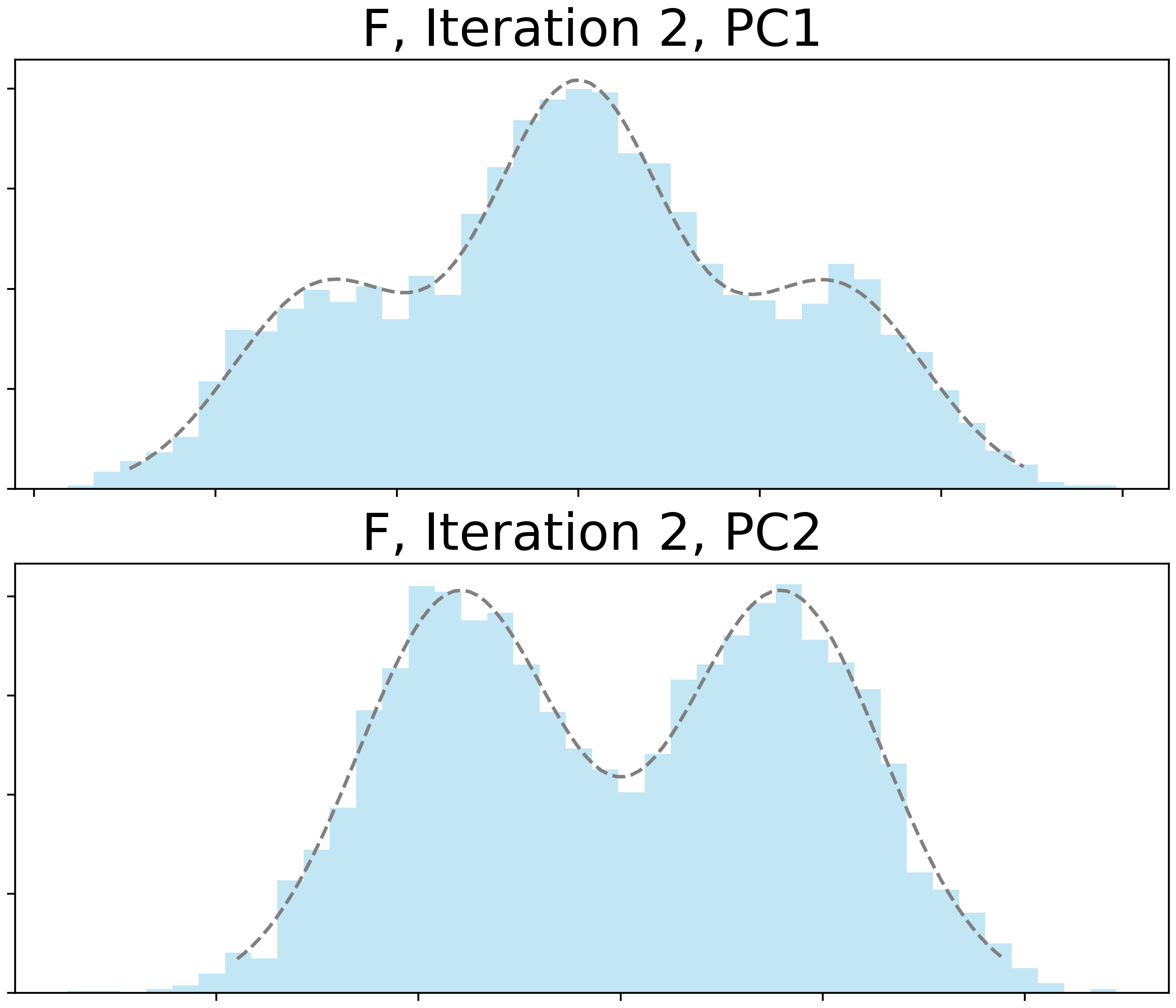}
\end{subfigure}
\caption{Distributions of iterates $\bF_0,\bF_1,\bF_2$ in the Centered Beta noise
example of Figure \ref{fig:sym:algocombat}. Shown are histograms of the empirical
distributions of the two columns of each iterate $\bF_t$ (denoted PC1 and PC2),
overlaid with the marginal density of the corresponding coordinate of the state
evolution law $\cN(\bmu_t \cdot U_*,\bSigma_t)$. This density agrees closely for
Bayes-OAMP, whereas a large discrepancy is observed for the state evolution
prediction of Gaussian Bayes-AMP. 
} 
\label{fig:sym:marginal}
\end{figure}

Figure~\ref{fig:sym:marginal} depicts the accuracy of the state evolution
predictions for the distribution of AMP iterates $\bF_t$ in the Centered Beta
example. Marginal histograms of the two columns of $\bF_t$ are
overlaid with the corresponding densities of $F_t$, computed
using state evolution parameters $(\bmu_t,\bSigma_t)$ empirically estimated as in 
Section \ref{subsec:implementation}. Note that these densities are exactly the
ones assumed for posterior-mean denoising in (\ref{eq:posteriormean}), to
obtain the next iterate $\bU_{t+1}$. We observe a close agreement for Bayes-OAMP,
and a large discrepancy with the state evolution predictions of Gaussian
Bayes-AMP.

\begin{figure}
\centering
\includegraphics[width=0.8\textwidth]{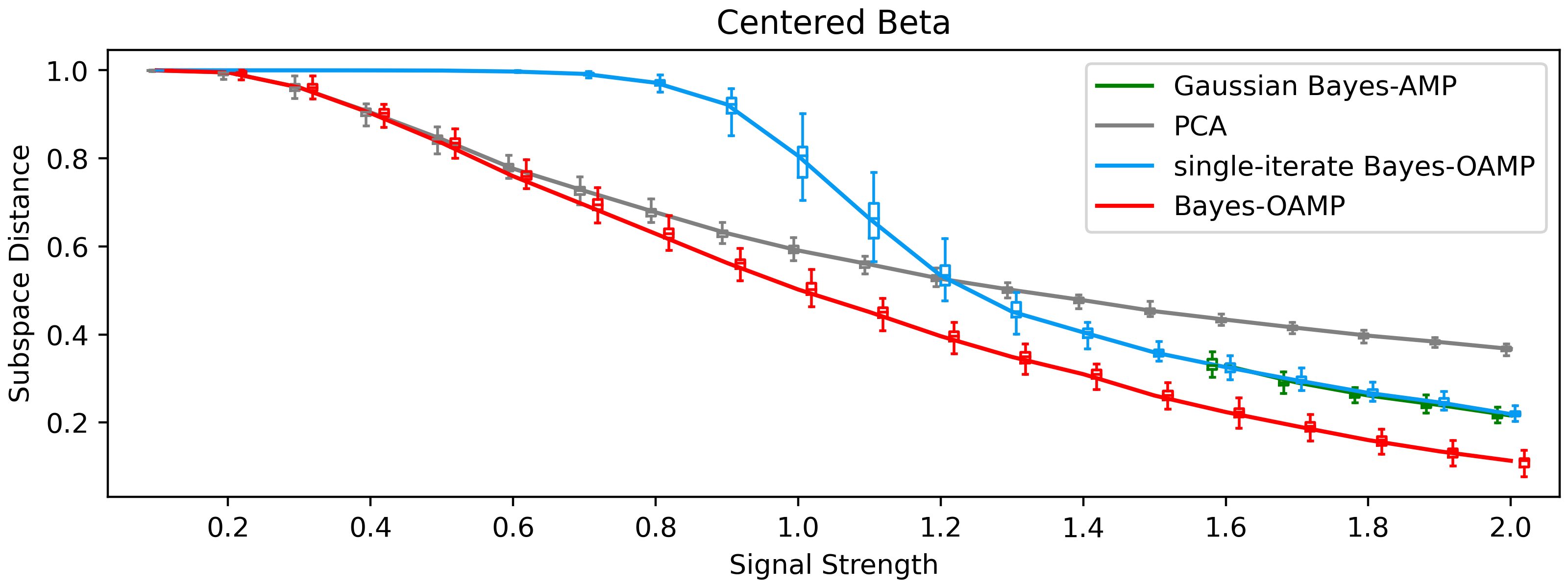}
\caption{Estimation errors for the sample eigenvector and three AMP algorithms,
in a symmetric model with Centered Beta noise spectrum, $n = 4000$, rank-1
signal with prior $U_*\sim \frac{1}{2}\delta_{-1} + \frac{1}{2}\delta_{1}$, and
signal strength $\theta$ varying from $0.1$ to $2.0$. The spectral ``phase
transition'' occurs at $\theta=0$.}
\label{fig:sym:BetaTheta}
\end{figure}

Figure~\ref{fig:sym:BetaTheta} illustrates estimation error across different
signal strengths, in an example having rank-1 signal and the same
Centered Beta noise.
Note that since the density of $\Beta(3,1)$ does not decay to 0 at its right
edge, the spectral phase transition point for this example is 0, and any
positive signal strength is super-critical.
We observe that Bayes-OAMP remains effective and improves over the spectral 
initialization for any positive signal strength.
In contrast, single-iterate Bayes-OAMP exhibits the following behavior:
below some critical signal strength, it diverges from the informative spectral
initialization to an uninformative solution. This type of behavior was
also reported in~\cite[Section 4]{mondelli2021pca} for a different prior.
Our implementation of Gaussian Bayes-AMP uses the Cauchy-transform $G(\cdot)$
based on the semicircle law to infer $\theta$, and thus is only applicable when
the largest sample eigenvalue exceeds 2, corresponding roughly to
$\theta \geq 1.6$ in this example. For these values of $\theta$, we observe its
accuracy to be close to that of single-iterate Bayes-OAMP.

\paragraph{Rectangular model.}
\begin{figure}
\centering
\includegraphics[width=0.25\textwidth]{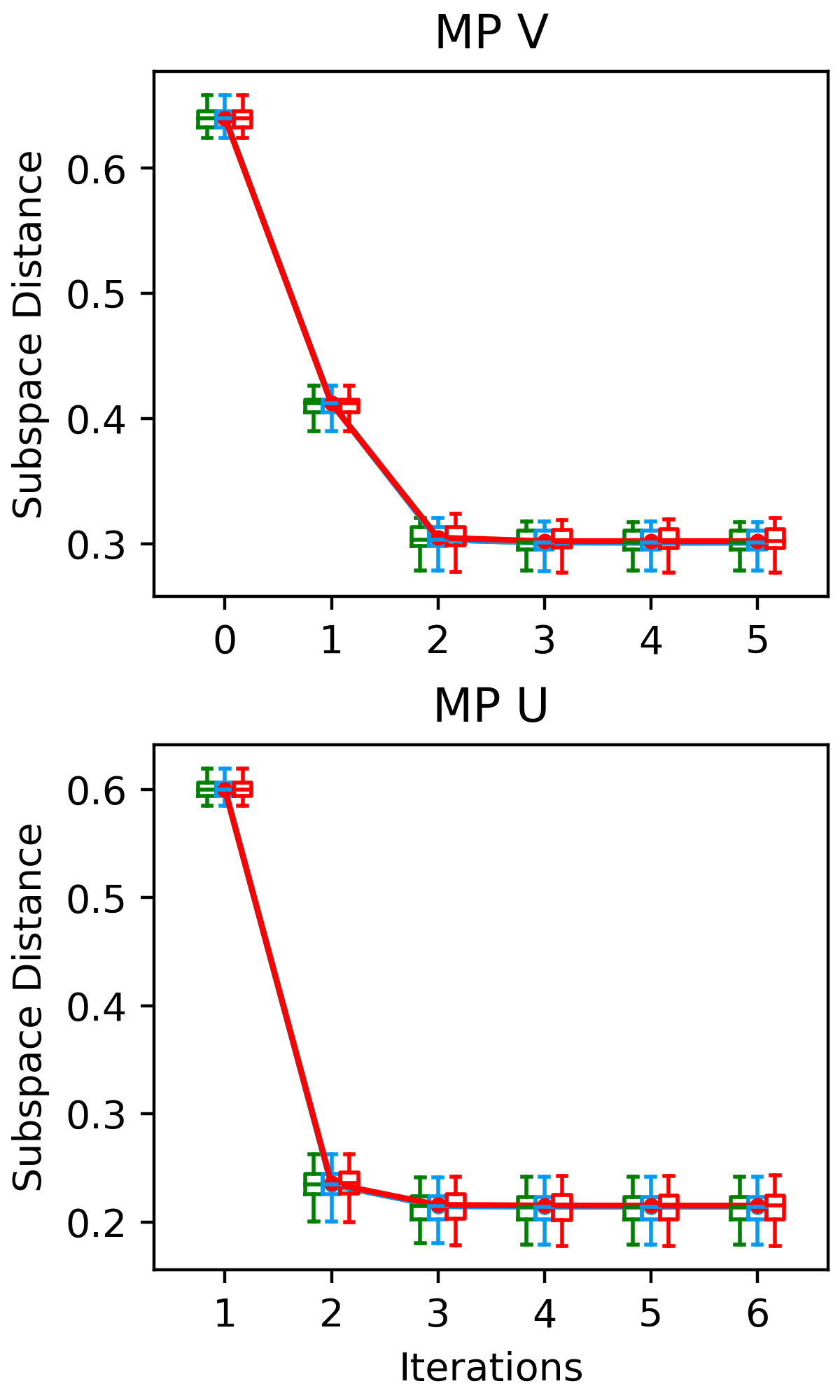}
\includegraphics[width=0.25\textwidth]{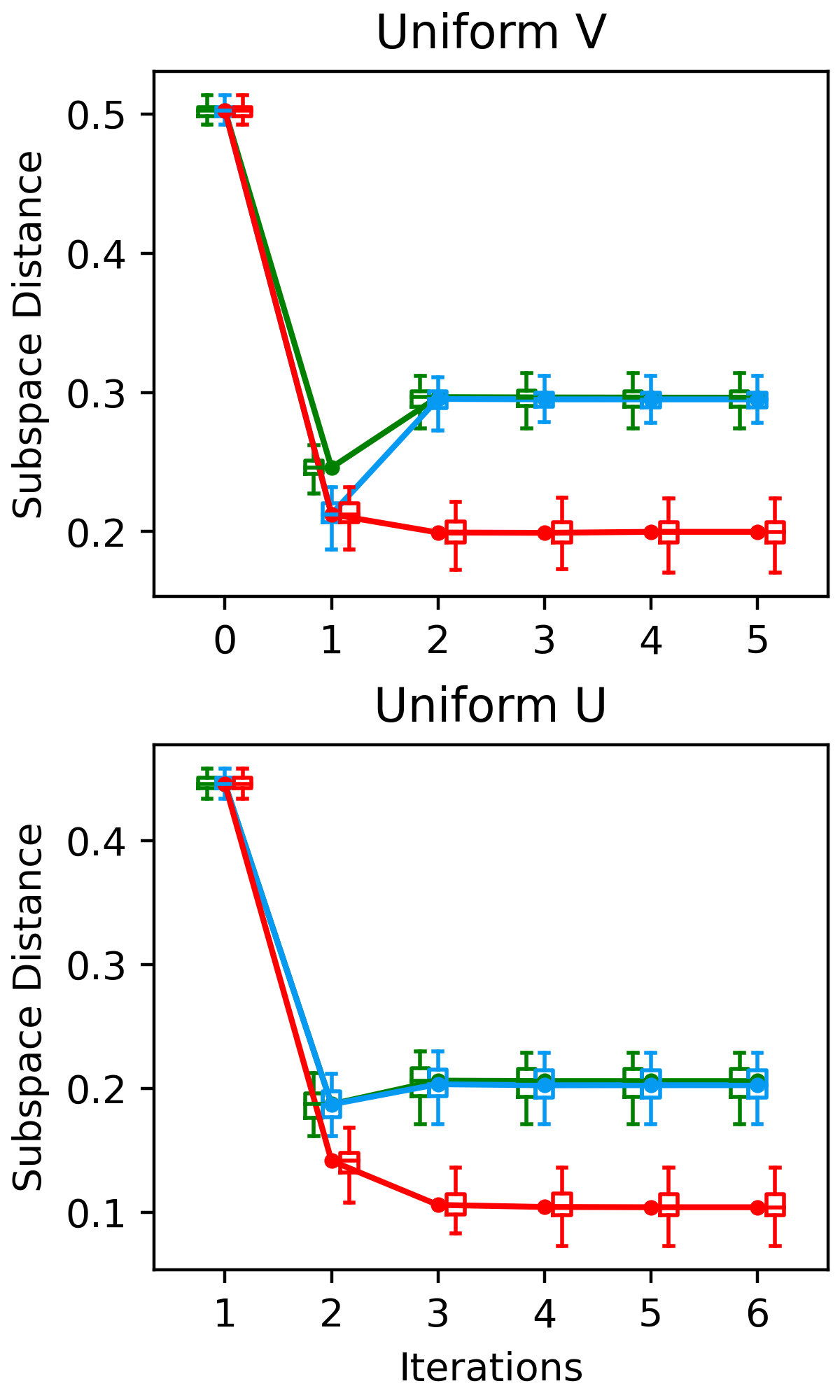}
\includegraphics[width=0.25\textwidth]{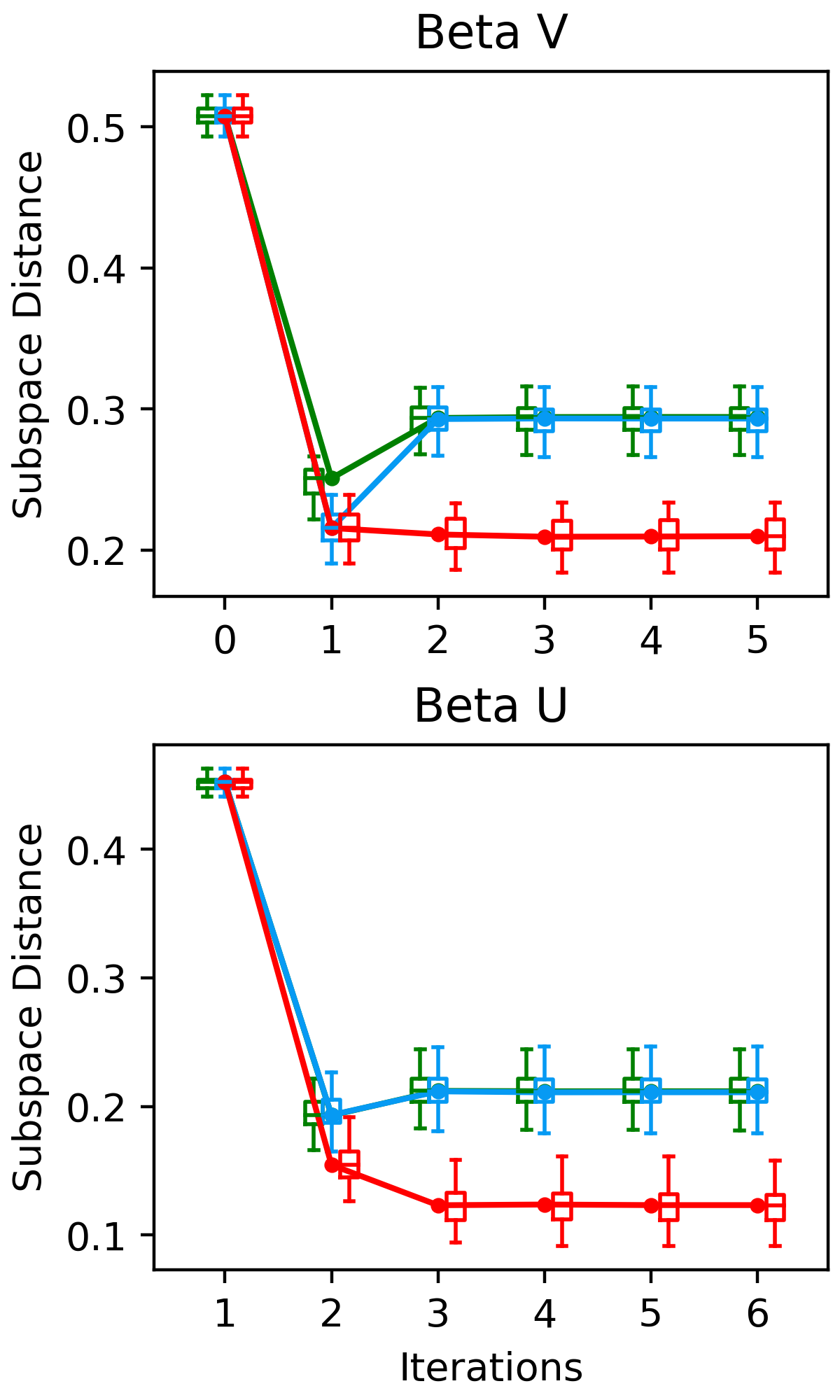}\\
\includegraphics[width = \textwidth]{figs/subspaceDist/sym/algo_label.png}
\caption{Estimation errors for AMP iterates $\bU_t$ and $\bV_t$ in the
rectangular spiked model with $(m,n)=(3000,4000)$, rank-2 signal, 
signal priors $U_*,V_* \sim
\frac{1}{2}\delta_{(0,1)}+\frac{1}{4}\delta_{(\sqrt{2},-1)}+\frac{1}{4}\delta_{(-\sqrt{2},-1)}$,
and signal strengths $(\theta_1,\theta_2)=(2,1.5)$. 
Iterates $\bV_0$ and $\bU_1$ correspond to the
spectral initializations. The noise spectral distributions are (left)
square-root of the Marcenko-Pastur law, (middle)
$\Uniform[\sqrt{3/7},2\sqrt{3/7}]$, and (right) the $\sqrt{5/3}\cdot\Beta(3,1)$ distribution.}
\label{fig:rec:algocombat}
\end{figure}

We consider the rectangular model\[
\bX=\sum_{k=1}^{K'} \frac{\theta_k}{\sqrt{mn}}\bu_*^k {\vb_*^k}^\top + \bW \in
\bbR^{m\times n}
\]with
$(m,n)=(3000,4000)$, again with $K'=K=2$,
signal strengths $(\theta_1,\theta_2)=(2,1.5)$, and the discrete three-point
prior (\ref{eq:threepointprior}) for both $U_*$ and $V_*$. We consider the
settings
\begin{itemize}
\item (Marcenko-Pastur) $\bW$ has i.i.d.\ $\cN(0,1/n)$ entries, and the limit
singular value distribution $\Lambda$ is the square-root of a Marcenko-Pastur law.
\item (Uniform) $\bW=\bO^\top \bLambda \bQ$ where $\bLambda$ has
i.i.d.\ $\Uniform[\sqrt{3/7},2\sqrt{3/7}]$ diagonal entries, and $(\bO,\bQ)$ are
uniformly random.
\item (Beta) $\bW=\bO^\top \bLambda \bQ$ where $\bLambda$ has i.i.d.\ 
$\sqrt{5/3}\cdot\Beta(3,1)$ diagonal entries, and $(\bO,\bQ)$ are uniformly random.
\end{itemize}
In all three settings, the limiting singular value distribution $\Lambda$ of $\bW$ is normalized so that
$\EE[\Lambda^2]=1$. 

The exact form of Bayes-OAMP algorithm for the rectangular model can be found in 
Appendix~\ref{sec:numerical_rec}. 
Figure \ref{fig:rec:algocombat} compares
per-iteration estimation errors, where we again observe that Bayes-OAMP achieves the
same error as Gaussian Bayes-AMP and single-iterate Bayes-OAMP for i.i.d.\
noise, but improves over the other procedures for the remaining noise settings.

\section{Conclusion}

We have developed AMP algorithms for both symmetric and rectangular spiked random 
matrix models in the context of orthogonally invariant noise. These algorithms
extend those of \cite{fan2020approximate} and \cite{mondelli2021pca} by allowing
for matrix-valued AMP iterates, multivariate non-linearities, and spectral
initializations using super-critical sample eigenvectors or singular vectors of the
observed data. We have derived the forms of the Onsager debiasing coefficients
for a general class of such algorithms, and established rigorous Gaussian state
evolutions for their iterates. These depend on the spectral distributions of
the noise via their (symmetric or rectangular) free cumulants.

We developed one application of such algorithms to estimate the
super-critical signal vectors in these models, by choosing Bayes posterior mean
denoising functions as the non-linearities in AMP. This Bayes-OAMP approach is
similar to the algorithms of \cite{rangan2012iterative,montanari2021estimation},
but applies the posterior mean conditional on all preceding AMP iterates,
instead of only the single preceding iterate. 
Subsequent work of
\cite{barbier2023fundamental} suggests that alternating this posterior mean
with identity nonlinearities may lead to further improvements of the Bayes-OAMP
method, and potentially attain the Bayes-optimal estimation error.
Computation of the
posterior mean is enabled by the above state evolution, whose parameters may be
estimated empirically. We observe in simulation that Bayes-OAMP can 
yield improved estimation accuracy over both the spectral initialization and
single-iterate AMP algorithms, and can be more robust than standard AMP
algorithms that are designed for white noise. These simulations suggest that
the method may be sufficiently accurate and stable for practical use in PCA
applications.

\clearpage
\appendix

\section{Rectangular AMP with independent initialization}\label{sec:indrect}
In this section, we present the results for AMP with independent initialization
in the rectangular case.
We will then discuss the rectangular signal-plus-noise models and the 
corresponding spectrally initialized AMP algorithm in
Appendix~\ref{sec: main rec}.

Let $\bW \in \RR^{m \times n}$ be a rectangular matrix, with singular
value decomposition
$\bW=\bO^\top \bLambda \bQ$. To simplify the exposition, let us assume that
\[\gamma \defeq m/n \leq 1.\]
For $m/n>1$, our results may be applied to $\bW^\top$.
We denote $\bLambda=\diag(\blambda)$ where
$\blambda \in \RR^m$ are the singular values of $\bW$. We assume that $\bO$
and $\bQ$ are Haar-distributed orthogonal bases of singular vectors, so
$\bW$ is bi-rotationally invariant in law.

For fixed $J,L \geq 0$ and $K \geq 1$, consider matrices of side
information $\bE \in \RR^{m \times J}$ and $\bF \in \RR^{n \times L}$, and
sequences of Lipschitz functions $u_2,u_3,\ldots$ and
$v_1,v_2,\ldots$. We study an AMP algorithm with
initialization $\bU_1 \in \RR^{m \times K}$ independent of $\bW$, having the
iterates (of dimensions $m \times K$ and $n \times K$)
\begin{align}
\bZ_t&=\bW^\top \bU_t-\bV_1b_{t1}^\top-\ldots-\bV_{t-1}b_{t,t-1}^\top\label{eq:AMPrectz}\\
\bV_t&=v_t(\bZ_1,\ldots,\bZ_t,\bF)\label{eq:AMPrectv} \\
\bY_t&=\bW\bV_t-\bU_1a_{t1}^\top-\ldots-\bU_ta_{tt}^\top\label{eq:AMPrecty}\\
\bU_{t+1}&=u_{t+1}(\bY_1,\ldots,\bY_t,\bE).\label{eq:AMPrectu}
\end{align}
Here again, $b_{ts},a_{ts} \in \RR^{K \times K}$ are Onsager debiasing coefficient
matrices, and $v_t(\cdot)$ and $u_{t+1}(\cdot)$ are applied row-wise.

\paragraph{Debiasing coefficients.} 
Let $\Lambda \in \RR$ be a non-negative random variable with compact support, which will be the limit singular value  distribution of $\Wb$.
Let $\{\kappa_{2j}\}_{j \geq 1}$ be the rectangular free cumulants of
$\Lambda$ with aspect ratio $\gamma=m/n$---see \cite[Section
2.4]{fan2020approximate} for definitions. For $T \geq 1$,
define analogously to (\ref{eq:phi}) the block-lower-triangular $TK \times TK$ matrices
\[\bphi_T=\Big(\langle \partial_s \bU_r \rangle\Big)_{r,s \in \{1,\ldots,T\}}
\qquad
\bpsi_T=\Big(\langle \partial_s \bV_r \rangle\Big)_{r,s \in \{1,\ldots,T\}}\]
with rows blocks indexed by $r$ and column blocks by $s$. Here,
denoting $y_{t,i},z_{t,i},e_i,f_i$ as the $i^\text{th}$
rows of $\bY_t,\bZ_t,\bE,\bF$, we set
\[\langle \partial_s \bU_r \rangle=\frac{1}{m}\sum_{i=1}^m \partial_s
u_r(y_{1,i},\ldots,y_{r-1,i},e_i),
\qquad \langle \partial_s \bV_r \rangle=\frac{1}{n}\sum_{i=1}^n 
\partial_s v_r(z_{1,i},\ldots,z_{r,i},f_i)\]
and use the conventions $\partial_s u_r=0$ for $s \geq r$ and $\partial_s v_r=0$
for $s>r$. Then the matrices $a_{ts}$ and $b_{ts}$ in (\ref{eq:AMPrectz}) and
(\ref{eq:AMPrecty}) up to iteration $T$ are the blocks of
\begin{align*}
\ba_T&=\sum_{j=0}^\infty \kappa_{2(j+1)} \bpsi_T
(\bphi_T\bpsi_T)^j \defeq
{\footnotesize\begin{pmatrix} a_{11} &  &  & 
\\ a_{21} & a_{22} &  & \\ \vdots & \vdots & \ddots &  \\ a_{T1} & a_{T2} &
\cdots & a_{TT} \end{pmatrix}},\\
\bb_T&=\gamma\sum_{j=0}^\infty \kappa_{2(j+1)}\bphi_T (\bpsi_T\bphi_T)^j
\defeq {\footnotesize\begin{pmatrix} 0 & & & & \\
b_{21} & 0 & & & \\
b_{31} & b_{32} & 0 & & \\
\vdots & \vdots & \ddots & \ddots & \\
b_{T1} & b_{T2} & \cdots & b_{T,T-1} & 0 \end{pmatrix}}.
\end{align*}

\paragraph{State Evolution.} The state of this algorithm up to iteration $T$
is characterized by two covariance matrices
$\bSigma_T,\bOmega_T \in \RR^{TK \times TK}$ describing the limit multivariate
Gaussian laws of the rows of $(\bY_1,\ldots,\bY_T)$ and $(\bZ_1,\ldots,\bZ_T)$.
These are defined recursively as follows.

Let $(U_1,E) \in \RR^{K+J}$ and $F \in \RR^L$ be
random vectors, which will be the limit empirical distribution of rows of
$(\bU_1,\bE)$ and $\bF$. Inductively for $t=1,2,3,\ldots$
having defined joint laws of $(U_1,\ldots,U_t,Y_1,\ldots,Y_{t-1},E)$ and
$(V_1,\ldots,V_{t-1},Z_1,\ldots,Z_{t-1},F)$, we define the $tK \times tK$ 
block matrices analogous to (\ref{eq:Delta}) and (\ref{eq:Phi}),
\[\bDelta_t=\Big(\EE[U_rU_s^\top]\Big)_{r,s \in \{1,\ldots,t\}},
\quad \bPhi_t=\Big(\EE[\partial_s
u_r(Y_1,\ldots,Y_{r-1},E)]\Big)_{r,s \in \{1,\ldots,t\}},\]
\[\bGamma_t=\Big(\EE[V_rV_s^\top]\Big)_{r,s \in \{1,\ldots,t\}},
\quad \bPsi_t=\Big(\EE[\partial_s v_r(Z_1,\ldots,Z_r,E)]\Big)_{r,s \in
\{1,\ldots,t\}},\]
with row blocks indexed by $r$ and column blocks by $s$. We then define
the covariance
\begin{align*}
\bOmega_t = \gamma \sum_{j=0}^\infty \bTheta^{(j)} 
[ \bPhi_t,\bPsi_t, \kappa_{2(j+1)}\bDelta_t, \kappa_{2(j+1)}\bGamma_t ]
\end{align*}
where we denote
\begin{align}
\label{eq:rec:Thetaj}
\bTheta^{(j)}[\bPhi,\bPsi,\kappa \bDelta,\kappa \bGamma]=\sum_{i=0}^j
(\bPhi\bPsi)^i (\kappa \bDelta) (\bPsi^\top\bPhi^\top)^{j-i}
+\sum_{i=0}^{j-1} (\bPhi\bPsi)^i\bPhi (\kappa \bGamma) \bPhi^\top
(\bPsi^\top\bPhi^\top)^{j-1-i}.
\end{align}
Note that the final $t^\text{th}$ row block of $\bPsi_t$ and $t^\text{th}$ row
and column blocks of $\bGamma_t$ are not
yet well-defined, as we have not yet defined $(V_t,Z_t)$. We permit this
ambiguity because the matrix products defining $\bTheta^{(j)}$ in
\eqref{eq:rec:Thetaj} do not depend on these blocks,
as the last column block of $\bPhi_t$ is 0. From
$\bOmega_t$, we define the joint law of $(V_1,\ldots,V_t,Z_1,\ldots,Z_t,F)$ by
\begin{equation}\label{eq:SErectZ}
(Z_1,\ldots,Z_t) \sim \cN(0,\bOmega_t) \independent F,
\quad V_s=v_s(Z_1,\ldots,Z_s,F) \text{ for each } s=1,\ldots,t.
\end{equation}

Now, having defined $(U_1,\ldots,U_t,Y_1,\ldots,Y_{t-1},E)$ and
$(V_1,\ldots,V_t,Z_t,\ldots,Z_t,F)$, all blocks of $\bGamma_t$ and $\bPsi_t$
are well-defined, and we define the covariance
\begin{equation*}
\bSigma_t=\sum_{j=0}^\infty 
\bXi^{(j)}[\bPhi_t,\bPsi_t, \kappa_{2(j+1)}\bDelta_t,\kappa_{2(j+1)} \bGamma_t]
\end{equation*}
where
\begin{align}
\label{eq:rec:Xij}
\bXi^{(j)}[\bPhi,\bPsi,\kappa\bDelta,\kappa\bGamma] = \sum_{i=0}^j
(\bPsi\bPhi)^i(\kappa\bGamma)(\bPhi^\top\bPsi^\top)^{j-i} + \sum_{i=0}^{j-1}
(\bPsi\bPhi)^i\bPsi(\kappa\bDelta)\bPsi^\top(\bPhi^\top\bPsi^\top)^{j-1-i}
\end{align}
From $\bSigma_t$, we then define the joint law of
$(U_1,\ldots,U_{t+1},Y_1,\ldots,Y_t,E)$ by
\begin{equation}\label{eq:SErectY}
(Y_1,\ldots,Y_t) \sim \cN(0,\bSigma_t) \independent (U_1,E), \quad
U_{s+1}=u_{s+1}(Y_1,\ldots,Y_s,E) \text{ for each } s=1,\ldots,t,
\end{equation}
completing these inductive definitions. Under these definitions, it may be
checked that the upper-left $(t-1) \times (t-1)$ blocks of $\bSigma_t$ and
$\bOmega_t$ coincide with $\bSigma_{t-1}$ and $\bOmega_{t-1}$.

This state evolution characterizes the iterates of the AMP algorithm
(\ref{eq:AMPrectz}--\ref{eq:AMPrectu}) under the following assumptions.

\begin{assumption}\label{assump:rectW}
The ratio $\gamma=m/n \leq 1$ is fixed as $m,n \to \infty$.
The matrix $\bW=\bO^\top \diag(\blambda)\bQ$ and random
variable $\Lambda$ satisfy
\begin{enumerate}[label=(\alph*)]
\item $\bO$ and $\bQ$ are random, independent, and Haar-distributed over the
orthogonal groups.
\item $\blambda$ is independent of $\bO,\bQ$, with empirical distribution
converging weakly a.s.\ to $\Lambda$ as $m,n \to \infty$.
\item $\Lambda$ has compact support. Denoting $\lambda_+=\max
\operatorname{supp}(\Lambda)$, $\max(\blambda) \to \lambda_+$ a.s.\ as 
$m,n \to \infty$.
\end{enumerate}
\end{assumption}
\begin{assumption}\label{assump:rectindinit}
The AMP initialization $\bU_1$, functions $u_2,u_3,\ldots$ and
$v_1,v_2,\ldots$, and random vectors $(U_1,E)$ and $F$ satisfy
\begin{enumerate}[label=(\alph*)]
\item $\bU_1,\bE,\bF$ are independent of $\bO,\bQ$, with
$(\bU_1,\bE) \toWtwo (U_1,E)$ and $\bF \toWtwo F$ as $m,n \to \infty$.
\item Each $u_{t+1}(\cdot)$ and $v_t(\cdot)$ is Lipschitz in all
arguments. For each $s=1,\ldots,t$, $\partial_s u_{t+1}(Y_1,\ldots,Y_t,E)$
and $\partial_s v_t(Z_1,\ldots,Z_t,F)$ exist and are continuous on sets of
probability 1 under the laws of $(Y_1,\ldots,Y_t,E)$ and $(Z_1,\ldots,Z_t,F)$
defined by (\ref{eq:SErectY}) and (\ref{eq:SErectZ}).
\end{enumerate}
\end{assumption}

\begin{theorem}\label{thm:rectindinit}
Suppose Assumptions \ref{assump:rectW} and \ref{assump:rectindinit} hold.
Fix any $T \geq 1$, consider the AMP algorithm
(\ref{eq:AMPrectz}--\ref{eq:AMPrectu}) up to iteration $T$,
and define $(U_1,\ldots,U_{T+1},Y_1,\ldots,Y_T,E)$ and
$(V_1,\ldots,V_T,Z_1,\ldots,Z_T,F)$ by (\ref{eq:SErectY}) and
(\ref{eq:SErectZ}). Then almost surely as $m,n \to \infty$,
\begin{align*}
(\bU_1,\ldots,\bU_{T+1},\bY_1,\ldots,\bY_T,\bE) &\toWtwo 
(U_1,\ldots,U_{T+1},Y_1,\ldots,Y_T,E), \\
(\bV_1,\ldots,\bV_T,\bZ_1,\ldots,\bZ_T,\bF) &\toWtwo 
(V_1,\ldots,V_T,Z_1,\ldots,Z_T,F).
\end{align*}
\end{theorem}

Theorem \ref{thm:rectindinit} extends
\cite[Theorem 5.3 and Corollary 5.4]{fan2020approximate} in ways that are
analogous to the extensions provided by the
preceding Theorem \ref{thm:symindinit} in the square symmetric
setting. The proof of Theorem \ref{thm:rectindinit} modifies the proofs of
\cite[Theorem 5.3 and Corollary 5.4]{fan2020approximate} using the same ideas as
described in Appendix \ref{appendix:indInitproof} to prove
Theorem \ref{thm:symindinit}, and we omit this proof for brevity.

\begin{remark}\label{remark:indinitestimatekappa_rec}
Similar to the symmetric case in Remark \ref{remark:indinitestimatekappa_sym},
we have defined $b_{ts},a_{ts}$ in (\ref{eq:AMPrectz}) and (\ref{eq:AMPrecty})
using the rectangular free cumulants $\{\kappa_j\}$ of the limit singular value
distribution. 
Theorem~\ref{thm:rectindinit} then also holds for any AMP algorithm where 
$b_{ts},a_{ts}$ are replaced by $b_{ts}',a_{ts}'$ such that 
$\|b_{ts}-b_{ts}'\| \to 0$ and
$\|a_{ts}-a_{ts}'\| \to 0$ a.s.\ as $m,n \to \infty$. In particular, they hold
if $b_{ts},a_{ts}$ are instead defined with $\{\kappa_j\}$ being any consistent
estimates of these limit free cumulants.
\end{remark}

\section{Spectral initialization for the rectangular spiked model}\label{sec: main rec}

In the rectangular setting, we consider analogously a rank-$K'$ spiked
signal-pluse-noise model
\begin{align}\label{eq:rec:rankkmodel}
\bX=\sum_{k=1}^{K'} \frac{\theta_k}{\sqrt{mn}}\bu_*^k {\vb_*^k}^\top + \bW \in
\bbR^{m\times n}.
\end{align}
Here $\bu_*^1,\ldots,\bu_*^{K'}$ and $\bv_*^1,\ldots,\bv_*^{K'}$ are $K'$ pairs
of left and right signal singular vectors, normalized so that
\begin{equation}\label{eq:uvnormalization}
\|\bu_*^k\|^2=m, \quad {\bu_*^j}^\top \bu_*^k=0,
\quad \|\bv_*^k\|^2=n, \quad {\bv_*^j}^\top \bv_*^k=0
\quad \text{ for all } j \neq k \in \{1,\ldots,K'\}.
\end{equation}
We again order the signal singular values $\theta_1,\ldots,\theta_{K'}$ (not
necessarily in sorted order) so that the first $K$ will correspond to the
spectral initialization. We will assume
$\bW=\bO^\top \bLambda \bQ$ is bi-rotationally invariant in law.

We denote by
\begin{align}\label{rec sigval}
  \lambda_1(\bX),\ldots,\lambda_{K'}(\bX)  
\end{align}
the largest $K'$ sample singular values of $\bX$, sorted in the same order as
$\theta_1,\ldots,\theta_{K'}$. We denote their associated sample left singular
vectors by $\fb_\pca^1,\ldots,\fb_\pca^{K'}$ and right singular
vectors by $\gb_\pca^1,\ldots,\gb_\pca^{K'}$, normalized such that
for all $j \neq k \in \{1,\ldots,K'\}$,
\begin{align}\label{rec sigvec}
  \|\fb_\pca^k\|^2=m,\;\;{\fb_\pca^k}^\top \bu_*^k \geq 0,\;\;
{\fb_\pca^j}^\top \fb_\pca^k=0,
\qquad \|\gb_\pca^k\|^2=n,\;\;{\gb_\pca^k}^\top \bv_*^k \geq 0,\;\;
{\gb_\pca^j}^\top \gb_\pca^k=0.  
\end{align}

\subsection{Preliminaries on sample singular vectors and the rectangular
$R$-transform}

\paragraph{Singular vectors and spectral phase transition.}
We review results of \cite{benaych2012singular} that characterize the leading
sample singular values/vectors of $\bX$ and the associated spectral phase
transition. (See \cite[Remark 2.6]{benaych2012singular} for the equivalence with
our setting.)

Let $\Lambda$ be the limit singular value distribution of $\bW$, and recall its
largest point of support $\lambda_+$ defined
in Assumption \ref{assump:rectW}(c). Recall $\gamma=m/n \leq 1$.
Following \cite{benaych2012singular}, for $z \in (\lambda_+, \infty)$, define
\begin{align}\label{eq:rec:Dtransform}
\varphi(z) &= \bbE\left[\frac{z}{z^2 - \Lambda^2}\right],
&
\bar{\varphi}(z) &= \gamma \varphi(z) + \frac{1-\gamma}{z},
&
D(z) &= \varphi(z)\bar{\varphi}(z).
\end{align}
The function $D(z)$ is strictly decreasing on $(\lambda_+,\infty)$.
Let $D^{-1}(z)$ be its functional inverse on $(0,D(\lambda_+))$,
where $D(\lambda_+) = \lim_{z \to \lambda_+} D(z)$. 
For each $k \in \{1,\ldots,K'\}$ where $1/\theta_k^2$ belongs to this domain of $D^{-1}(z)$, define
\begin{align}
\lambda_{\pca,k} & =  D^{-1}(1/\theta_k^2),
&
\mu_{\pca,k}^2 & =
\frac{-2\varphi(\lambda_{\pca,k})}{\theta_k^2 D'(\lambda_{\pca,k})},
& 
\nu_{\pca,k}^2 & = 
\frac{-2\bar{\varphi}(\lambda_{\pca,k})}{\theta_k^2{D'(\lambda_{\pca,k})}}.
\label{eq:rec:pcasol}
\end{align}
The following theorem summarizes results of
\cite[Theorems 2.8 and 2.9]{benaych2012singular}.

\begin{theorem}[\cite{benaych2012singular}]\label{thm: recSpike}
Suppose $\theta_1,\ldots,\theta_{K'}$ are distinct, $\gamma=m/n \leq 1$, and
these are fixed as $m,n \to \infty$. Suppose $\bW$ satisfies
Assumption \ref{assump:rectW}. Then for each $k \in \{1,\ldots,K'\}$
where $\theta_k>(D(\lambda_+))^{-1/2}$, almost surely
\[\lim_{m,n\to\infty} \lambda_k(\bX)=\lambda_{\pca,k}, \quad
\lim_{m,n\to\infty} \left(\frac{{\bff_\pca^k}^\top
\bu^k_*}{m}\right)^2=\mu_{\pca,k}^2, \quad
\lim_{m,n\to\infty} \left(\frac{{\bg_\pca^k}^\top \bv^k_*}{n}\right)^2 =
\nu_{\pca,k}^2,\]
\[\lim_{m,n\to\infty} \left(\frac{{\bff_\pca^k}^\top
\bu^j_*}{m}\right)^2=0, \quad
\lim_{m,n\to\infty} \left(\frac{{\bg_\pca^k}^\top \bv^j_*}{n}\right)^2 = 0
\quad \text{ for all } j \in \{1,\ldots,K'\} \setminus \{k\}.\]
For each other $k \in \{1,\ldots,K'\}$, almost surely
$\lim_{m,n \to \infty} \lambda_k(\bX)=\lambda_+$.
\end{theorem}

This describes a spectral phase transition phenomenon analogous to Theorem
\ref{thm:symspike}, where signal singular values are ``super-critical'' if
$\theta_k>(D(\lambda_+))^{-1/2}$. If $D(\lambda_+)=\infty$, then all signals are
super-critical. Whether this occurs is again determined by the decay of the law
of $\Lambda$ at the spectral edge $\lambda_+$,
c.f.\ \cite[Proposition~2.11]{benaych2012singular}.

\paragraph{Rectangular $R$-transform.}
To introduce the rectangular $R$-transform, first denote
\[T(z)=(1+z)(1+\gamma z), \qquad
U(z)=\frac{-\gamma-1+\sqrt{(1+\gamma)^2+4\gamma z}}{2\gamma},\]
so that $T(U(z-1))=z$. Following~\cite[Eq.\ (8)]{benaych2012singular}, 
the rectangular $R$-transform is defined for $z \in (0,D(\lambda_+))$ by
\begin{equation}\label{eq:rectRdef}
R(z)=U(z(D^{-1}(z))^2 - 1).
\end{equation}
Recalling the rectangular free cumulants $\{\kappa_{2j}\}_{j \geq 1}$ of
$\Lambda$, for sufficiently small $z>0$, $R(z)$ and its derivative admit the
convergent series expansions
\begin{align}\label{eq:rec:Rseries}
R(z)=\sum_{j=1}^\infty \kappa_{2j} z^j,
\qquad
R'(z)=\sum_{j=0}^\infty (j+1)\kappa_{2(j+1)}z^j.
\end{align}
See \cite[Section 3.4]{benaych2009rectangular}, with the notational
identification $H_\mu(z^{-2})=D(z)$ and hence
$(D^{-1}(z))^2=1/H_\mu^{-1}(z)$ as shown in \cite[Appendix
C.3]{fan2020approximate}.

The identity $T(U(z-1))=z$ gives $T(R(z))=(1+\gamma R(z))(1+R(z))
=z(D^{-1}(z))^2$. Then the definition of $\lambda_{\pca,k}$ in
\eqref{eq:rec:pcasol} may be written as
\begin{align*}
\lambda_{\pca,k}^2/\theta_k^2=
(1+\gamma R(\theta_k^{-2}))(1+R(\theta_k^{-2})).
\end{align*}
Rearranging this identity, we may factor $\theta_k^2$ as the product of
\begin{align}
\label{eq:rec:def:thetauv}
\theta_{v,k} &\defeq \frac{\lambda_{\pca,k}}{\sqrt{\gamma} (1 + R(\theta_k^{-2}))}\quad \textnormal{and}\quad
\theta_{u,k} \defeq \frac{\lambda_{\pca,k} \sqrt{\gamma}}{1 + \gamma
R(\theta_k^{-2})}.
\end{align}
Then $\theta_{u,k}$ and $\theta_{v,k}$ satisfy
\begin{align*}
\theta_{u,k} \cdot \theta_{v,k} &= \theta_k^2,
&
\frac{\lambda_{\pca,k}}{\sqrt{\gamma}\theta_{v,k}} - R(\theta_k^{-2}) & = 1,
&
\frac{\lambda_{\pca,k} \sqrt{\gamma}}{\theta_{u,k}} - \gamma R(\theta_k^{-2}) &= 1.
\end{align*}
\cite[Appendix C.3]{fan2020approximate} verifies that
$\mu_{\pca,k}^2$ and $\nu_{\pca,k}^2$ in \eqref{eq:rec:pcasol} 
may also be expressed via $R(\theta_k^{-2})$ and $R'(\theta_k^{-2})$ as
\begin{equation}\label{eq:rec:pcaSolFZequi}
\mu_{\pca,k}^2
=\frac{T(R(\theta_k^{-2})) - \theta_k^{-2}
T'(R(\theta_k^{-2}))R'(\theta_k^{-2})}{1 + \gamma R(\theta_k^{-2})}, \quad
\nu_{\pca,k}^2 =
\frac{T(R(\theta_k^{-2})) - \theta_k^{-2} T'(R(\theta_k^{-2}))R'(\theta_k^{-2})}{1
+ R(\theta_k^{-2})}.
\end{equation}

\subsection{AMP algorithm}

We again isolate the first $K\leq K'$ (unsorted) signal components
for the spectral initialization, and define
\begin{align*}
S &= \diag(\theta_{1},\ldots, \theta_{K})\in\RR^{K\times K}, &
S' &= \diag(\theta_{1},\ldots, \theta_{K'})\in\RR^{K'\times K'},\\
S_u &= \diag(\theta_{u,1},\ldots, \theta_{u,K})\in\RR^{K\times K}, &
 S_v &= \diag(\theta_{v,1},\ldots, \theta_{v,K})\in\RR^{K\times K},\\
\bU_*'&=(\bu_*^1,\ldots,\bu_*^{K'})\in\RR^{m\times K'},
&
\bV_*'&=(\bv_*^1,\ldots,\bv_*^{K'})\in\RR^{n\times K'},\\
\bF_\pca&=(\bff_\pca^1,\ldots,\bff_\pca^K)\in\RR^{m\times K},
&
\bG_\pca&=(\bg_\pca^1,\ldots,\bg_\pca^K)\in\RR^{n\times K}.
\end{align*}
Similar to the symmetric setting, if $\theta_1,\ldots,\theta_K$ are
unknown, then it is sufficient to use consistent estimates of these values, and 
the estimation procedure is outlined in Section~\ref{subsec:implementation_rec}.
We consider an AMP algorithm with matrix-valued
iterates of dimensions $n\times K$ and $m \times K$, initialized spectrally at 
\begin{align}
\begin{aligned}\label{eq:rec:spectralInitAlgo}
\bF_0 &= \bF_\pca, &\bU_0 &= \bU_1 = \bF_\pca S_u^{-1},\\
\bG_0 = \bG_1 &= \bG_\pca, & \bV_0 &= \bG_\pca S_v^{-1}.
\end{aligned}
\end{align}
Duplications of $\bU_1=\bU_0$ and $\bG_1=\bG_0$ are
introduced here to simplify the expression
of the state evolution to follow. For $t \geq 1$ and sequences of Lipschitz
functions $v_t: \bbR^{tK}\to \bbR^K$ and $u_{t+1}:\bbR^{(t+1)K} \to \bbR^{K}$,
this algorithm computes the iterations
\begin{align}
\label{eq:rec:AMPPCA}
\begin{aligned}
\bV_t &= v_t(\bG_1,\ldots, \bG_t),\\
\bF_t &= \bX\bV_t - \sum_{j=0}^t  \bU_j a_{tj}^\top,\\
\bU_{t+1} &= u_{t+1}(\bF_0, \ldots, \bF_t),\\
\bG_{t+1} &= \bX^\top \bU_{t+1} - \sum_{j=0}^{t}
\bV_j b_{t+1, j}^\top.
\end{aligned}
\end{align}
Thus each $v_t(\cdot)$ and $u_{t+1}(\cdot)$ may depend on all preceding iterates
$\bG_s$ and $\bF_s$, including the spectral initializations $\bG_1$ and $\bF_0$.

\paragraph{Debiasing coefficients.}
We define two $(T+1)K \times (T+1) K$ block-lower-triangular matrices 
\begin{align*}
\bphi_T = \begin{pmatrix}
0 & 0 & \cdots & 0 & 0\\
S_u^{-1}& 0 & \cdots & 0 & 0\\
\langle\partial_0\bU_2\rangle & \langle\partial_1\bU_2\rangle & \cdots & 0 & 0\\
\vdots & \vdots & \ddots & \vdots & \vdots\\
\langle\partial_0\bU_T\rangle & \langle\partial_1\bU_T\rangle & \cdots &\langle\partial_{T-1}\bU_T\rangle & 0
\end{pmatrix},
\quad
\bpsi_T = \begin{pmatrix}
S_v^{-1} & 0 & 0 & \cdots & 0 \\
0 &  \langle\partial_1\bV_1\rangle & 0 & \cdots & 0 \\
0 & \langle\partial_1\bV_2\rangle & \langle\partial_2\bV_2\rangle & \cdots & 0 \\
\vdots & \vdots & \vdots & \ddots & \vdots \\
0 & \langle\partial_1\bV_T\rangle & \langle\partial_2\bV_T\rangle & \cdots &\langle\partial_T\bV_T\rangle  
\end{pmatrix}.
\end{align*}
Here $S_u^{-1},S_v^{-1}$ may be interpreted as
$\langle\partial_0\Ub_1\rangle,\langle\partial_0\Vb_0\rangle$, and the first
column of $\bpsi_T$ as $\langle\partial_0\Vb_t\rangle=0$ for all $t\geq 1$.
For each fixed $s \geq 1$,
define $\tilde\kappa_{2s},\hat\kappa_{2s} \in \RR^{K \times K}$ by the
matrix series
\begin{align}\label{eq:rec:kappaseries}
	\tilde\kappa_{2s} = \sum_{j=0}^{\infty} \kappa_{2(j+s)}S^{-2j},\qquad 
	\hat\kappa_{2s} = \sum_{j=0}^{\infty} (j+1) \kappa_{2(j+s)}S^{-2j}.
\end{align}
Then define the $(T + 1)K\times (T + 1)K$ matrices
\begin{align*}
\ba_T &=  \sum_{j=0}^\infty \kappa_{2(j+1)} \bpsi_T (\bphi_T\bphi_T)^{j},
&
\tilde\ba_T &= \sum_{j=0}^\infty  \bpsi_T (\bphi_T\bphi_T)^{j} \odot \tkappa_{2(j+1)},
\\
\bb_T & = \gamma\sum_{j=0}^\infty \kappa_{2(j+1)} \bphi_T (\bpsi_T\bphi_T)^{j},
&
\tilde\bb_T & = \gamma\sum_{j=0}^\infty  \bphi_T (\bpsi_T\bphi_T)^{j} \odot \tkappa_{2(j+1)}.
\end{align*}
Indexing blocks by $\{0,\ldots,T\}$ and writing $[t,s]$ to denote the $K\times K$ submatrix corresponding to row block $t$ and column block $s$, we set the debiasing coefficients of~\eqref{eq:rec:AMPPCA} up to iteration $T$ as
\begin{align}
\label{eq:rec:debiaseCoef}
a_{ts} &=
\begin{cases}
\tilde\ba_T[t,s] &\textnormal{if $s = 0$},\\
\ba_T[t,s] &\textnormal{otherwise}.
\end{cases}
&
b_{ts} &=
\begin{cases}
\tilde\bb_T[t,s] &\textnormal{if $s = 0$},\\
\bb_T[t,s] &\textnormal{otherwise}.
\end{cases}
\end{align}

\paragraph{State Evolution.} The state of this
AMP algorithm is described in terms of two recursively defined
sequences of mean transformations
\[\bmu_T=\begin{pmatrix} \mu_0 \\ \vdots \\ \mu_T \end{pmatrix}
\in \RR^{(T+1)K \times K'},
\qquad 
\bnu_T=\begin{pmatrix} \nu_0 \\ \vdots \\ \nu_T \end{pmatrix}
\in \RR^{(T+1)K \times K'},
\]
and covariance matrices $\bSigma_T = \{\sigma_{st}\}_{0\leq s,t\leq T}$ and
$\bOmega_T = \{\omega_{st}\}_{0\leq s,t\leq T}$.

Let $U_*',V_*' \in \bbR^{K'}$ be random vectors satisfying
$\EE[U_*'{U_*'}^\top]=\EE[V_*'{V_*'}^\top]=\Id$, representing the limit empirical
distributions of rows of $\bU_*',\bV_*'$. Define $\mu_\pca = \diag(\mu_{\pca,1},
\ldots, \mu_{\pca,K})\in \bbR^{K \times K'}$ and $\nu_\pca = \diag(\nu_{\pca,1},
\ldots, \nu_{\pca,K})\in \bbR^{K\times K'}$, where the last $K'-K$ columns are 0.
We initialize 
\begin{align*}
\bmu_0 &= \mu_0 = \mu_\pca,
&
\bSigma_0 &= \sigma_{00} = \Id - \mu_\pca\mu_\pca^\top, \\
\bnu_1 & = 
\begin{pmatrix}
	\nu_0 \\
	 \nu_1
\end{pmatrix}
= 
\begin{pmatrix}
	\nu_\pca \\
	 \nu_\pca
\end{pmatrix},
&
\bOmega_1 &= 
\begin{pmatrix}
\omega_{00} & \omega_{01}\\
\omega_{10} & \omega_{11}
\end{pmatrix}
= 
\begin{pmatrix}
\Id - \nu_\pca\nu_\pca^\top & \Id - \nu_\pca\nu_\pca^\top\\
\Id - \nu_\pca\nu_\pca^\top & \Id - \nu_\pca\nu_\pca^\top
\end{pmatrix}
\end{align*}
corresponding to $\Fb_0$ and $(\Gb_0,\Gb_1)$ in \eqref{eq:rec:spectralInitAlgo}.
For each $t \geq 1$,
having defined $(\bmu_{t-1}, \bSigma_{t-1}, \bnu_{t},\bOmega_{t})$, the next
state $(\bmu_t, \bSigma_t,\bnu_{t+1},\bOmega_{t+1})$ is constructed as follows:

To define $(\bmu_t,\bSigma_t)$, first define joint laws for random vectors
$(U_*',U_0,\ldots,U_t,F_0,\ldots,F_{t-1})$ and
$(V_*',V_0,\ldots,V_t,G_0,\ldots,G_t)$ by
\begin{align}
\label{eq:rec:limitDistPcaAmp}
\begin{aligned}
(F_0,\ldots, F_{t-1})\mid U_*' &\sim
\calN(\bmu_{t-1}\cdot U_*',\,\bSigma_{t-1})\\
U_0 &= U_1 = S_u^{-1}F_0 \text{ and }
U_s = u_s(F_0, \ldots,F_{s-1})\textnormal{ for } s = 2,\ldots,t,\\
(G_0,\ldots, G_{t})\mid V_*' &\sim
\calN( \bnu_{t}\cdot V_*',\,\bOmega_{t})\\
V_0 &= S_v^{-1} G_0 \text{ and }
V_s = v_s(G_1, \ldots,G_{s})\textnormal{ for } s = 1,\ldots,t.
\end{aligned}
\end{align}
Then define $\bmu_t$ to have the blocks
\begin{align}\label{def:recmu}
    \mu_s =  \bbE[V_s{V_*'}^\top] \cdot S'/\sqrt{\gamma} \text{ for each }
s=0,\ldots,t.
\end{align}
For $s=0$, it may be checked from (\ref{eq:rec:def:thetauv}) and
(\ref{eq:rec:pcaSolFZequi}) that
$\theta_{u,k}/\theta_{v,k}=\gamma(\mu_{\pca,k}^2/\nu_{\pca,k}^2)$, and hence
that this coincides with the above initialization $\mu_\pca$.
Next, decompose the second moment matrix of $(U_0,\ldots,U_t)$ into four parts
in the same way as~\eqref{eq: sym Delta_t},
\begin{align*}
\bDelta_t \defeq &\underbrace{{\footnotesize\begin{pmatrix}
		0 & 0 & \cdots & 0 \\
		0 & \bbE[U_1U_1^\top] & \cdots & \bbE[U_1U_t^\top] \\
		\vdots & \vdots & \ddots & \vdots\\
		0 & \bbE[U_tU_1^\top] & \cdots & \bbE[U_tU_t^\top] \\
	\end{pmatrix}}}_{\widebar{\bDelta}_t} + \underbrace{{\footnotesize\begin{pmatrix}
		0 & \bbE[U_0U_1^\top] & \cdots & \bbE[U_0U_t^\top] \\
		0 & 0 & \cdots & 0 \\
		\vdots & \vdots & \ddots & \vdots\\
		0 & 0 & \cdots & 0\\
	\end{pmatrix}}}_{\widetilde\bDelta_t}\notag\\
&\hspace{2in} + \underbrace{{\footnotesize\begin{pmatrix}
		0 & 0 & \cdots & 0 \\
		\bbE[U_1U_0^\top] & 0 & \cdots & 0 \\
		\vdots & \vdots & \ddots & \vdots\\
		\bbE[U_tU_0^\top] & 0 & \cdots & 0\\
	\end{pmatrix}}}_{\widetilde\bDelta_t^\top}
+ \underbrace{{\footnotesize\begin{pmatrix}
		\bbE[U_0U_0^\top] & 0 & \cdots & 0 \\
		0 & 0 & \cdots & 0 \\
		\vdots & \vdots & \ddots & \vdots\\
		0 & 0 & \cdots & 0 \\
	\end{pmatrix}}}_{\widehat\bDelta_t}.
\end{align*}
Decompose the second moment matrix for $(V_0,\ldots,V_t)$ analogously as
$\bGamma_t \defeq \widebar\bGamma_t + \widetilde\bGamma_t + \widetilde\bGamma_t^\top +
\widehat\bGamma_t$. Set
\begin{align*}
\bDelta^{(j)}_t &= \kappa_{2(j+1)}\widebar\bDelta_t + \tilde\kappa_{2(j+1)}\odot\widetilde\bDelta_t + \widetilde\bDelta_t^\top\odot\tilde\kappa_{2(j+1)} + \hat\kappa_{2(j+1)}\odot\widehat\bDelta_t
,\\
{\makebox[\widthof{$\bDelta$}]{$\bGamma$}}^{(j)}_t & =
\kappa_{2(j+1)}\widebar{\makebox[\widthof{$\bDelta$}]{$\bGamma$}}_t +
\tilde\kappa_{2(j+1)}\odot\widetilde{\makebox[\widthof{$\bDelta$}]{$\bGamma$}}_t
+
\widetilde{\makebox[\widthof{$\bDelta$}]{$\bGamma$}}_t^\top\odot\tilde\kappa_{2(j+1)}
+ \hat\kappa_{2(j+1)}\odot\widehat{\makebox[\widthof{$\bDelta$}]{$\bGamma$}}_t.
\end{align*}
Define the large-$n$ limits of $\bphi_t,\bpsi_t$ as
\[\bPhi_t=\Big(\EE[\partial_s u_r(F_0,\ldots,F_{r-1})]\Big)_{r,s \in
\{0,\ldots,t\}},
\qquad \bPsi_t=\Big(\EE[\partial_s v_r(G_1,\ldots,G_r)]\Big)_{r,s \in
\{0,\ldots,t\}}\]
again with the identifications $\partial_0 u_1=S_u^{-1}$,
$\partial_0 v_0=S_v^{-1}$, and $\partial_0 v_t=0$ for $t \geq 1$. Then,
recalling $\bXi^{(j)}[\cdot,\cdot,\cdot,\cdot]$ from~\eqref{eq:rec:Xij}, define
\begin{align}\label{def:recSigma}
\bSigma_t = \sum_{j=0}^\infty \bXi^{(j)}[\bPhi_t, \bPsi_t, \bDelta_t^{(j)},
\bGamma_t^{(j)}].
\end{align}

Next, to define $(\bnu_{t+1},\bOmega_{t+1})$, first define from
$(\bmu_t,\bSigma_t,\bnu_t,\bOmega_t)$ the joint laws for random vectors
$(U_*',U_0,\ldots,U_{t+1}, F_0, \ldots, F_t)$ and $(V_*',V_0, \ldots, V_{t}, G_0,\ldots, G_t)$ according to \eqref{eq:rec:limitDistPcaAmp}.
Then define $\bnu_{t+1}$ to have the blocks $\nu_s = \bbE[U_s{U_*'}^\top] \cdot
S' \sqrt{\gamma}$ for $s=0,\ldots,t+1$. For $s=0$, this coincides with the
initialization $\nu_\pca$.
Then, extending the above definitions of
$\bPhi_t,\bPsi_t,\bDelta_t^{(j)},\bGamma_t^{(j)}$ from $t$ to $t+1$ and recalling
$\bTheta^{(j)}[\cdot,\cdot,\cdot,\cdot]$ from~\eqref{eq:rec:Thetaj}, define
\begin{align*}
\bOmega_{t+1} = \sum_{j=0}^\infty \bTheta^{(j)}[\bPhi_{t+1}, \bPsi_{t+1},
\bDelta_{t+1}^{(j)}, \bGamma_{t+1}^{(j)}].
\end{align*}
The last row block of $\bPsi_{t+1}$ and the last row and column block of
$\bGamma_{t+1}^{(j)}$ are undefined, as we have not yet defined
$(V_{t+1}, G_{t+1})$. As in Section \ref{sec:indrect},
we permit this ambiguity because the last column block of $\bPhi_{t+1}$ is 0, so
the matrix products defining $\bTheta^{(j)}$ do not depend on these undefined blocks.
It may be checked from these recursive definitions that the first $t$ blocks of
$\bmu_t$ coincide with
$\bmu_{t-1}$, and the upper-left $t \times t$ blocks of
$\bSigma_t$ coincide with $\bSigma_{t-1}$. 
The same holds for $\bnu_{t+1}$ and $\bOmega_{t+1}$.

Our main result shows that this state evolution provides a rigorous
characterization of the AMP algorithm
(\ref{eq:rec:spectralInitAlgo}--\ref{eq:rec:AMPPCA})
with spectral initialization, under the following assumptions.

\begin{assumption}\label{assump:rec:AMP}
\begin{enumerate}[label=(\alph*)]
\item $\Ub_*'=(\bu_*^1,\ldots,\bu_*^{K'})$ and
$\Vb_*'=(\bv_*^1,\ldots,\bv_*^{K'})$ are independent of $\Ob$ and $\Qb$, satisfy
(\ref{eq:uvnormalization}), and $\bU_*' \toWtwo U_*'$,
$\bV_*' \toWtwo V_*'$ a.s.\ as $m,n \to \infty$ where
$\EE[U_*'{U_*'}^\top]=\EE[V_*'{V_*'}^\top]=\Id$.

\item Each $u_{t+1}(\cdot)$ is Lipschitz in all arguments. For each
$s=0,\ldots,t$ and all $(\bmu,\bSigma)$ in a sufficiently small open neighborhood of
$(\bmu_t,\bSigma_t)$ defined by (\ref{def:recmu}) and (\ref{def:recSigma}),
$\partial_s u_{t+1}(F_0,\ldots,F_t)$ exists and is continuous on
a set of probability 1 under the marginal law of $(F_0,\ldots,F_t)$ defined by
$(F_0,\ldots,F_t) \mid U_*' \sim \cN(\bmu \cdot U_*',\bSigma)$. The same holds for
each $v_t(\cdot)$ with respect to $(\bnu_t,\bOmega_t)$ and $(G_1,\ldots,G_t)$.
\item The values $\theta_1, \ldots, \theta_{K'}$ are distinct. For each $k \in
\{1,\ldots, K\}$, $\theta_k \geq G(1/\lambda_+)>0$ and there exists some
constant $\iota\in(0,1)$ such that
\begin{align*}
\frac{\lambda_+ \sqrt{\gamma}}{\theta_{u,k}} + \gamma\sum_{j=1}^\infty
\frac{|\kappa_{2j}|}{\theta_k^{2j}\iota^{2j-1}} < 1  \quad\text{and}\quad
\frac{\lambda_+}{\theta_{v,k}\sqrt{\gamma}} + \sum_{j=1}^\infty \frac{|\kappa_{2j}|}{\theta_k^{2j}\iota^{2j-1}}<1.
\end{align*}
\end{enumerate}
\end{assumption}

Assumption \ref{assump:rec:AMP}(c) requires the signal
singular values $\theta_1,\ldots,\theta_K$ for the first $K$ selected signals to
exceed a constant depending only $\gamma=m/n$ and the
law of $\Lambda$. In particular, these signal values are super-critical in the
sense of Theorem \ref{thm:rectindinit}, the series (\ref{eq:rec:Rseries}) for
$R(z)$ and $R'(z)$ are absolutely convergent at $z=1/\theta_k^2$, and the series
(\ref{eq:rec:kappaseries}) defining $\tilde{\kappa}_{2s},\hat{\kappa}_{2s}$ are
also absolutely convergent.

\begin{theorem}\label{thm:rec:AMPPCA}
Consider the rectangular spiked model~\eqref{eq:rec:rankkmodel}, where
Assumptions~\ref{assump:rectW} and \ref{assump:rec:AMP} hold. For any $T\geq 1$,
consider the spectrally initialized AMP algorithm
(\ref{eq:rec:spectralInitAlgo}--\ref{eq:rec:AMPPCA}) up to iteration T, and
define~$(U_*',U_0,\ldots,U_{T+1}, F_0,\ldots,F_T)$
and~ 
$(V_*',V_0,\ldots,V_{T}, G_0,\ldots,G_{T})$ 
by~\eqref{eq:rec:limitDistPcaAmp}. Then almost surely as $m,n\to \infty$,
\begin{align*}
(\bU_*',\bU_0,\ldots,\bU_{T+1}, \bF_0,\ldots,\bF_T)&\toWtwo
(U_*',U_0,\ldots,U_{T+1}, F_0,\ldots,F_T),
\\
(\bV_*',\bV_0,\ldots,\bV_{T}, \bG_0,\ldots,\bG_{T})
&\toWtwo
(V_*',V_0,\ldots,V_{T}, G_0,\ldots,G_{T}).
\end{align*}
\end{theorem}

\begin{remark}\label{remark:PCAinitestimatekappa_rec}
As in Remark \ref{remark:indinitestimatekappa_rec},
we have defined $b_{ts},a_{ts}$ in (\ref{eq:rec:debiaseCoef}) using the free 
cumulants $\{\kappa_j\}$ of the limit spectral distributions, as well as the 
true signal values $\theta_1,\ldots,\theta_K$. 
Theorem~\ref{thm:rec:AMPPCA} then also holds when $b_{ts},a_{ts}$ are replaced 
by $b_{ts}',a_{ts}'$ such that
$\|b_{ts}-b_{ts}'\| \to 0$ and $\|a_{ts}-a_{ts}'\| \to 0$ a.s., and in
particular if $\{\kappa_j\}$, $\{\tilde\kappa_j\}$, $\{\hat\kappa_j\}$, and
$\theta_1,\ldots,\theta_K$ are
replaced by consistent estimates of these quantities.
\end{remark}

\section{Orthogonal AMP for Bayesian PCA}\label{sec:numerical_rec}

We discuss in this section an application 
to estimating the signal vectors $\bu_*^k$ and $\bv_*^k$ in the
preceding signal-plus-noise model in Eq.~\eqref{eq:rec:rankkmodel}.

Analogously, the distributions of $U_*',V_*' \in \RR^{K'}$
for the row-wise limits of $\bU_*',\bV_*'$ may be interpreted as Bayesian
``priors'' for these rows. We also consider a setting where $K \leq K'$, and consider the following
additional assumption for the laws of $U_*'$ and $V_*'$.

\begin{assumption}\label{assump:BayesOAMP_rec}
The last $K'-K$ coordinates of $U_*'$ and $V_*'$ have mean 0, and are independent of
the first $K$ coordinates. For the rectangular model, the same holds for $V_*'$.
\end{assumption}

Similarly, the components $\bu_*^k$ and $\bv_*^k$ should ideally be grouped into
small subsets of dependent signals, with the signals within each subset estimated
together to maximally leverage their joint structure. 

Let us
write
\[(F_0,\ldots,F_t) \mid U_* \sim \cN(\bmu_t \cdot U_*,\,\bSigma_t),
\qquad (G_0,\ldots,G_t) \mid V_* \sim \cN(\bnu_t \cdot V_*,\,\bOmega_t)\]
where $U_*,V_*$ are the first $K$ coordinates of $U_*',V_*'$, and
$\bmu_t,\bnu_t \in \RR^{(t+1)K \times K}$.

\subsection{Bayes-OAMP}\label{subsec:BayesOAMP_rec}

For the rectangular model (\ref{eq:rec:rankkmodel}), we consider analogously the
Bayes-OAMP algorithm which estimates the first $K$ components $\bU_* \in \RR^{m
\times K}$ and $\bV_* \in \RR^{n \times K}$ of $\bU_*'$ and $\bV_*'$, using in
(\ref{eq:rec:spectralInitAlgo}--\ref{eq:rec:AMPPCA}) the denoisers
\begin{align*}
v_t(g_1,\ldots,g_t)&=\EE[V_* \mid (G_1,\ldots,G_t)=(g_1,\ldots,g_t)],\\
u_{t+1}(f_0,\ldots,f_t)&=\EE[U_* \mid (F_0,\ldots,F_t)=(f_0,\ldots,f_t)].
\end{align*}
Recall that $\bG_1=\bG_\pca$ and $\bF_0=\bF_\pca$, so these 
posterior means are conditional also on the spectral initializations. As in the
symmetric setting, the asymptotic mean-squared-errors satisfy
\[\MSE(\bV_{t+1}) \leq \MSE(\bV_t) \leq \MSE(\bG_\pca),
\qquad \MSE(\bU_{t+1}) \leq \MSE(\bU_t) \leq \MSE(\bF_\pca).\]

\subsection{Estimating the debiasing corrections and state evolution}\label{subsec:implementation_rec}

Numerical implementations of the Bayes-OAMP algorithms require estimating the
debiasing coefficients and state evolution parameters that describe the
conditional laws of $(F_0,\ldots,F_t)$ and $(G_0,\ldots,G_t)$. We describe here
one approach for this estimation for the rectangular model.

We estimate the law
of $\Lambda$ by the empirical singular value distribution of $\bX$, with leading
$K'$ singular values removed. We then compute the
empirical moments, and estimate the rectangular free cumulants $\{\kappa_{2s}\}$
via the moment-cumulant relations, see e.g.\ \cite[Section
2.4]{fan2020approximate}. Using this estimated law of $\Lambda$ to compute
$\varphi(\cdot)$, $\bar\varphi(\cdot)$, $D(\cdot)$, and $D'(\cdot)$ in~\eqref{eq:rec:Dtransform}, we
estimate $\theta_k^2$ by $1/\sqrt{D(\lambda_{\pca,k})}$ and
$D^{-1}(1/\theta_k)^2$ by $\lambda_{\pca,k}$ based on \eqref{eq:rec:pcasol}.
We then estimate $\mu_{\pca,k}^2,\nu_{\pca,k}^2$ using \eqref{eq:rec:pcasol},
$R(\theta_k^{-2})$ using \eqref{eq:rectRdef}, $R'(\theta_k^{-2})$ using~\eqref{eq:rec:pcaSolFZequi} and
$\theta_{u,k},\theta_{v,k}$ using \eqref{eq:rec:def:thetauv}.

Applying \eqref{eq:rec:kappaseries} and \eqref{eq:rec:Rseries}, we
estimate $\tilde{\kappa}_2$ by $\diag(\theta_1^2R(\theta_1^{-2}),
\ldots,\theta_K^2R(\theta_K^{-2}))$ and
$\hat{\kappa}_2$ by $\diag(R'(\theta_1^{-2}),\ldots,R'(\theta_K^{-2}))$,
and the remaining $\tilde{\kappa}_{2s}$ and $\hat{\kappa}_{2s}$ based on the
recursions
\begin{align*}
\tilde{\kappa}_{2s} &= \sum_{j = 0}^\infty \kappa_{2(s+j)} S^{-2j}
= \kappa_{2s} + \sum_{j=0}^\infty \kappa_{2(s+1+j)} S^{-2j}\cdot S^{-2} =
\kappa_{2s} + \tilde{\kappa}_{2(s+1)} S^{-2},\\
\hat{\kappa}_{2s} &= \sum_{j = 0}^\infty (j+1) \kappa_{2(s+j)} S^{-2j}
= \tilde{\kappa}_{2s} + \sum_{j = 0}^\infty (j+1)\cdot \kappa_{2(s+1+j)} S^{-2j} \cdot S^{-2}
= \tilde{\kappa}_{2s} + \hat{\kappa}_{2(s+1)} S^{-2}. 
\end{align*}
This yields consistent estimates of the debiasing coefficients $b_{ts},a_{ts}$.
The state evolution parameters $(\bmu_t,\bSigma_t,\bnu_t,\bOmega_t)$ are then
estimated using $\bphi_t,\bpsi_t$ and the empirical second-moment matrices
$m^{-1}\bU_s^\top \bU_t$ and $n^{-1}\bV_s^\top \bV_t$ as in the symmetric setting.

\section{Proof for symmetric square matrices}\label{appendix:symproof}

\subsection{State evolution for auxiliary AMP}

As discussed in Section \ref{sec: sym proof sketch}, in the setting of Theorem
\ref{thm:sym}, we consider an auxiliary
AMP algorithm starting at time index $-\tau$, with an initialization
$\bU_{-\tau}^{(\tau)} \in \RR^{n \times K}$ independent of $\bW$ and having
the iterates, for $t=-\tau,-\tau+1,-\tau+2,\ldots$
\begin{equation}\label{eq:AMPspikedindinit}
\bF_t^{(\tau)}=\bX\bU_t^{(\tau)}-\sum_{s=-\tau}^t
\bU_s^{(\tau)}b_{ts}^{(\tau)\top}, \qquad
\bU_{t+1}^{(\tau)}=u_{t+1}(\bF_{-\tau}^{(\tau)},\ldots,\bF_t^{(\tau)}).
\end{equation}
The coefficients $b_{ts}^{(\tau)}$ are defined as follows: For $T \geq 1$, we define
\[\bphi^{(\tau)}_{\all,T}=\Big(\langle \partial_s \bU_r^{(\tau)}
\rangle\Big)_{r,s \in \{-\tau,\ldots,T\}},
\qquad \bb_{\all,T}^{(\tau)}=\sum_{j=0}^\infty
\kappa_{j+1}\big(\bphi_{\all,T}^{(\tau)}\big)^j\]
where $\{\kappa_j\}_{j \geq 1}$ are the free cumulants of the limit eigenvalue
distribution $\Lambda$ for $\bW$. We take the
above debiasing coefficients up to iteration $T$ to be the blocks,
for $s,t \in \{-\tau,\ldots,T\}$,
\begin{equation}\label{eq: sym debias coef aux}
b_{ts}^{(\tau)}=\bb_{\all,T}^{(\tau)}[t,s]
\end{equation}

Supposing that $(\bU_*',\bU_{-\tau}^{(\tau)}) \toWtwo
(U_*',U_{-\tau}^{(\tau)})$, we define recursively the following state evolution:
Having defined the joint law of $(U_*',U_{-\tau}^{(\tau)},\ldots,U_t^{(\tau)},
F_{-\tau}^{(\tau)},\ldots,F_{t-1}^{(\tau)})$ for some $t \geq -\tau$, we define
$\bmu_{\all,t}^{(\tau)} \in \RR^{(t+\tau+1)K \times K'}$ having the blocks
\[\mu_s^{(\tau)}=\EE[U_s{U_*'}^\top] \cdot S' \in \RR^{K \times
K'} \text{ for each } s=-\tau,\ldots,t.\]
Recalling the function $\bTheta^{(j)}[\cdot,\cdot]$ from~\eqref{eq:Sigma}, we define also the $(t+\tau+1)K \times (t+\tau+1)K$ matrices
\[\bDelta_{\all,t}^{(\tau)}=\Big(\EE[U_r^{(\tau)}U_s^{(\tau)\top}]\Big)_{r,s \in
\{-\tau,\ldots,t\}}, \qquad
\bPhi_{\all,t}^{(\tau)}=\Big(\EE[\partial_s u_r(F_{-\tau}^{(\tau)},\ldots,F_{r-1}^{(\tau)})]
\Big)_{r,s \in \{-\tau,\ldots,t\}},\]
\[\bSigma_{\all,t}^{(\tau)}=\sum_{j=0}^\infty \bTheta^{(j)}
[\bPhi_{\all,t}^{(\tau)},\,\kappa_{j+2}\bDelta_{\all,t}^{(\tau)}].\]
Then we define the next joint law of $(U_*',U_{-\tau}^{(\tau)},\ldots,
U_{t+1}^{(\tau)},F_{-\tau}^{(\tau)},\ldots,F_t^{(\tau)})$ by
\[(F_{-\tau}^{(\tau)},\ldots,F_t^{(\tau)}) \mid U_*' \sim
\cN\Big(\bmu_{\all,t}^{(\tau)} \cdot U_*',\;
\bSigma_{\all,t}^{(\tau)}\Big), \quad
U_{s+1}^{(\tau)}=u_{s+1}(F_{-\tau}^{(\tau)},\ldots,F_s^{(\tau)})
\text{ for } s=-\tau,\ldots,t.\]

\begin{corollary}\label{cor:auxAMPSE}
In the symmetric spiked model (\ref{eq:sym:rankkModel}), suppose
Assumption \ref{assump:symW} holds for $\bW$. Suppose the initialization
$\bU_{-\tau}^{(\tau)} \in
\RR^{n \times K}$ is independent of $\bW$, and $(\bU_*',\bU_{-\tau}^{(\tau)})
\toWtwo (U_*',U_{-\tau}^{(\tau)})$ a.s.\ as $n \to \infty$. Suppose each function
$u_{t+1}(\cdot)$ is Lipschitz, and each derivative $\partial_s
u_{t+1}(F_{-\tau}^{(\tau)},\ldots,F_t^{(\tau)})$ exists and is continuous on a
set of probability 1 under the above law of
$(F_{-\tau}^{(\tau)},\ldots,F_t^{(\tau)})$. Then for any
$T \geq 1$, a.s.\ as $n \to \infty$,
\[(\bU_*',\bU_{-\tau}^{(\tau)},\ldots,\bU_{T+1}^{(\tau)},
\bF_{-\tau}^{(\tau)},\ldots,\bF_T^{(\tau)})
\toWtwo (U_*',U_{-\tau}^{(\tau)},\ldots,U_{T+1}^{(\tau)},
F_{-\tau}^{(\tau)},\ldots,F_T^{(\tau)}).\]
\end{corollary}
\begin{proof}
The proof is similar to \cite[Theorem 3.1(a)]{fan2020approximate}.
Recalling $\bX=n^{-1}\bU_*' S' {\bU_*'}^\top+\bW$,
we write the iterations (\ref{eq:AMPspikedindinit}) as
\[\bF_t^{(\tau)}=\frac{1}{n}\bU_*' \cdot S'\cdot{\bU_*'}^\top \bU_t^{(\tau)}
+\bW\bU_t^{(\tau)}-\sum_{s=-\tau}^t \bU_s^{(\tau)}b_{ts}^{(\tau)\top}.\]
Then, approximating $n^{-1}S' \cdot {\bU_*'}^\top \bU_t^{(\tau)}$ by
$S' \cdot \EE[U_*'U_t^{(\tau)\top}]=\mu_t^{(\tau)\top}$, we consider an
alternative AMP sequence with the same initialization
$\tilde{\bU}_{-\tau}=\bU_{-\tau}^{(\tau)}$ and ``side information'' $\bU_*'$,
defined by
\begin{align*}
\tilde{\bZ}_t&=\bW\tilde{\bU}_t
-\sum_{s=-\tau}^t \tilde{\bU}_s\tilde{b}_{ts}^\top,\\
\tilde{\bF}_t&=\tilde{\bZ}_t+\bU_*' \mu_t^{(\tau)\top},\\
\tilde{\bU}_{t+1}&=\tilde{u}_{t+1}(\tilde{\bZ}_{-\tau},\ldots,\tilde{\bZ}_t,
\bU_*') \defeq u_{t+1}(\tilde{\bF}_{-\tau},\ldots,\tilde{\bF}_t).
\end{align*}
Here, the debiasing coefficients are defined as
\[\tilde{b}_{ts}=\langle \partial_s
\tilde{u}_{t+1}(\tilde{\bZ}_{-\tau},\ldots,\tilde{\bZ}_t,\bU_*') \rangle
=\langle \partial_s u_{t+1}(\tilde{\bF}_{-\tau},\ldots,\tilde{\bF}_t) \rangle.\]
Using Theorem \ref{thm:symindinit} to analyze this AMP algorithm, we have
a.s.\ as $n \to \infty$ that
\[(\bU_*',\tilde{\bU}_{-\tau},\ldots,\tilde{\bU}_{T+1},\tilde{\bF}_{-\tau},
\ldots,\tilde{\bF}_T) \toWtwo (U_*',U_{-\tau}^{(\tau)},\ldots,
U_{t+1}^{(\tau)},F_{-\tau}^{(\tau)},\ldots,F_t^{(\tau)})\]
for each fixed $t \geq -\tau$, as described by the above state evolution.
Then, applying the same inductive argument as
in \cite[Theorem 3.1(a)]{fan2020approximate}, we obtain a.s.\ as $n \to \infty$
\[n^{-1}\|\bU_t^{(\tau)}-\tilde{\bU}_t\|_F^2 \to 0,
\quad n^{-1}\|\bF_t^{(\tau)}-\tilde{\bF}_t\|_F^2 \to 0\]
for each fixed $t \geq -\tau$, which implies this corollary.
\end{proof}

As discussed in Section \ref{sec: sym proof sketch},
we now specialize this auxiliary AMP algorithm (\ref{eq:AMPspikedindinit})
to the two-phase algorithm
\begin{align*}
u_{t+1}(\bF_{-\tau}^{(\tau)},\ldots,\bF_t^{(\tau)}) = \begin{cases}
\bF_t^{(\tau)} S^{-1}
& \text{ for } -\tau+1 \leq t <0,\\
u_{t+1}(\bF_0^{(\tau)},\ldots,\bF_t^{(\tau)}) & \text{ for } t\geq 0.
\end{cases}
\end{align*}
The auxiliary algorithm is initialized at
\[\Ub_{-\tau}^{(\tau)} = (\ub_{-\tau}^1,\ldots,\ub_{-\tau}^K) \text{ with }
\ub_{-\tau}^k = \Big(\mu_{\pca,k}\ub_*^k +
\sqrt{1-\mu_{\pca,k}^2}\cdot\bz_k\Big)\Big/\theta_k \text{ for each } k=1,\ldots,K,\]
where $\mu_{\pca,k}$ is as defined in Theorem \ref{thm:symspike}, and
$\bz_1,\ldots,\bz_K$ are independent standard Gaussian random vectors also
independent of $\bW$.

For each $t\geq 1$, we adopt the following block decomposition of
$\bphi_{\all,t}^{(\tau)}$, where the first block corresponds to indices
$\{-\tau,\ldots,-1\}$ and the second to indices $\{0,\ldots,t\}$:
\begin{align*}
\bphi_{\all,t}^{(\tau)} = \begin{pmatrix}
\bphi_{--}^{(\tau)} & \bphi_{-t}^{(\tau)}\\
\bphi_{t-}^{(\tau)} & \bphi_t^{(\tau)}
\end{pmatrix},\quad \text{where } \bphi_{--}^{(\tau)}\in\RR^{\tau K\times \tau K} \text{ and } \bphi_t^{(\tau)}\in\RR^{(t+1)K\times(t+1)K}.
\end{align*}
Note that, due to the lower-triangular form of $\bphi_{\all,t}^{(\tau)}$ and the linear update rule for the first $\tau$ steps, we have $\bphi_{-t}^{(\tau)}=0$ and
\begin{align*}
    \bphi_{--}^{(\tau)} = \begin{pmatrix}
        0 & 0 & \cdots & 0 & 0\\
        S^{-1} & 0 & \cdots & 0 & 0\\
        0 & S^{-1} & \cdots & 0 & 0\\
        \vdots & \vdots & \ddots & \vdots & \vdots\\
        0 & 0 & \cdots & S^{-1} & 0
    \end{pmatrix}, \qquad \bphi_{t-}^{(\tau)} = \begin{pmatrix}
        0 & \cdots & 0 & S^{-1}\\
        0 & \cdots & 0 & 0\\
        \vdots & \ddots & \vdots & \vdots\\
        0 & \cdots & 0 & 0
    \end{pmatrix}.
\end{align*}
Applying Lemma \ref{lemma: aux algebra fact} with $\Ab=\bphi_{\all,t}^{(\tau)}$ and $\Bb=S^{-1}$, for all $r\in \{0,\ldots,t\}$ and $c\in \{1,\ldots,\tau\}$,
\begin{align}\label{eq: sym aux Phi prop}
    (\bphi_{\all,t}^{(\tau)})^j[r,-c] = \begin{cases}
    (\bphi_t^{(\tau)})^{j-c}[r,0] S^{-c} & 1\leq c\leq j,\\
    0 & j<c.
    \end{cases}
\end{align}

Similarly, for the state evolution, we decompose
\begin{align}\label{eq: sym aux DeltaPhiSigma}
\bmu_{\all,t}^{(\tau)}=\begin{pmatrix} \bmu_{-}^{(\tau)} \\ \bmu_t^{(\tau)}
\end{pmatrix},\quad
    \bDelta_{\all,t}^{(\tau)} = \begin{pmatrix}
	    \bDelta_{--}^{(\tau)} & \bDelta_{-t}^{(\tau)}\\
	    \bDelta_{t-}^{(\tau)} & \bDelta_t^{(\tau)}
	\end{pmatrix},\quad \bPhi_{\all,t}^{(\tau)} = \begin{pmatrix}
	    \bPhi_{--}^{(\tau)} & \bPhi_{-t}^{(\tau)}\\
	    \bPhi_{t-}^{(\tau)} & \bPhi_t^{(\tau)}
	\end{pmatrix},\quad \bSigma_{\all,t}^{(\tau)} = \begin{pmatrix}
        \bSigma_{--}^{(\tau)} & \bSigma_{-t}^{(\tau)} \\
        \bSigma_{t-}^{(\tau)} & \bSigma_t^{(\tau)}
    \end{pmatrix}.
\end{align}

\subsection{Phase I -- linear AMP}

We first establish the convergence of the iterates and the associated state
evolution of the first $\tau$ steps of this auxiliary AMP algorithm, as $\tau \to
\infty$. For notational convenience, in this section only, we re-index these
iterates as $1,2,\ldots,\tau$ and provide a standalone result for such a linear AMP
algorithm.

Let $\Ub_1=(\ub_1^1,\ldots,\ub_1^K)$ with each $\ub_1^k =
(\mu_{\pca,k}\ub_*^k + \sqrt{1 - \mu_{\pca,k}^2} \cdot \bz_k) / \theta_k$.
The first $\tau$ iterates of the above auxiliary AMP algorithm has the structure
of the following linear AMP:
\begin{equation}\label{eq: linear AMP}
        \Fb_t = \Xb\Ub_t - \sum_{i=1}^t \kappa_{t-i+1} \Ub_i S^{-(t-i)},
\qquad \Ub_{t+1} = \Fb_t S^{-1}.
\end{equation}
We write $\bF_0=\Ub_1 S$.
Up to iterate $\tau$, let
$\bmu_\tau=(\mu_t)_{1\leq t\leq\tau}$ and $\bSigma_\tau=(\sigma_{st})_{1 \leq
s,t \leq \tau}$ be the parameters of
the state evolution describing this linear AMP, where $\mu_t \in \RR^{K \times
K'}$ and $\sigma_{st} \in \RR^{K \times K}$.
Recall
\[\mu_\pca = \diag(\mu_{\pca,1},\ldots,\mu_{\pca,K})\in\RR^{K\times K'},
\quad \Id-\mu_\pca\mu_\pca^\top=\diag(1-\mu_{\pca,1}^2,\ldots,1-\mu_{\pca,K}^2)
\in \RR^{K \times K}.\]
Then Corollary \ref{cor:auxAMPSE} ensures
\begin{align*}
    (\Ub_*',\Ub_1,\ldots,\Ub_\tau,\Fb_1,\ldots,\Fb_\tau) \toWtwo
(U_*',U_1,\ldots,U_\tau,F_1,\ldots,F_\tau)
\end{align*}
where the limiting distribution is defined by
\[(F_1,\ldots,F_\tau) \mid U_*' \sim \cN(\bmu_\tau \cdot U_*',\bSigma_\tau),\]
\begin{align}\label{eq: sym se linear AMP}
U_1 \sim S^{-1} \cdot \cN(\mu_\pca \cdot U_*',\,\Id-\mu_\pca\mu_\pca^\top),
\quad U_t= S^{-1} F_{t-1} \text{ for } 2\leq t\leq\tau.
\end{align}

\begin{lemma}\label{lemma: sym linear AMP}
Under Assumptions \ref{assump:symW} and \ref{assump:symAMP}(a) and (c),
the following hold for the linear AMP algorithm \eqref{eq: linear AMP}:
\begin{enumerate}[label=(\alph*)]
\item $\lim_{t\to\infty}\uplim_{n\to\infty} \|\Fb_t - \Fb_\pca\|_\F/\sqrt{n}=0$ a.s.
\item The state evolution satisfies
$\mu_t=\mu_\pca$ for every
$t \geq 1$, and $\lim_{\min(s,t)\to\infty} \sigma_{st}=\Id-\mu_\pca\mu_\pca^\top$.
\end{enumerate}
\end{lemma}

\begin{proof}
Recall the sample eigenvalues $\lambda_k(\bX)$ for $k=1,\ldots,K'$
from (\ref{eq:lambdaX}), constituting the $K_+$ largest and $K_-$ smallest
eigenvalues of $\bX$, with associated eigenvectors $\fb_\pca^k$
satisfying (\ref{eq:fpca}). Denote the remaining eigenvalues
and eigenvectors as $\lambda_i(\bX)$ and $\fb_\pca^i$ for $i=K'+1,\ldots,n$ in
any order, with the same normalization $\|\fb_\pca^i\|=\sqrt{n}$.
Let
\[\cS=\Big\{k \in \{1,\ldots,K'\}:\theta_k>1/G(\lambda_+)
\text{ or } \theta_k<1/G(\lambda_-)\Big\}\]
be the indices corresponding to ``super-critical'' signal eigenvalues as
characterized by Theorem \ref{thm:symspike}.
Denote $\|\Lambda\|_\infty=\max(|\lambda_+|,|\lambda_-|)$,
fix a small constant $\delta>0$, and define the event
\[\cE_n=\big\{|\lambda_i(\bX)|<\|\Lambda\|_\infty+\delta \text{ for all }
i \notin \cS\big\}.\]
Then $\cE_n$ occurs almost surely for all large $n$,
where this bound for $i \in \{1,\ldots,K'\} \setminus \cS$ follows from Theorem
\ref{thm:symspike}, and that for $i \in \{K'+1,\ldots,n\}$ follows from
Assumption \ref{assump:symW}(c) and Weyl's eigenvalue interlacing inequality.

Let $\fb_t^k$ be the $k^{\text{th}}$ column of the linear AMP iterate
$\Fb_t$. For part (a), it suffices to show
\[\lim_{t\to\infty}\limsup_{n\to\infty} \|\fb_t^k - \fb_\pca^k\|/\sqrt{n}=0\]
for each $k\in\{1,\ldots,K\}$. Fixing any such $k$,
by the definition of linear AMP in \eqref{eq: linear AMP}, we have
\begin{align}\label{eq: sym f linear AMP}
    \fb_t^k = \frac{1}{\theta_k}\cdot \Xb\fb_{t-1}^k - \sum_{j=1}^t \frac{\kappa_{t-j+1}}{\theta_k^{t-j+1}}\cdot \fb_{j-1}^k.
\end{align}

We first show that the component of $\fb_t^k$ orthogonal to
$\fb_\pca^k$ vanishes a.s.\ in the limits $t \to \infty$ and $n \to\infty$.
Define $r_t^{k,i} = {\fb_\pca^i}^\top\fb_t^k/n$ for each $i \in
\{1,\ldots,n\}$. Then applying (\ref{eq: sym f linear AMP}),
\begin{align}\label{eq: sym residual linear AMP}
    r_t^{k,i} &= \frac{1}{\theta_k}\cdot \frac{(\fb_\pca^i)^\top
\Xb\fb_{t-1}^k}{n} - \sum_{j=0}^{t-1} \frac{\kappa_{t-j}}{\theta_k^{t-j}}\cdot
\frac{(\fb_\pca^i)^\top\fb_{j}^k}{n}\notag\\
    &= \frac{\lambda_i(\bX)}{\theta_k}\cdot r_{t-1}^{k,i} - \sum_{j=0}^{t-1}
\frac{\kappa_{t-j}}{\theta^{t-j}_k}\cdot r_{j}^{k,i}.
\end{align}

For any $i\in \cS \setminus\{k\}$, by the initialization $\fb_0^k =
\theta_k\ub_1^k=\mu_{\pca,k}\bu_*^k+\sqrt{1-\mu_{\pca,k}^2} \cdot \bz_k$, we have
\begin{align*}
r_0^{k,i} &= \frac{\mu_{\pca,k}}{n}(\fb_\pca^i)^\top\ub_*^k +
\frac{\sqrt{1-\mu_{\pca,k}^2}}{n}(\fb_\pca^i)^\top\bz_k \to 0
\end{align*}
a.s.\ as $n \to \infty$,
where the first term converges to 0 by Theorem \ref{thm:symspike},
and the second term converges to 0 since ${\fb_\pca^i}^\top \bz_k/n
\sim \cN(0,1/n)$.
Thus, it follows from the recursion \eqref{eq: sym residual linear AMP} that
\begin{align}\label{eq: sym residual spike}
    \lim_{n\to\infty} r_{t}^{k,i} = 0 \text{ a.s.\ for each fixed }
t\geq 0 \text{ and } i\in \cS \setminus\{k\}.
\end{align}

For any $i\notin \cS$, consider a space $\cX$ of bounded infinite-dimensional
vectors with elements in $[0,\infty)$.
For each $t\geq 0$, we define an element $\brho_{t}^{k,i}\in\cX$ by
padding 0's after $(r_{t}^{k,i},r_{t-1}^{k,i},\ldots,r_{0}^{k,i})$, i.e.,
\begin{align*}
    \brho_{t}^{k,i} = (r_{t}^{k,i}, r_{t-1}^{k,i},\ldots, r_{0}^{k,i},0,0,\ldots).
\end{align*}
For some $\iota\in(0,1)$ chosen as in Assumption \ref{assump:symAMP}(c), let us consider a norm $\|\cdot\|$ on $\cX$ defined by
\begin{align*}
    \|(x_0,x_{-1},x_{-2},\ldots)\| = \sup_{j\geq 0} |x_{-j}|\cdot \iota^{j}.
\end{align*}
Consider a map $g:\cX\to\cX$ defined as $g(\brho_{t-1}^{k,i})=\brho_{t}^{k,i}$
for each $t\geq 1$. We verify that $g$ is contractive with respect to the norm
$\|\cdot\|$: Let $\{\brho_t\}_{t\geq1}$ and $\{\tilde\brho_t\}_{t\geq 1}$ be two sequences of vectors in $\cX$ given by
\begin{align*}
    \begin{cases}
    \brho_t = (r_t,r_{t-1},\ldots,r_0,0,\ldots)\\
    \tilde\brho_t = (\tilde r_t,\tilde r_{t-1},\ldots,\tilde r_0,0,\ldots)
    \end{cases}
\end{align*}
where both $\{r_t\}_{t\geq 1}$ and $\{\tilde r_t\}_{t\geq 1}$ satisfy the same recursion as in \eqref{eq: sym residual linear AMP}. Then we have
\begin{align*}
    \|g(\brho_{t-1}) - g(\tilde\brho_{t-1})\| = \|\brho_t - \tilde\brho_t\|
    &= \sup_{0\leq j\leq t} |r_j - \tilde r_j| \cdot \iota^{t-j}
    = \max\left\{|r_t - \tilde r_t|, \iota\cdot\|\brho_{t-1} - \tilde\brho_{t-1}\|\right\}.
\end{align*}
We then need to control $|r_t-\tilde r_t|$. It follows from \eqref{eq: sym
residual linear AMP} that
\begin{align*}
    |r_t-\tilde r_t| &= \left| \frac{\lambda_i(\bX)}{\theta_k}\cdot(r_{t-1}-\tilde r_{t-1}) - \sum_{j=0}^{t-1} \frac{\kappa_{t-j}}{\theta_k^{t-j}}(r_j-\tilde r_j)\right|\\
    &\leq \frac{|\lambda_i(\Xb)|}{|\theta_k|}|r_{t-1}-\tilde r_{t-1}| + \sum_{j=0}^{t-1}\frac{|\kappa_{t-j}|}{|\theta_k|^{t-j}\iota^{t-1-j}} |r_j-\tilde r_j|\iota^{t-1-j}\\
    &\leq \left(\frac{|\lambda_i(\bX)|}{|\theta_k|} + \sum_{j=1}^t\frac{|\kappa_j|}{|\theta_k|^j\iota^{j-1}}\right) \cdot \max_{0\leq j\leq t-1}|r_{j}-\tilde r_{j}|\iota^{t-1-j}\\
    &\leq \underbrace{\left(\frac{\|\Lambda\|_\infty+\delta}{|\theta_k|} + \sum_{j=1}^\infty \frac{|\kappa_j|}{|\theta_k|^j\iota^{j-1}}\right)}_{\eta} \cdot \|\brho_{t-1} - \tilde\brho_{t-1}\|
\end{align*}
where the last inequality holds on the event $\cE_n$ defined above. For sufficiently
small $\delta>0$, we have $\eta\in(0,1)$ by Assumption~\ref{assump:symAMP}(c).
Then denoting $\rho=\max(\eta,\iota)\in(0,1)$, we obtain
\begin{align*}
    \|g(\brho_{t-1}) - g(\tilde\brho_{t-1})\| &\leq \rho \cdot \|\brho_{t-1} - \tilde\brho_{t-1}\|.
\end{align*}
Therefore, $g$ is a $\rho$-contraction. Applying this property to
$\{\brho_{t}^{k,i}\}_{t\geq 1}$ yields
\begin{align}\label{eq: sym f project}
    |r_{t}^{k,i}| &\leq \|(r_{t}^{k,i},r_{t-1}^{k,i},\ldots,r_{0}^{k,i},0,\ldots) - (0,0,\ldots)\|
    \leq \rho^t \cdot \|(r_{0}^{k,i},0,\ldots) - (0,0,\ldots)\|
    = \rho^t \cdot |r_{0}^{k,i}|.
\end{align}
This holds simultaneously for all $i\notin \cS$ on the event $\cE_n$.

Now write $\fb_t^k=\xi_t^k\fb_\pca^k + \rb_t^k$ where $\rb_t^k$ is
orthogonal to $\fb_\pca^k$. Since $\{\fb_\pca^i/\sqrt{n}\}_{i=1,\ldots,n}$ is
an orthonormal basis of $\RR^n$, we can expand $\rb_t^k$ as
\begin{align*}
\rb_t^k = \sum_{i=1,i\neq k}^n \frac{{\fb_\pca^i}^\top\fb_t^k}{n}
\cdot \fb_\pca^i = \underbrace{\sum_{i \in \cS \setminus \{k\}}
r_{t}^{k,i}\fb_\pca^i}_{\cI_1} + \underbrace{\sum_{i \notin \cS} r_{t}^{k,i}\fb_\pca^i}_{\cI_2}.
\end{align*}
By \eqref{eq: sym residual spike}, we have
$\lim_{n\to\infty}\|\cI_1\|/\sqrt{n}=0$ for each fixed $t \geq 1$. For $\cI_2$,
we apply \eqref{eq: sym f project} on the event $\cE_n$:
\begin{align*}
\frac{\|\cI_2\|^2}{n} &= \sum_{i \notin \cS} (r_{t}^{k,i})^2 \leq \rho^{2t} \cdot
\sum_{i \notin \cS} (r_0^{k,i})^2 \leq \rho^{2t} \cdot \frac{\|\rb_0^k\|^2}{n}
\leq \rho^{2t} \cdot \frac{\|\fb_0^k\|^2}{n}.
\end{align*}
By the initialization
$\fb_0^k=\mu_{\pca,k}\ub_*^k+\sqrt{1-\mu_{\pca,k}^2}\cdot\bz_k$, we have
$\lim_{n \to \infty} \|\fb_0^k\|^2/n=1$.
Thus, as desired,
\begin{align}\label{eq: sym residual bound}
    \lim_{t\to\infty}\uplim_{n\to\infty} \frac{\|\rb_t^k\|}{\sqrt{n}}
 = 0.
\end{align}

We now show that $\xi_t^k\defeq{\fb_t^k}^\top \fb_\pca^k/n \to 1$ by using the
state evolution \eqref{eq: sym se linear AMP} of linear AMP.
By this state evolution, we have
\begin{align*}
    \mu_{t+1} = \EE[U_{t+1}{U_*}'^\top] \cdot S' = S^{-1} \cdot \EE[F_t
{U_*'}^\top] \cdot S' = S^{-1}\cdot \EE[(\mu_t U_*' + Z_t){U_*'}^\top] \cdot S' = S^{-1} \mu_t S'
\end{align*}
where the second equality follows from $U_{t+1} = S^{-1} F_t$ and the last
equality is due to $Z_t\independent U_*'$ and $\EE[{U_*'}{U_*'}^\top]=\Id$. 
From the definition of $\bU_1$, it may be checked that
$\mu_1 = \mu_\pca \in \RR^{K \times K'}$. Then it
follows from the above that $\mu_t=\mu_\pca$ for all $t\geq 1$. Thus for each $k\in\{1,\ldots, K\}$, we have
\begin{align*}
\mu_{\pca,k}=\lim_{n \to \infty} \frac{{\fb_\pca^k}^\top \ub_*^k}{n},
\qquad
    \mu_{\pca,k} = \lim_{n\to\infty}\frac{{\fb_t^k}^\top\ub_*^k}{n} =
    \lim_{n\to\infty} \xi_t^k\cdot\frac{{\fb_\pca^k}^\top\ub_*^k}{n} +
    \frac{{\rb_t^k}^\top\ub_*^k}{n}
\end{align*}
where the left equality follows from Theorem~\ref{thm:symspike}, and the right
equality applies $\mu_t=\mu_\pca$. This further implies that
\begin{align*}
    \uplim_{n\to\infty} \left|(\xi_t^k-1)\frac{{\fb_\pca^k}^\top\ub_*^k}{n}\right| &\leq \uplim_{n\to\infty} \frac{\|\rb_t^k\|}{\sqrt{n}}.
\end{align*}
Since $\lim_{n\to\infty} {\fb_\pca^k}^\top\ub_*^k/n=\mu_{\pca,k}\neq 0$ a.s.\ by Theorem \ref{thm:symspike}, taking the limit as $t\to\infty$ on both sides, it follows from \eqref{eq: sym residual bound} that
\begin{align}\label{eq: sym xi linear AMP}
    \lim_{t\to\infty}\uplim_{n\to\infty} |\xi_t^k-1| = 0.
\end{align}
Then, recalling $\fb_t^k=\xi_t^k\fb_\pca^k+\rb_t^k$ and
combining with \eqref{eq: sym residual bound},
\begin{align*}
    \lim_{t\to\infty}\uplim_{n\to\infty} \frac{\|\fb_t^k - \fb_\pca^k\|}{\sqrt{n}}=0.
\end{align*}
This shows both part (a) and the claim about $\mu_t$ in part (b).

It remains to show the convergence of $\sigma_{st}$ in part (b). By \eqref{eq:
sym se linear AMP} and the above identities
$\EE[F_t{U_*'}^\top]=\mu_t=\mu_\pca$, we have
\begin{align*}
    \sigma_{st} &= \EE[(F_s-\mu_s U_*')(F_t-\mu_t U_*')^\top] = \EE[F_sF_t^\top] -
\mu_\pca\mu_\pca^\top = \lim_{n\to\infty} \frac{\Fb_s^\top\Fb_t}{n} -
\mu_\pca\mu_\pca^\top.
\end{align*}
Thus, using the notation $\langle \bu\bv \rangle=\bu^\top \bv/n$
and recalling the decomposition $\bff_t^k = \xi_t^k \bff_\pca^k + \br_t^k$,
we can bound the difference between $(1-\mu_{\pca,k}^2)\ind\{k=k'\}$ and the 
$(k,k')^{\text{th}}$ entry of $\sigma_{st}$ as
\begin{align*}
    &\quad|\sigma_{st}[k,k'] - (1-\mu_{\pca,k}^2)\ind\{k=k'\}|\\
    &\leq \uplim_{n\to\infty}\left| \xi_s^k\xi_t^{k'}\langle\fb_\pca^k\fb_\pca^{k'}\rangle - \ind\{k=k'\} + \xi_s^k\langle\fb_\pca^k\rb_t^{k'}\rangle + \xi_t^{k'}\langle\fb_\pca^{k'}\rb_s^k\rangle + \langle\rb_s^k\rb_t^{k'}\rangle\right|\\
    &\leq \uplim_{n\to\infty} \ind\{k=k'\}|\xi_s^k\xi_t^{k'}-1| + |\xi_s^k\langle\fb_\pca^k\rb_t^{k'}\rangle| + |\xi_t^{k'}\langle\fb_\pca^{k'}\rb_s^k\rangle| + |\langle\rb_s^k\rb_t^{k'}\rangle|
\end{align*}
where the last inequality applies the triangle inequality and the orthogonality and normalization of $\fb_\pca^k$ in
(\ref{eq:fpca}). Finally, by the convergence of $\rb_t$ in \eqref{eq: sym
residual bound} and that of $\xi_t$ in \eqref{eq: sym xi linear AMP} as $t \to
\infty$, we get
\begin{align*}
    \lim_{\min(s,t)\to\infty} \sigma_{st}[k,k'] = (1-\mu_{\pca,k}^2)\cdot
\ind\{k=k'\},
\end{align*}
i.e.\ $\lim_{\min(s,t) \to \infty} \sigma_{st}=\Id-\mu_\pca\mu_\pca^\top$.
\end{proof}

The auxiliary AMP iterations $\bU_{-\tau}^{(\tau)},\bF_{-\tau}^{(\tau)},
\bU_{-\tau+1}^{(\tau)},\ldots$ are defined by this linear AMP algorithm up to
the iterates $\bU_0^{(\tau)}$ and $\bF_0^{(\tau)}$. Thus,
translating Lemma \ref{lemma: sym linear AMP} back to the indexing of this
auxiliary AMP algorithm, and recalling the spectral initializations
$\bF_0=\bF_{\pca}$ and $\bU_0=\bF_{\pca}S^{-1}$,
the lemma implies the following:
\begin{multline}
\label{eq:symlinearAMP}
\lim_{\tau \to \infty} \uplim_{n \to \infty}
\frac{\|\bF_{i}^{(\tau)}-\bF_0\|_\F}{\sqrt{n}}=0, \quad
\lim_{\tau \to \infty} \uplim_{n \to \infty}
\frac{\|\bU_{i}^{(\tau)}-\bU_0\|_\F}{\sqrt{n}}=0
\text{ for all fixed } i \leq 0,\\
\mu_{i}^{(\tau)}=\mu_\pca, \quad
\lim_{\tau \to \infty}
\bSigma_{\all,t}^{(\tau)}[i,j]=\Id-\mu_\pca\mu_\pca^\top \text{ for all fixed }
i,j \leq 0.
\end{multline}

\subsection{Phase II - auxiliary AMP}\label{sec: sym aux proof}
Now we proceed to prove Theorem~\ref{thm:sym} which provides a precise
characterization of the state evolution of the AMP algorithm with spectral
initialization for symmetric matrices.\\

\begin{proof}[Proof of Theorem~\ref{thm:sym}]
We show by induction that the following statements hold a.s.\ for each $t
\geq 0$. In particular, part \ref{enum: sym W2 convergence} is the main result
that we want to prove, and all of the other parts serve as a road map for the
proof. Parts (a)--(d) will apply standard comparison arguments, and the specific
definitions (\ref{eq: sym debias coef spec}) and (\ref{def: sym Sigma_T}) will
be used to verify parts (e) and (f).

\begin{enumerate}[label=(\alph*), ref=(\alph*)]

\item \label{enum: sym uconverge} $\lim_{\tau\to\infty} \uplim_{n\to\infty} \| \Ub_t^{(\tau)} - \Ub_t \|_\F /
\sqrt{n} = 0$ and $\lim_{n\to\infty}\|\Ub_t\|_\F/\sqrt{n}<C_t$ for a constant
$C_t>0$.

\item \label{enum: sym W2 convergence}
$(\Ub_*',\Ub_0,\ldots,\Ub_t,\Fb_0,\ldots,\Fb_{s-1}) \toWtwo
(U_*',U_0,\ldots,U_t,F_0,\ldots,F_{t-1})$ where the limit distribution is
defined by the recursion \eqref{eq:limitDistPCAAMP}.

\item \label{enum: sym PhiDeltaconverge} $\limtau \bPhi_t^{(\tau)} = \bPhi_t$
and $\limtau \bDelta_t^{(\tau)} = \bDelta_t$.

\item\label{enum: sym samplePhiconverge} $\limtaun \|\bphi_t^{(\tau)} -
\bphi_t\| = 0$.

\item \label{enum: sym fconverge} $\lim_{\tau\to\infty} \uplim_{n\to\infty} \| \Fb_t^{(\tau)} - \Fb_t\|_\F /
\sqrt{n} = 0$ and $\lim_{n \to \infty} \|\Fb_s\|_\F/\sqrt{n}<C_t$ for a
constant $C_t>0$.

\item \label{enum: sym stateconverge} $\lim_{\tau\to\infty}\bmu_t^{(\tau)} =
\bmu_t$ and $\limtau \bSigma_t^{(\tau)} = \bSigma_t$.
\end{enumerate}

Denote by $t^{(a)}, t^{(b)},\ldots,t^{(f)}$ the claims of parts (a-f)
{at} iteration $t$. We induct on $t$. For the base case $t=0$, parts (a), (e), and (f) follow from
(\ref{eq:symlinearAMP}) proved in the previous section, and the
initializations $\bU_0=\bF_{\pca}S^{-1}$ and $\bF_0=\bF_{\pca}$. For \ref{enum:
sym PhiDeltaconverge} and \ref{enum: sym samplePhiconverge}, we have
$\bphi_0^{(\tau)}=\bphi_0=\bPhi_0^{(\tau)}=\bPhi_0=0$, while
$\lim_{\tau \to \infty} \bDelta_0^{(\tau)}=\bDelta_0$ follows from part (a).
For \ref{enum: sym W2 convergence}, $(\bU_*',\bU_0) \toWtwo (U_*',U_0)$ follows
from the convergence $(\bU_*',\bU_0^{(\tau)}) \toWtwo (U_*',U_0^{(\tau)})$ for
the auxiliary AMP algorithm, together with (\ref{eq:symlinearAMP}). Thus all statements hold for $t=0$. 

For the induction step, let $t \geq 1$, and assume $s^{(a-f)}$ holds for $0\leq s\leq t-1$. Let us
then prove statements (a--f) for iteration $t$.

\paragraph{Part~\ref{enum: sym uconverge}}
By Assumption~\ref{assump:symAMP}(b), $u_t$ is Lipschitz, so there exists some $L>0$ such that
\begin{align*}
	\frac{\|\Ub_t^{(\tau)} - \Ub_t\|_\F}{\sqrt{n}} =
\frac{\|u_t(\Fb_0^{(\tau)},\ldots,\Fb_{t-1}^{(\tau)}) -
u_t(\Fb_0,\ldots,\Fb_{t-1})\|_\F}{\sqrt{n}} \leq L\cdot \sum_{s=0}^{t-1}
\frac{\|\Fb_s^{(\tau)} - \Fb_s\|_\F}{\sqrt{n}}.
\end{align*}
Then by the induction hypothesis $t-1^{(e)}$, 
$\lim_{\tau\to\infty} \uplim_{n\to\infty} \|\Ub_t^{(\tau)} - \Ub_t\|_\F/\sqrt{n} = 0$.
Moreover, for a constant $C_t>0$,
\begin{align*}
    \lim_{n\to\infty}\frac{\|\Ub_t\|_{\F}^2}{n} = \limn
\frac{\left\|u_t(\Fb_0,\ldots,\Fb_{t-1})\right\|_{\F}^2}{n} =
\EE\left[\|u_t(F_0,\ldots,F_{t-1})\|^2\right]<C_t.
\end{align*}

\paragraph{Part~\ref{enum: sym W2 convergence}}
Let $u_{*,i}\in\RR^{K'}$ denote the $i^{\text{th}}$ row of $\Ub_*'\in\RR^{n\times
K'}$, and similarly for the other matrix variables. Let $g$ be any pseudo-Lipschitz
function, i.e.\ $|g(x)-g(y)| \leq C(1+\|x\|+\|y\|)\|x-y\|$ for a constant $C>0$.
Then there exist some constant $C'>0$ such that 
\begin{align*}
&\quad \frac{1}{n} \sum_{i=1}^n \left|
g(u_{*,i},u^{(\tau)}_{0,i},\ldots,u^{(\tau)}_{t,i},
f^{(\tau)}_{0,i},\ldots,f^{(\tau)}_{t-1,i}) -
g(u_{*,i},u_{0,i},\ldots,u_{t,i},f_{0,i},\ldots,f_{t-1,i})\right| \\
&\leq \frac{C}{n} \sum_{i=1}^n \left(1 + \|u_{*,i}\| + \sum_{s=0}^t (\|u_{s,i}\|
+\|u_{s,i}^{(\tau)}\|)+\sum_{s=0}^{t-1}
(\|f_{s,i}\|+\|f_{s,i}^{(\tau)}\|)\right) \cdot
\left(\sum_{s=0}^t \|u_{s,i} - u_{s,i}^{(\tau)}\| + \sum_{s=0}^{t-1}\|f_{s,i} -
f_{s,i}^{(\tau)}\|\right)\\
&\leq \frac{C'}{n} \left(\sum_{i=1}^n\left(1 + \|u_{*,i}\|^2 + \sum_{s=0}^t
(\|u_{s,i}\|^2 +\|u_{s,i}^{(\tau)}\|^2) +\sum_{s=0}^{t-1} (\|f_{s,i}\|^2
+\|f_{s,i}^{(\tau)}\|^2)\right)\right)^{1/2} \cdot \\
&\hspace{2in}\left(\sum_{i=1}^n \sum_{s=0}^t \|u_{s,i} - u_{s,i}^{(\tau)}\|^2 +
\sum_{s=0}^{t-1}\|f_{s,i} - f_{s,i}^{(\tau)}\|^2\right)^{1/2}\\
&\leq 
C' \underbrace{\left(1+\frac{\|\Ub_*'\|_{\F} + \sum_{s=0}^t
\|\Ub_{s}\|_{\F} + \|\Ub_s^{(\tau)}\|_\F
+\sum_{s=0}^{t-1}
\|\Fb_{s}\|_\F+\|\Fb_{s}^{(\tau)}\|_\F}{\sqrt{n}}\right)}_{\cI_1} \cdot \\
&\hspace{3in}\underbrace{\frac{\sum_{s=0}^t \|\Ub_{s} - \Ub_{s}^{(\tau)}\|_{\F}
+\sum_{s=0}^{t-1} \|\Fb_{s} - \Fb_{s}^{(\tau)}\|_\F}{\sqrt{n}}}_{\cI_2}
\end{align*}
where the last two inequalities follow from Cauchy-Schwartz and the triangle
inequality for the Frobenius norm, respectively.
By $t^{(a)}$ and the induction hypothesis $t-1^{(e)}$,
$\limtaun \cI_1<C_t$ for a constant $C_t>0$,
and $\lim_{\tau\to\infty} \uplim_{n\to\infty} \cI_2=0$. It then follows that
\begin{align*}
\lim_{\tau\to\infty} \uplim_{n\to\infty} \frac{1}{n}
\sum_{i=1}^n
\left|g(u_{*,i},u^{(\tau)}_{0,i},\ldots,u^{(\tau)}_{t,i},f^{(\tau)}_{0,i},\ldots,f^{(\tau)}_{t-1,i})
- g(u_{*,i},u_{0,i},\ldots,u_{t,i},f_{0,i},\ldots,f_{t-1,i}) \right| = 0.
\end{align*}
By the induction hypothesis $t-1^{(f)}$ and Assumption
\ref{assump:symAMP}(b), for all sufficiently large $\tau$ and each $1 \leq s<r
\leq t$, $\partial_s u_r(F_0^{(\tau)},\ldots,F_{r-1}^{(\tau)})$ exists and is
continuous on a set of probability 1 under the law of
$F_0^{(\tau)},\ldots,F_{r-1}^{(\tau)}$ prescribed by Corollary
\ref{cor:auxAMPSE}. Thus, for all sufficiently large $\tau$,
Corollary \ref{cor:auxAMPSE} applies to this
auxiliary AMP algorithm up to iteration $t$, to yield
\begin{align*}
\limn \frac{1}{n} \sum_{i=1}^n
g(u_{*,i},u^{(\tau)}_{0,i},\ldots,u^{(\tau)}_{t,i},f^{(\tau)}_{0,i},\ldots,f^{(\tau)}_{t-1,i})
&=
\bbE\left[g (U_*', U^{(\tau)}_0,\ldots, U^{(\tau)}_t,
F^{(\tau)}_0,\ldots,F^{(\tau)}_{t-1})\right].
\end{align*}
Recalling $(F_0^{(\tau)},\ldots,F_{t-1}^{(\tau)}) \mid U_*' \sim
\cN(\bmu_{t-1}^{(\tau)} U_*', \bSigma_{t-1}^{(\tau)})$ and $(F_0,\ldots,F_{t-1})
\mid U_*' \sim \cN(\bmu_{t-1}U_*',\bSigma_{t-1})$, let us couple these laws
by setting
$(F_0^{(\tau)},\ldots,F_{t-1}^{(\tau)})=\bmu_{t-1}^{(\tau)}U_*' +
(\bSigma_{t-1}^{(\tau)})^{1/2}Z$ and $(F_0,\ldots,F_{t-1})=\bmu_{t-1}U_*' +
\bSigma_{t-1}^{1/2}Z$ for a standard Gaussian vector $Z$ independent of $U_*'$.
Since $\limtau \bmu_{t-1}^{(\tau)}=\bmu_{t-1}$ and $\limtau
\bSigma_{t-1}^{(\tau)}=\bSigma_{t-1}$ by $t-1^{(f)}$,
and $g(\cdot)$ is pseudo-Lipschitz, by the dominated convergence theorem,
\begin{align*}
\limtau
\bbE\left[g (U_*', U^{(\tau)}_0,\ldots, U^{(\tau)}_t, F^{(\tau)}_0,\ldots,F^{(\tau)}_{t-1})\right]
=\bbE\left[g (U_*', U_0,\ldots, U_t, F_0,\ldots,F_{t-1})\right].
\end{align*}
Combining the above three displays, this shows for any pseudo-Lipschitz function $g$ that
\[\limn\frac{1}{n} \sum_{i=1}^n
g(u_{*,i},u_{0,i},\ldots,u_{t,i},f_{0,i},\ldots,f_{t-1,i})
=g(U_*',U_0,\ldots,U_t,F_0,\ldots,F_{t-1}),\]
which implies the desired Wasserstein-2 convergence in part (b).

\paragraph{Part~\ref{enum: sym PhiDeltaconverge}}
Since the upper-left $tK \times tK$ submatrix of $\bPhi_t^{(\tau)}$ is exactly
equal to $\bPhi_{t-1}^{(\tau)}$ and the last column block of
$\bPhi_t^{(\tau)}$ is 0, by the induction hypothesis $t-1^{(c)}$, it suffices to show that the last row
block of $\bPhi_t^{(\tau)}$ converges to that of $\bPhi_t$, and similarly for $\bDelta_t^{(\tau)}$ and $\bDelta_t$. For any $s\in\{0,\ldots,t-1\}$, we have
\begin{align*}
    \bPhi_t^{(\tau)}[t,s] = \EE\left[\partial_s u_t(F_0^{(\tau)},\ldots,F_{t-1}^{(\tau)})\right] \quad\textnormal{and}\quad \bPhi_t[t,s]=\EE\left[\partial_su_t(F_0,\ldots,F_{t-1})\right].
\end{align*}
By Assumption \ref{assump:symAMP}(b), $\partial_su_t$ is bounded and continuous on
a set of probability 1 under $(F_0,\ldots,F_{t-1})$,
so by weak convergence of $(F_0^{(\tau)},\ldots,F_{t-1}^{(\tau)})$ to 
$(F_0,\ldots,F_{t-1})$ under the above coupling,
\begin{align*}
    \limtau \bPhi_t^{(\tau)}[t,s] = \bPhi_t[t,s].
\end{align*}
Similarly, for any $s\in\{0,\ldots,t\}$, $\limtau
\bDelta_t^{(\tau)}[t,s]=\bDelta_t[t,s]$ by the induced coupling
of $(U_0^{(\tau)},\ldots,U_t^{(\tau)})$ and $(U_0,\ldots,U_t)$ and the dominated
convergence theorem.

\paragraph{Part~\ref{enum: sym samplePhiconverge}}
As above, we only need to show that the last row block of $\bphi_t^{(\tau)}$
converges to that of $\bphi_t$. For any $s\in\{0,\ldots,t-1\}$, by the auxiliary
AMP state evolution of Corollary \ref{cor:auxAMPSE},
\begin{align}\label{eq: sym aux samplephi limit}
    \lim_{n\to\infty}\bphi_t^{(\tau)}[t,s] &= \lim_{n\to\infty}\langle\partial_s u_t(\Fb_0^{(\tau)},\ldots,\Fb_{t-1}^{(\tau)})\rangle
    = \EE\left[\partial_s u_t(F_0^{(\tau)},\ldots,F_{t-1}^{(\tau)})\right]
    = \bPhi_t^{(\tau)}[t,s]
\end{align}
where the second equality again follows from $\partial_s u_t$
being bounded and continuous on a set of probability 1 under
$(F_0^{(\tau)},\ldots,F_{t-1}^{(\tau)})$, for sufficiently large $\tau$.
Similarly, by $t^{(b)}$,
\begin{align}
\label{eq: sym samplephi limit}
\lim_{n\to\infty} \bphi_t[t,s] = \EE\left[\partial_su_t(F_0,\ldots,F_{t-1})\right] = \bPhi_t[t,s].
\end{align}
Combining \eqref{eq: sym aux samplephi limit} and \eqref{eq: sym samplephi
limit}, it then follows from $t^{(c)}$ that
\begin{align*}
\limtaun \left(\bphi_t^{(\tau)}[t,s] - \bphi[t,s]\right)=0.
\end{align*}

\paragraph{Part~\ref{enum: sym fconverge}}

We now control the difference between $\Fb_t^{(\tau)}$ and $\Fb_t$. By their
definitions in \eqref{eq:AMPspikedindinit} and~$\eqref{eq:AMPPCAf}$, applying the triangle inequality yields
\begin{align}\label{eq: sym f diff}
	\frac{\|\Fb_t^{(\tau)} - \Fb_t\|_\F}{\sqrt{n}} \leq
\frac{\|\Xb\|\|\Ub_t^{(\tau)} - \Ub_t\|_\F}{\sqrt{n}} +
\frac{1}{\sqrt{n}}\left\|\sum_{i=-\tau}^{t}  \Ub_i^{(\tau)}b_{t,i}^{(\tau)\top}
- \sum_{i=0}^t \Ub_ib_{t,i}^\top \right\|_\F.
\end{align}
Since $\limn \|\Xb\|<C$ a.s.\ for a constant $C>0$, the first term vanishes in the limit $n\to\infty$ by
$t^{(a)}$.
For the second term in \eqref{eq: sym f diff}, we can further decompose and bound it by
\begin{align}\label{eq: step 2 f difference 2}
	\frac{1}{\sqrt{n}}\cdot\left\|\sum_{i=-\tau}^{t} \Ub_i^{(\tau)}
b_{t,i}^{(\tau)\top} - \sum_{i=0}^t \Ub_ib_{t,i}^\top \right\|_\F \leq
\underbrace{\frac{1}{\sqrt{n}}\cdot \left\|\sum_{i=-\tau}^{0} \Ub_i^{(\tau)}
b_{t,i}^{(\tau)\top} - \Ub_0b_{t,0}^\top\right\|_\F}_{\cI_1} +
\underbrace{\sum_{i=1}^t
\frac{1}{\sqrt{n}}\cdot\left\|\Ub_i^{(\tau)}b_{t,i}^{(\tau)\top}  -
\Ub_ib_{t,i}^\top\right\|_\F}_{\cI_2}.
\end{align}

\subparagraph{Term $\cI_1$.} By the triangle inequality,
\begin{align}\label{eq: f diff at 0}
\frac{1}{\sqrt{n}}\cdot\left\|\sum_{i=-\tau}^0
\Ub_i^{(\tau)}b_{t,i}^{(\tau)\top}  - \Ub_0b_{t,0}^\top\right\|_\F
&\leq \underbrace{\left\| \sum_{i=-\tau}^0 b_{t,i}^{(\tau)} - b_{t,0}
\right\|}_{\mathrm{\cI_{1,1}}} \cdot \frac{\|\Ub_0\|_\F}{\sqrt{n}} +
\underbrace{\sum_{i=-\tau}^0 \|b_{t,i}^{(\tau)}\| \cdot \frac{\|\Ub_i^{(\tau)} -
\Ub_0\|_\F}{\sqrt{n}}}_{\cI_{1,2}}.
\end{align}
Since $\|\Ub_0\|_\F/\sqrt{n} = \|\Fb_\pca S^{-1}\|_\F/\sqrt{n}<C$
for a constant $C>0$, it suffices to bound $\cI_{1,1}$ and $\cI_{1,2}$.
To bound $\cI_{1,1}$, recall the definition of debiasing coefficients in
\eqref{eq: sym debias coef spec}:
	\[b_{t,0} = \widetilde\Bb_t[t,0] = \sum_{j=0}^\infty \bphi_t^j[t,
0]\tilde\kappa_{j+1}=\sum_{j=0}^t \bphi_t^j[t,0]\tilde\kappa_{j+1},\]
where the second equality applies $\bphi_t^j=0$ when $j>t$. Recall also the
debiasing coefficients in the auxiliary AMP sequence from~\eqref{eq: sym debias coef aux}:
For every $i\in\{-\tau,\ldots,0\}$, 
\begin{equation}\label{eq:btitau}
b_{t,i}^{(\tau)} = \Bb_t^{(\tau)}[t,i] = \sum_{j=0}^\infty \kappa_{j+1}
\left(\bphi_{\all,t}^{(\tau)}\right)^j[t,i] = \sum_{j=-i}^{-i+t} \kappa_{j+1}
\left(\bphi_t^{(\tau)}\right)^{j+i}[t,0] S^i
\end{equation}
where the last equality follows from~\eqref{eq: sym aux Phi prop} and the fact that $(\bphi_t^{(\tau)})^j=0$ when $j>t$. Therefore,
\begin{align*}
\sum_{i=-\tau}^0 b_{t,i}^{(\tau)} &= \sum_{i=-\tau}^0 \sum_{j=-i}^{-i+t} \kappa_{j+1} \left(\bphi_t^{(\tau)}\right)^{j+i}[t,0] S^{i}\\
&= \sum_{j=0}^t \left(\bphi_t^{(\tau)}\right)^j[t,0] \left(\sum_{i=-\tau}^0 \kappa_{j-i+1}S^i \right)\\
&= \sum_{j=0}^t \left(\bphi_t^{(\tau)}\right)^j[t,0]\tilde\kappa_{j+1} +
\sum_{j=0}^t \left(\bphi_t^{(\tau)}\right)^j[t,0]\sum_{i=\tau+1}^\infty \kappa_{j+i+1} S^{-i}
\end{align*}
and thus
\begin{align*}
	\sum_{i=-\tau}^0 b_{t,i}^{(\tau)} - b_{t,0} &= \sum_{j=0}^t
\rbr{(\bphi_t^{(\tau)})^j[t,0] - \bphi_t^j[t,0]} \tilde\kappa_{j+1} +
\sum_{j=0}^t (\bphi_t^{(\tau)})^j[t,0] \left(\sum_{i=\tau+1}^\infty
\kappa_{j+i+1}S^{-i}\right).
\end{align*}
Taking the limits $\tau \to \infty$ and $n \to \infty$,
the first term converges to 0 by $t^{(d)}$. Since each $u_t$ is Lipschitz, we
have $\|(\bphi_t^{(\tau)})^j[t,0]\|\leq C_t$ for some constant $C_t>0$ and all
$j=0,\ldots,t$, and thus the second term will also converge to 0 as
$\tau\to\infty$ by the absolute convergence of the series defining
$\tkappa_{j+1}$ in (\ref{eq:sym:kappaseries}), as a consequence of
Assumption \ref{assump:symAMP}(c). We therefore conclude
\begin{align}\label{eq: b converge at 0}
	\limtaun \sum_{i=-\tau}^0 b_{t,i}^{(\tau)} - b_{t,0} = 0.
\end{align}

To bound $\cI_{1,2}$, we apply Lemma \ref{lemma:ziplock} with $x^{(\tau)}_i =
\|b_{t,-i}^{(\tau)}\|$ and $y_i^{(\tau)} =
\uplim_{n\to\infty}\|\Ub_{-i}^{(\tau)} - \Ub_0\|_\F/\sqrt{n}$ for every
$i=0,\ldots,\tau$. We need to verify that $\{x_i^{(\tau)}\}_{i=0,\ldots,\tau}$
and  $\{y_i^{(\tau)}\}_{i=0,\ldots,\tau}$ satisfy the conditions of Lemma
\ref{lemma:ziplock}. By (\ref{eq:btitau}) and Assumption \ref{assump:symAMP}(c),
for some constants $C_t, C'_t> 0$ independent of $i$ and $\tau$,
\begin{align*}
\sum_{i=0}^\tau |x_i^{(\tau)}| = \sum_{i=0}^{\tau} \|b_{t,-i}^{(\tau)}\|
&\leq  \sum_{i=0}^{\tau} \sum_{j=i}^{i+t} |\kappa_{j+1}| \cdot \left\|(\bphi_t^{(\tau)})^{j-i}[t,0] S^{-i}\right\|\\
&\leq C_t \sum_{i=0}^{\tau} \sum_{j=0}^{t} \frac{|\kappa_{i+j+1}|}{\min_{k\in\{1,\ldots,K\}}\theta_k^{i}}
\leq C_t \sum_{j=0}^t \sum_{i=0}^\infty
\frac{|\kappa_{i+j+1}|}{\min_{k\in\{1,\ldots,K\}}\theta_k^i}< C_t'.
\end{align*}
Therefore $\{x_i^{(\tau)}\}_{i=0,\ldots,\tau}$ is uniformly bounded. Furthermore,
\[\sum_{i=\cT}^\tau |x_i^{(\tau)}| 
\leq C_t \sum_{j=0}^t \sum_{i=\cT}^\infty
\frac{|\kappa_{i+j+1}|}{\min_{k\in\{1,\ldots,K\}}\theta_k^i},\]
so for any $\epsilon>0$, there exists some $\cT>0$ such that
$\sum_{i=\cT}^{\tau} |x_i^{(\tau)}|<\epsilon$ for all $\tau \geq \cT$.
For $y_i^{(\tau)}$, we have
\[|y_i^{(\tau)}| \leq \limn
\frac{\|\bU_{-i}^{(\tau)}\|_\F}{\sqrt{n}}+\limn \frac{\|\bU_0\|_\F}{\sqrt{n}}.\]
Lemma \ref{lemma: sym linear AMP} implies that the first limit exists, depends
on $(i,\tau)$ only via the difference $\tau-i$, and 
$\lim_{\tau-i \to \infty} \limn
\|\bU_{-i}^{(\tau)}\|_\F/\sqrt{n}<C$ for a constant $C>0$. Then there is
a constant $C'>0$ independent of $(i,\tau)$ for which
$|y_i^{(\tau)}|<C'$, so $\{y_i^{(\tau)}\}_{i=0,\ldots,\tau}$ is uniformly bounded.
By (\ref{eq:symlinearAMP}),
$\lim_{\tau\to\infty} y_i^{(\tau)}=0$ for each $0\leq i\leq\cT$.
Hence, applying Lemma \ref{lemma:ziplock},
\begin{align}\label{eq: f diff at 0 2}
    \lim_{\tau\to\infty} \uplim_{n\to\infty} \sum_{i=-\tau}^0 \|b_{t,i}^{(\tau)}\|\cdot
\frac{\|\Ub_i^{(\tau)}-\Ub_0\|_\F}{\sqrt{n}}=0.
\end{align}
Combining \eqref{eq: f diff at 0}, \eqref{eq: b converge at 0} and \eqref{eq: f diff at 0 2}, we obtain that
\begin{align}\label{eq: sym term1 converge}
    \lim_{\tau\to\infty} \uplim_{n\to\infty} \frac{1}{\sqrt{n}}\cdot \left\|\sum_{i=-\tau}^0 \Ub_i^{(\tau)}
b_{t,i}^{(\tau)\top} - \Ub_0b_{t,0}^\top\right\|_\F = 0.
\end{align}

\subparagraph{Term $\cI_2$.}
The convergence of the term $\cI_2$ in \eqref{eq: step 2 f difference
2} is a more straightforward comparison: By the triangle inequality,
	\[\sum_{i=1}^t\frac{1}{\sqrt{n}} \cdot
\left\|\Ub_i^{(\tau)}b_{t,i}^{(\tau)\top} - \Ub_ib_{t,i}^\top\right\|_\F \leq
\sum_{i=1}^t\|b_{t,i}^{(\tau)}\| \cdot \frac{\|\Ub_i^{(\tau)} -
\Ub_i\|_\F}{\sqrt{n}} + \sum_{i=1}^t\|b_{t,i}^{(\tau)} - b_{t,i}\| \cdot
\frac{\|\Ub_i\|_\F}{\sqrt{n}}.\]
For each $i\in\{0,\ldots,t\}$, by the definition of the debiasing coefficients in \eqref{eq: sym debias coef aux}, we have
\begin{align*}
    b_{t,i}^{(\tau)} &= \sum_{j=0}^\infty \kappa_{j+1}
\left(\bphi_t^{(\tau)}\right)^j[t,i]=
\sum_{j=0}^t \kappa_{j+1}\left(\bphi_t^{(\tau)}\right)^j[t,i].
\end{align*}
Since each $u_t$ is Lipschitz,
this implies $\|b_{t,i}^{(\tau)}\|<C_t$ for a constant $C_t>0$.
Recalling the definitions of $b_{t,i}^{(\tau)}$ in \eqref{eq: sym
debias coef aux} and $b_{t,i}$ in \eqref{eq: sym debias coef spec},
and applying $t^{(d)}$, also $\limtaun b_{t,i}^{(\tau)}-b_{t,i}=0$.
Then combining with $t^{(a)}$,
\begin{align}\label{eq: sym term2 converge}
	\lim_{\tau\to\infty} \uplim_{n\to\infty} \sum_{i=1}^t \frac{1}{\sqrt{n}}\cdot
\left\|\Ub_i^{(\tau)}b_{t,i}^{(\tau)\top} - \Ub_ib_{t,i}^\top\right\|_\F = 0.
\end{align}
Combining \eqref{eq: sym f diff}, \eqref{eq: step 2 f difference 2}, \eqref{eq:
sym term1 converge}, and \eqref{eq: sym term2 converge} shows
$\lim_{\tau\to\infty} \uplim_{n\to\infty} \|\Fb_t^{(\tau)}-\Fb_t\|_\F/\sqrt{n} = 0$ as desired for part (e).

\paragraph{Part~\ref{enum: sym stateconverge}}
Recall that $\bSigma_t^{(\tau)}$ is the lower-right $(t+1)K\times(t+1)K$
submatrix of $\bSigma_{\all,t}^{(\tau)}$ as in \eqref{eq: sym aux
DeltaPhiSigma}. By the block decompositions of $\bPhi_{\all,t}^{(\tau)}$ and
$\bDelta_{\all,t}^{(\tau)}$ in \eqref{eq: sym aux DeltaPhiSigma}, we obtain
\begin{align}\label{eq: sym aux Sigma block_multiplication}
	\bSigma_t^{(\tau)} &= \underbrace{\sum_{j=0}^\infty \kappa_{j+2}
\sum_{i=0}^j (\bPhi_{\all,t}^{(\tau)})_{t-}^i \bDelta_{--}^{(\tau)}
(\bPhi_{\all,t}^{(\tau)\top})_{-t}^{j-i}}_{\widehat\bSigma_t^{(\tau)}} +
\underbrace{\sum_{j=0}^\infty \kappa_{j+2} \sum_{i=0}^j (\bPhi_t^{(\tau)})^i
\bDelta_{t-}^{(\tau)} (\bPhi_{\all,t}^{(\tau)\top})_{-t}^{j-i}}_{\widetilde\bSigma_t^{(\tau)}} \notag\\
	&\qquad + \underbrace{\sum_{j=0}^\infty \kappa_{j+2} \sum_{i=0}^j
(\bPhi_{\all,t}^{(\tau)})_{t-}^i \bDelta_{-t}^{(\tau)} (\bPhi_t^{(\tau)\top})^{j-i}}_{(\widetilde{\bSigma}_t^{(\tau)})^\top} + \underbrace{\sum_{j=0}^\infty \kappa_{j+2} \sum_{i=0}^j (\bPhi_t^{(\tau)})^i \bDelta_t^{(\tau)} (\bPhi_t^{(\tau)\top})^{j-i}}_{\widebar\bSigma_t^{(\tau)}}
\end{align}
where we observe that the third term is the transpose of the second term. To
show $\limtau\bSigma_t^{(\tau)} = \bSigma_t$, 
we recall
$\bDelta_t=\widebar\bDelta_t+\widetilde\bDelta_t+\widetilde\bDelta_t^\top+\widehat\bDelta_t$,
and correspondingly decompose $\bSigma_{t}$ in \eqref{def: sym Sigma_T} as follows:
\begin{align}
    \bSigma_t &= \sum_{j=0}^\infty \bTheta^{(j)}[\bPhi_t,
\kappa_{j+2}\widebar\bDelta_t + \tilde\kappa_{j+2}\odot\widetilde\bDelta_t +
\widetilde\bDelta_t^\top\odot\tilde\kappa_{j+2} +
\hat\kappa_{j+2}\odot\widehat\bDelta_t]\nonumber\\
    &= \underbrace{\sum_{j=0}^\infty\sum_{i=0}^j\bPhi_t^i\left[\hat\kappa_{j+2}\odot \widehat\bDelta_t -
\tilde\kappa_{j+2}\odot\widehat\bDelta_t -
\widehat\bDelta_t^\top\odot\tilde\kappa_{j+2} + \kappa_{j+2}
\widehat\bDelta_t\right] (\bPhi_t^\top)^{j-i}}_{\widehat\bSigma_t} +
\underbrace{\sum_{j=0}^\infty\kappa_{j+2}\sum_{i=0}^j\bPhi_t^i\bDelta_t(\bPhi_t^\top)^{j-i}}_{\widebar\bSigma_t}\notag\\
    &\qquad + \underbrace{\sum_{j=0}^\infty\sum_{i=0}^j
\bPhi_t^i\left[(\tilde\kappa_{j+2}-\kappa_{j+2}\Id)\odot(\widetilde\bDelta_t+\widehat\bDelta_t)
+ (\widetilde\bDelta_t^\top+\widehat\bDelta_t^\top)\odot(\tilde\kappa_{j+2} -
\kappa_{j+2}\Id)\right](\bPhi_t^\top)^{j-i}}_{\widetilde\bSigma_t}.\label{eq: sym Sigma decomp}
\end{align}
We will show that, for any $r,c\in\{0,1,\ldots,t\}$,
\begin{align*}
\lim_{\tau\to\infty}\widehat\bSigma_t^{(\tau)}[r,c] = \widehat\bSigma_t[r,c], \quad
\lim_{\tau\to\infty}\left(\widetilde\bSigma_t^{(\tau)} +
\widetilde\bSigma_t^{(\tau)\top}\right)[r,c] = \widetilde\bSigma_t [r,c], \quad  \lim_{\tau\to\infty} \widebar\bSigma_t^{(\tau)}[r,c] = \widebar\bSigma_t [r,c].
\end{align*}

\subparagraph{Convergence of $\widehat\bSigma_t^{(\tau)}$} We have
\begin{align*}
\widehat\bSigma_t^{(\tau)}[r,c] &= \sum_{j=0}^\infty \kappa_{j+2}
\sum_{i=0}^j \sum_{\alpha,\beta=1}^{\tau} (\bPhi_{\all,t}^{(\tau)})^i[r,-\alpha]
\bDelta^{(\tau)}_{\all, t}[-\alpha,-\beta] (\bPhi_{\all,t}^{(\tau)\top})^{j-i}[-\beta,c]\notag\\
&=\sum_{j=0}^\infty \sum_{i=0}^j \sum_{\alpha=1}^{i\wedge\tau}\sum_{\beta=1}^{(j-i)\wedge\tau} \kappa_{j+2} (\bPhi_t^{(\tau)})^{i-\alpha}[r,0] S^{-\alpha} \bDelta^{(\tau)}_{--}[-\alpha,-\beta]S^{-\beta} (\bPhi_t^{(\tau)\top})^{j-i-\beta}[0,c]
\end{align*}
where the second equality follows from \eqref{eq: sym aux Phi prop}. Let us
write $\bDelta^{(\tau)}_{--}[-\alpha,-\beta]=
(\bDelta^{(\tau)}_{--}[-\alpha,-\beta]-\bDelta_t[0,0])+\bDelta_t[0,0]$, and
introduce $p=i-\alpha$ and $q=j-i-\beta$ to re-index this summation by
$(p,q,\alpha,\beta)$. This yields
\begin{align*}
\widehat\bSigma_t^{(\tau)}[r,c] 
&= \underbrace{\sum_{p,q=0}^\infty (\bPhi_t^{(\tau)})^p[r,0]
\left(\sum_{\alpha,\beta=1}^\tau \kappa_{p+q+\alpha+\beta+2} 
S^{-\alpha}(\bDelta_t^{(\tau)}[-\alpha,-\beta] -
\bDelta_t[0,0])S^{-\beta}\right) (\bPhi_t^{(\tau)\top})^q[0,c]}_{\widehat\cI_1^{(\tau)}}\notag\\
&\qquad + \underbrace{\sum_{p,q=0}^\infty (\bPhi_t^{(\tau)})^p[r,0]
\left(\sum_{\alpha,\beta=1}^\tau \kappa_{p+q+\alpha+\beta+2} 
S^{-\alpha} \bDelta_t[0,0]S^{-\beta}\right) (\bPhi_t^{(\tau)\top})^q[0,c]
}_{\widehat\cI_2^{(\tau)}}.
\end{align*}
Note that the summations over $p,q$ are in fact finite and may be restricted to
$p,q \in [0,t]$, because $(\bPhi_t^{(\tau)})^p=0$ for all $p>t$.

We first show $\widehat\cI_1^{(\tau)}$ vanishes as $\tau\to\infty$.
Since each block
$\|\bPhi_t^{(\tau)}[r,0]\|$ is bounded by a constant and the summation may be
restricted to $p,q \in [0,t]$, there exists some constant $C_t>0$ such that
\begin{align*}
|\widehat\cI_1^{(\tau)}| 
&\leq C_t \max_{\ell=0}^{2t} \sum_{\alpha,\beta=1}^{\tau}
\frac{|\kappa_{\ell+\alpha+\beta+2}|}{\min_{k\in\{1,\ldots,K\}}|\theta_k|^{\alpha+\beta}}\cdot \left\|\bDelta^{(\tau)}_{--}[-\alpha,-\beta] - \bDelta_t[0,0]\right\|\notag\\
&\leq C_t \max_{\ell=0}^{2t} \sum_{i=2}^{2\tau}
\frac{i\cdot|\kappa_{\ell+i+2}|}{\min_{k\in\{1,\ldots,K\}}|\theta_k|^i} \cdot \sup_{\substack{\alpha, \beta\geq 1 \\\alpha + \beta = i}} \left\|\bDelta^{(\tau)}_{--
}[-\alpha,-\beta] - \bDelta_t[0,0]\right\|.
\end{align*}
Fixing any $\ell \in [0,2t]$, we apply Lemma \ref{lemma:ziplock} with
$x_i^{(2\tau)} = i\cdot|\kappa_{\ell+i+2}| / \min_{k\in\{1,\ldots,K\}}|\theta_k|^i$ and
\begin{align*}
    y_i^{(2\tau)} = \sup_{\substack{\alpha,\beta\geq 1 \\ \alpha + \beta = i}}
\left\|\bDelta^{(\tau)}_{--}[-\alpha,-\beta] - \bDelta_t[0,0]\right\|.
\end{align*}
By Assumption \ref{assump:symAMP}(c),
$\{x_i^{(2\tau)}\}_{i=1,\ldots,2\tau}$ satisfies the condition of Lemma
\ref{lemma:ziplock}.
By Lemma \ref{lemma: sym linear AMP}, for any fixed $\alpha,\beta \geq 0$,
\begin{align}
\lim_{\tau\to\infty}\bDelta^{(\tau)}_{--}[-\alpha,-\beta]
&=\lim_{\tau \to \infty} \EE[U_{-\alpha}^{(\tau)}U_{-\beta}^{(\tau)\top}]\notag\\
&=\lim_{\tau \to \infty} S^{-1} \EE[F_{-\alpha-1}^{(\tau)}F_{-\beta-1}^{(\tau)\top}]
S^{-1}\notag\\
&=S^{-1} \cdot\left(
\lim_{\tau\to\infty}\bSigma_{\all,t}^{(\tau)}[-\alpha-1,-\beta-1]+\mu_{-\alpha-1}^{(\tau)}\cdot\mu_{-\beta-1}^{(\tau)\top}\right)
\cdot S^{-1}=S^{-2},\label{eq:Deltatau}
\end{align}
where the last equality applies (\ref{eq:symlinearAMP}). Identifying also
\begin{equation*}
\bDelta_t[0,0]=\EE[U_0U_0^\top]=S^{-2}
\end{equation*}
because $\bU_0=\bF_{\pca}S^{-1} \toWtwo U_0$, this gives
$\lim_{\tau\to\infty} y_i^{(2\tau)}=0$ for every fixed $i$. Observe that
$\EE[\|F_{-\alpha}^{(\tau)}\|]$ depends on $(\alpha,\tau)$ only via the
difference $\tau-\alpha$, and (\ref{eq:symlinearAMP}) implies
$\lim_{\tau-\alpha \to \infty} \EE[\|F_{-\alpha}^{(\tau)}\|]<C$ for a constant
$C>0$. Then $\EE[\|F_{-\alpha}^{(\tau)}\|]<C'$ for a constant $C'>0$ and all
$(\alpha,\tau)$, implying that
$\{y_i^{(2\tau)}\}$ is uniformly bounded. Then it follows from Lemma
\ref{lemma:ziplock} that $\lim_{\tau\to\infty} |\widehat\cI_1^{(\tau)}|=0$.

Next, we show that $\limtau \widehat\cI_2^{(\tau)}=\widehat\bSigma_t[r,c]$.
Taking $\tau\to\infty$, by the convergence of $\bPhi_t^{(\tau)}$ to
$\bPhi_t$ in $t^{(c)}$,
    \[\lim_{\tau\to\infty} \widehat\cI_2^{(\tau)} = \sum_{p,q=0}^\infty
\bPhi_t^p[r,0] \left(\sum_{\alpha,\beta=1}^\infty\kappa_{p+q+\alpha+\beta+2}
S^{-\alpha}\bDelta_t[0,0]S^{-\beta}\right) (\bPhi_t^\top)^q[0,c].\]
Recall that $\bDelta_t[0,0]=S^{-2}$, so this commutes with $S^{-\beta}$. Then,
applying
\begin{align*}
&\sum_{\alpha,\beta=1}^\infty\kappa_{p+q+\alpha+\beta+2}
S^{-\alpha}\bDelta_t[0,0]S^{-\beta}\\
&=\sum_{\alpha,\beta=0}^\infty\kappa_{p+q+\alpha+\beta+2}
S^{-(\alpha+\beta)}\bDelta_t[0,0]
-2\sum_{\alpha=0}^\infty \kappa_{p+q+\alpha+2} S^{-\alpha}\bDelta_t[0,0]
+\kappa_{p+q+2}\bDelta_t[0,0]\\
&=\left(\hkappa_{p+q+2}-2\tkappa_{p+q+2}+\kappa_{p+q+2}\Id\right)\bDelta_t[0,0],
\end{align*}
we identify this limit as $\widehat{\bSigma}_t[r,c]$. So
\begin{align}\label{eq: sym hatSigma converge}
    \lim_{\tau\to\infty} \widehat\bSigma_t^{(\tau)} = \widehat\bSigma_t.
\end{align}

\subparagraph{Convergence of $\widetilde\bSigma_t^{(\tau)}$} Next, we determine
the limit of $\widetilde\bSigma_t^{(\tau)}$. Applying
\eqref{eq: sym aux Phi prop} and re-indexing the summation similarly as above 
by setting $p = i$ and $q = j-i-\beta$,
\begin{align*}
    \widetilde\bSigma_t^{(\tau)}[r,c] &=
\sum_{j=0}^\infty\sum_{i=0}^j\sum_{\alpha=0}^t\sum_{\beta=1}^\tau \kappa_{j+2}
(\bPhi_t^{(\tau)})^i [r,\alpha] \bDelta_{t-}^{(\tau)}[\alpha,-\beta]
(\bPhi_{\all,t}^{(\tau)\top})^{j-i}[-\beta,c]\\
    &= \sum_{j=0}^\infty\sum_{i=0}^j\sum_{\alpha=0}^t\sum_{\beta=1}^\tau \kappa_{j+2} (\bPhi_t^{(\tau)})^i[r,\alpha] \bDelta_{t-}^{(\tau)}[\alpha,-\beta] S^{-\beta} (\bPhi_t^{(\tau)\top})^{j-i-\beta}[0,c]\\
    &= \underbrace{\sum_{p,q=0}^\infty (\bPhi_t^{(\tau)})^p[r,\alpha]
\left(\sum_{\alpha=0}^t\sum_{\beta=1}^\tau \kappa_{p+q+\beta+2} 
 (\bDelta_{t-}^{(\tau)}[\alpha,-\beta] - \bDelta_t[\alpha, 0])
S^{-\beta}\right) (\bPhi_t^{(\tau)\top})^q[0,c]}_{\widetilde\cI_1^{(\tau)}}\\
    &\qquad + \underbrace{\sum_{p,q=0}^\infty(\bPhi_t^{(\tau)})^p[r,\alpha]
\left(\sum_{\alpha=0}^t\sum_{\beta=1}^\tau \kappa_{p+q+\beta+2}
\bDelta_t[\alpha, 0] S^{-\beta}\right)  (\bPhi_t^{(\tau)\top})^q[0,c]}_{\widetilde\cI_2^{(\tau)}}.
\end{align*}
For each fixed $\alpha \in [0,t]$ and $\beta>0$, note that by $t^{(a)}$,
\[\lim_{\tau \to \infty} \EE[\|U_\alpha^{(\tau)}-U_\alpha\|^2]=0\]
and by $t^{(a)}$ and (\ref{eq:Deltatau}),
\begin{align*}
\limtau \EE[\|U_{-\beta}^{(\tau)}-U_0\|^2]
&=\limtau \EE[\|U_{-\beta}^{(\tau)}-U_0^{(\tau)}\|^2]\\
&=\limtau \operatorname{Tr} \Big(\bDelta^{(\tau)}_{\all,t}[-\beta,-\beta]
-2\bDelta^{(\tau)}_{\all, t}[-\beta,0]+\bDelta^{(\tau)}_{\all,t}[0,0]\Big)=0.
\end{align*}
So
\[\limtau \bDelta^{(\tau)}_{t-}[\alpha,-\beta]-\bDelta_t[\alpha,0]
=\limtau \EE[U_\alpha^{(\tau)}U_{-\beta}^{(\tau)\top}]
-\EE[U_\alpha U_0^\top]=0.\]
Then by Lemma \ref{lemma:ziplock} and a similar argument as used previously
for $\widehat\cI_1^{(\tau)}$, this implies
$\lim_{\tau\to\infty}\widetilde\cI_1^{(\tau)}=0$. For
$\widetilde\cI_2^{(\tau)}$,
we take $\tau \to \infty$ and write
$\sum_{\beta=1}^\infty \kappa_{p+q+\beta+2}S^{-\beta}
=\tkappa_{p+q+2}-\kappa_{p+q+2}\Id$. Then
\begin{align*}
    \limtau\widetilde\bSigma_t^{(\tau)}[r,c] &= \sum_{p,q=0}^\infty
\sum_{\alpha=0}^t  \bPhi_t^p[r,\alpha] \Big(\bDelta_t[\alpha,
0] (\tkappa_{p+q+2}-\kappa_{p+q+2}\Id)\Big) (\bPhi_t^\top)^q[0,c]\\
&=\sum_{p,q=0}^\infty\Big[
\bPhi_t^p \Big((\widetilde\bDelta_t^\top+\widehat\bDelta_t^\top)
\odot (\tkappa_{p+q+2}-\kappa_{p+q+2}\Id)\Big) (\bPhi_t^\top)^q\Big][r,c].
\end{align*}
Summing this limit with its transpose, we obtain
\begin{align}\label{eq: sym tildeSigma sum converge}
    \limtau (\widetilde\bSigma_t^{(\tau)}+\widetilde\bSigma_t^{(\tau)\top})
    &= \widetilde\bSigma_t.
\end{align}

\subparagraph{Convergence of $\overline{\bSigma}_t^{(\tau)}$} By the convergence
of $\bPhi_t^{(\tau)},\bDelta_t^{(\tau)}$ to $\bPhi_t,\bDelta_t$ in $t^{(c)}$ and the fact that the summation in $\widebar\bSigma_t^{(\tau)}$ is finite and may be restricted to $j\leq 2t$, we immediately have
\begin{align}\label{eq: sym barSigma converge}
    \limtau \widebar\bSigma_t^{(\tau)} &= 
    \limtau \sum_{j=0}^\infty \kappa_{j+2} \sum_{i=0}^j (\bPhi_t^{(\tau)})^i
\bDelta_t^{(\tau)} (\bPhi_t^{(\tau)\top})^{j-i}\notag\\
    &= \sum_{j=0}^\infty \kappa_{j+2} \sum_{i=0}^j \bPhi_t^i \bDelta_t (\bPhi_t^\top)^{j-i}= \widebar\bSigma_t.
\end{align}
Collecting the decomposition of $\bSigma_t^{(\tau)}$ in \eqref{eq: sym aux Sigma
block_multiplication}, that of $\bSigma_t$ in \eqref{eq: sym Sigma decomp}, and the convergence of each component in \eqref{eq: sym hatSigma converge}, \eqref{eq: sym tildeSigma sum converge} and \eqref{eq: sym barSigma converge},
we obtain $\limtau \bSigma_t^{(\tau)} = \bSigma_t$ as desired for part (f).
\end{proof}

\section{Proof for rectangular matrices}

The proof strategy for the rectangular case is similar as that for the symmetric case, where an auxiliary AMP algorithm serves as a bridge between the actual algorithm and the desired state evolution.

\subsection{State evolution for auxiliary AMP}

In the setting of Theorem \ref{thm:rec:AMPPCA}, consider the auxiliary AMP algorithm
with initialization $\bU_{-\tau}^{(\tau)} \in \RR^{m \times K}$ independent of
$\bW$, having the iterates for $t=-\tau,-\tau+1,-\tau+2,\ldots$
\begin{align}
\begin{aligned}\label{alg:auxAMPrec}
\bG_t^{(\tau)}&=\bX^\top \bU_t^{(\tau)}-\sum_{s=-\tau}^{t-1} \bV_s^{(\tau)}
b_{ts}^{(\tau)\top},
\qquad \bV_t^{(\tau)}=v_t(\bG_{-\tau}^{(\tau)},\ldots,\bG_t^{(\tau)}),\\
\bF_t^{(\tau)}&=\bX\bV_t^{(\tau)}-\sum_{s=-\tau}^t
\bU_t^{(\tau)}a_{ts}^{(\tau)\top},
\qquad \bU_{t+1}^{(\tau)}=u_{t+1}(\bF_{-\tau}^{(\tau)},\ldots,\bF_t^{(\tau)}).
\end{aligned}
\end{align}
For $T \geq 1$, we define
\[\bphi_{\all,T}^{(\tau)}=\Big(\langle \partial_s \bU_r^{(\tau)} \rangle\Big)_{r,s \in
\{-\tau,\ldots,T\}}, \qquad
\bpsi_{\all,T}^{(\tau)}=\Big(\langle \partial_s \bV_r^{(\tau)} \rangle\Big)_{r,s \in
\{-\tau,\ldots,T\}},\]
\[\ba_{\all,T}^{(\tau)}=\sum_{j=0}^\infty
\kappa_{2(j+1)}\bpsi_{\all,T}^{(\tau)}(\bphi_{\all,T}^{(\tau)}
\bpsi_{\all,T}^{(\tau)})^j, \qquad
\bb_{\all,T}^{(\tau)}=\gamma \sum_{j=0}^\infty
\kappa_{2(j+1)}\bphi_{\all,T}^{(\tau)}(\bpsi_{\all,T}^{(\tau)}
\bphi_{\all,T}^{(\tau)})^j\]
where $\{\kappa_{2j}\}_{j \geq 1}$ are the rectangular free cumulants of the
limit singular value distribution $\Lambda$ for $\bW$. We set the above
debiasing coefficients as the blocks
\begin{align}\label{eq: rec debiascoef aux}
    a_{ts}^{(\tau)}=\ba_{\all,T}^{(\tau)}[t,s],\qquad b_{ts}^{(\tau)}=\bb_{\all,T}^{(\tau)}[t,s].
\end{align}

Supposing that $(\bU_*',\bU_{-\tau}^{(\tau)}) \toWtwo (U_*',U_{-\tau}^{(\tau)})$,
we define the following state evolution: Having defined joint laws
$(U_*',U_{-\tau}^{(\tau)},\ldots,U_t^{(\tau)},F_{-\tau}^{(\tau)},\ldots,
F_{t-1}^{(\tau)})$ and $(V_*',V_{-\tau}^{(\tau)},\ldots,V_{t-1}^{(\tau)},
G_{-\tau}^{(\tau)},\ldots,G_{t-1}^{(\tau)})$, we define
\[\bnu_{\all,t}^{(\tau)}=\begin{pmatrix} \nu_{-\tau}^{(\tau)} \\ \vdots \\
\nu_t^{(\tau)}\end{pmatrix},
\qquad \text{ where } \nu_s^{(\tau)}=\EE[U_s{U_*'}^\top] \cdot S'\sqrt{\gamma} 
\in \RR^{K \times K'} \text{ for each } s=-\tau,\ldots,t.\]
We set
\[\bDelta_{\all,t}^{(\tau)}=\Big(\EE[U_rU_s^\top]_{r,s \in \{-\tau,\ldots,t\}}
\Big)_{r,s \in \{-\tau,\ldots,t\}},
\qquad \bPhi_{\all,t}^{(\tau)}=\Big(\EE[\partial_s
u_r(F_1,\ldots,F_{r-1})]\Big)_{r,s \in \{-\tau,\ldots,t\}},\]
\[\bGamma_{\all,t}^{(\tau)}=\Big(\EE[V_rV_s^\top]\Big)_{r,s \in \{-\tau,\ldots,t\}},
\quad \bPsi_{\all,t}^{(\tau)}=\Big(\EE[\partial_s
v_r(G_1,\ldots,G_r)]\Big)_{r,s \in \{-\tau,\ldots,t\}},\]
leaving the last row and column blocks of $\bGamma_{\all,t}^{(\tau)}$ and
$\bPsi_{\all,t}^{(\tau)}$ momentarily undefined, and set
\begin{align*}
\bOmega_{\all,t}^{(\tau)} = \gamma\sum_{j=0}^\infty \bTheta^{(j)}[\bPhi_{\all,t}^{(\tau)}, \bPsi_{\all,t}^{(\tau)}, \kappa_{2(j+1)}\bDelta_{\all,t}^{(\tau)}, \kappa_{2(j+1)}\bGamma_{\all,t}^{(\tau)}],
\end{align*}
where $\bTheta^{(j)}[\cdot,\cdot,\cdot,\cdot]$ is defined in \eqref{eq:rec:Thetaj}.
We define the joint law of $(V_*',V_{-\tau}^{(\tau)},\ldots,V_t^{(\tau)},
G_{-\tau}^{(\tau)},\ldots,G_t^{(\tau)})$ by
\[(G_{-\tau}^{(\tau)},\ldots,G_t^{(\tau)}) \mid V_*'
\sim \cN\Big(\bnu_{\all,t}^{(\tau)}\cdot V_*',\;
\bOmega_{\all,t}^{(\tau)}\Big),\]
\[V_s^{(\tau)}=v_s(G_{-\tau}^{(\tau)},\ldots,G_s^{(\tau)})
\text{ for } s=-\tau,\ldots,t.\]
Recalling $\bXi^{(j)}[\cdot,\cdot,\cdot,\cdot]$ from~\eqref{eq:rec:Xij}, we set
\begin{align}
\bmu_{\all,t}^{(\tau)}&=\begin{pmatrix} \mu_{-\tau}^{(\tau)} \\ \vdots \\
\mu_t^{(\tau)}\end{pmatrix},
\qquad \text{ where } \mu_s^{(\tau)}=\EE[V_s{V_*'}^\top] \cdot S'/\sqrt{\gamma}
\text{ for each } s=-\tau,\ldots,t,\notag\\
\bSigma_{\all,t}^{(\tau)} &= \sum_{j=0}^\infty \bXi^{(j)} [\bPhi_{\all,t}^{(\tau)}, \bPsi_{\all,t}^{(\tau)}, \kappa_{2(j+1)}\bDelta_{\all,t}^{(\tau)}, \kappa_{2(j+1)}\bGamma_{\all,t}^{(\tau)}],
\label{eq: rec aux Sigma}
\end{align}
and define the joint law of $(U_*',U_{-\tau}^{(\tau)},\ldots,U_{t+1}^{(\tau)},
F_{-\tau}^{(\tau)},\ldots,F_t^{(\tau)})$ by
\[(F_{-\tau}^{(\tau)},\ldots,F_t^{(\tau)}) \mid U_*'
\sim \cN\Big(\bmu_{\all,t}^{(\tau)} \cdot U_*',\;
\bSigma_{\all,t}^{(\tau)}\Big),\]
\[U_{s+1}^{(\tau)}=u_{s+1}(F_{-\tau}^{(\tau)},\ldots,F_s^{(\tau)})
\text{ for } s=-\tau,\ldots,t.\]

\begin{corollary}\label{cor: auxAMPrec}
In the rectangular spiked model (\ref{eq:rec:rankkmodel}), suppose Assumption
\ref{assump:rectW} holds for $\Wb$. Suppose the initialization
$\Ub_{-\tau}^{(\tau)}\in\RR^{m\times K}$ is independent of $\Wb$, and $(\bU_*',\bU_{-\tau}^{(\tau)})
\toWtwo (U_*',U_{-\tau}^{(\tau)})$ a.s.\ as $m,n \to \infty$. Suppose 
$u_t(\cdot)$ and $v_t(\cdot)$ are Lipschitz, and $\partial_s u_{t+1}(F_{-\tau}^{(\tau)},\ldots,F_t^{(\tau)})$ and $\partial_s v_t(G_{-\tau}^{(\tau)},\ldots,G_t^{(\tau)})$ exist and are continuous on a set of probability 1 under the above laws of $(F_{-\tau}^{(\tau)},\ldots,F_t^{(\tau)})$ and $(G_{-\tau}^{(\tau)},\ldots,G_t^{(\tau)})$ respectively. Then for any $T \geq 1$, a.s.\ as $m,n \to \infty$,
\begin{align*}
(\bU_*',\bU_{-\tau}^{(\tau)},\ldots,\bU_{T+1}^{(\tau)},
\bF_{-\tau}^{(\tau)},\ldots,\bF_T^{(\tau)}) &\toWtwo
(U_*',U_{-\tau}^{(\tau)},\ldots,U_{T+1}^{(\tau)},F_{-\tau}^{(\tau)},
\ldots,F_T^{(\tau)})\\
(\bV_*',\bV_{-\tau}^{(\tau)},\ldots,\bV_T^{(\tau)},
\bG_{-\tau}^{(\tau)},\ldots,\bG_T^{(\tau)}) &\toWtwo
(V_*',V_{-\tau}^{(\tau)},\ldots,V_T^{(\tau)},G_{-\tau}^{(\tau)},
\ldots,G_T^{(\tau)})
\end{align*}
\end{corollary}

\begin{proof}
The proof is similar to that of Corollary \ref{cor:auxAMPSE}:
Recalling $\bX = \frac{1}{\sqrt{mn}}\bU_*' S' {\bV_*'}^\top + \bW$, we write the
iterations~(\ref{alg:auxAMPrec}) as 
\begin{align*}
\bG_t^{(\tau)}&=\frac{1}{\sqrt{mn}}\bV_*' {S'} {\bU_*'}^\top \bU_t^{(\tau)} + \bW^\top\bU_t^{(\tau)} 
-\sum_{s=-\tau}^{t-1} \bV_s^{(\tau)} b_{ts}^{(\tau)\top},
\\
\bF_t^{(\tau)}&=\frac{1}{\sqrt{mn}}\bU_*' S' {\bV_*'}^\top\bV_t^{(\tau)} + \bW \bV^{(\tau)\top} 
-\sum_{s=-\tau}^t\bU_t^{(\tau)}a_{ts}^{(\tau)\top}.
\end{align*}
Then, approximating $\frac{{S'}{\bU_*'}^\top \bU_t^{(\tau)}}{\sqrt{mn}}$ by 
$\sqrt{\gamma}{S'}\bbE[U_*'{U_t^{(\tau)\top}}]=\nu_t^{(\tau)\top}$ and
$\frac{S'{\bV_*'}^\top \bV_t^{(\tau)}}{\sqrt{mn}}$ by $(1/\sqrt{\gamma})S' \bbE[V_*'V_t^{(\tau)\top}] 
= \mu_t^{(\tau)\top}$,
we consider an alternative AMP sequence with the same initialization 
$\tilde{\bU}_{-\tau} = \bU_{-\tau}^{(\tau)}$ and side information 
$\bU_*'$ and $\bV_*'$, defined by 
\begin{align*}
\tilde{\bZ}_t&=\bW^\top \tilde{\bU}_t
-\sum_{s=-\tau}^{t-1} \tilde{\bV}_sb_{ts}^{\top},
\quad \tilde{\bG}_t = \tilde{\bZ} + \bV_*' \nu_t^{(\tau)\top},
\quad \tilde{\bV}_t=\tilde{v}_t(\tilde\bZ_{-\tau},\ldots,\tilde\bZ_t, \bV_*')
\defeq v_t(\tilde\bG_{-\tau},\ldots,\tilde\bG_t),
\\%
\tilde{\bY}_t&=\bW\tilde{\bV}_t^{(\tau)}
-\sum_{s=-\tau}^t\tilde{\bU}_ta_{ts}^{\top},
\quad \tilde\bF_t = \tilde{\bY}_t + \bU_*' \mu_t^{(\tau)\top},
\quad \tilde\bU_{t+1}=\tilde{u}_{t+1}(\tilde\bY_{-\tau}\ldots,\tilde\bY_t,
\bU_*')
\defeq u_{t+1}(\tilde\bF_{-\tau}\ldots,\tilde\bF_t).
\end{align*}
We may apply Theorem~\ref{thm:rectindinit} to analyze this AMP algorithm,
together with an inductive argument as
in~\cite[Theorem 3.4(a)]{fan2020approximate} to show
\begin{align*}
n^{-1} \|\bG_t^{(\tau)} - \tilde\bG_t\|_F^2\to 0,\;
n^{-1} \|\bV_t^{(\tau)} - \tilde\bV_t\|_F^2\to 0,\;
m^{-1} \|\bF_t^{(\tau)} - \tilde\bF_t\|_F^2 \to 0,\;
m^{-1} \|\bU_{t+1}^{(\tau)} - \tilde\bU_{t+1}\|_F^2 \to 0
\end{align*}
for each fixed $t \geq -\tau$, which implies this corollary.
\end{proof}

Now we specialize the auxiliary AMP algorithm in \eqref{alg:auxAMPrec} to the two-phased algorithm where
\begin{align}\label{eq: rec aux denoiser v}
    v_t(\Gb_{-\tau}^{(\tau)},\ldots,\Gb_t^{(\tau)}) = \begin{cases}
    \Gb_t^{(\tau)}S_v^{-1} & -\tau\leq t\leq 0,\\
    v_t(\Gb_1^{(\tau)},\ldots,\Gb_t^{(\tau)}) & t\geq 1
    \end{cases}
\end{align}
and
\begin{align}\label{eq: rec aux denoiser u}
    u_{t+1}(\Fb_{-\tau}^{(\tau)},\ldots,\Fb_t^{(\tau)}) = \begin{cases}
    \Fb_t^{(\tau)}S_u^{-1} & -\tau\leq t\leq 0,\\
    u_{t+1}(\Fb_0^{(\tau)},\ldots,\Fb_t^{(\tau)}) & t\geq 1.
    \end{cases}
\end{align}
This auxiliary algorithm is initialized at
\begin{align*}
    \Ub_{-\tau}^{(\tau)} = (\ub_{-\tau}^1,\ldots,\ub_{-\tau}^k) \quad \text{with } \ub_{-\tau}^k = (\mu_{\pca,k}\ub_*^k + \sqrt{1-\mu_{\pca,k}^2}\cdot\yb_k) / \theta_{u,k} \text{ for each } k=1,\ldots,K,
\end{align*}
where each $\mu_{\pca,k}$ is defined in Theorem \ref{thm: recSpike}, and $\yb_1,\ldots,\yb_K$ are independent standard Gaussian random vectors also independent of $\Wb$.   

Similar as in the symmetric case, for each $t\geq 1$, we adopt the following block decomposition of $\bphi_{\all,t}^{(\tau)}$:
\begin{align*}
\bphi_{\all,t}^{(\tau)} = \begin{pmatrix}
\bphi_{--}^{(\tau)} & \bphi_{-t}^{(\tau)}\\
\bphi_{t-}^{(\tau)} & \bphi_t^{(\tau)}
\end{pmatrix}, \text{ where } \bphi_{--}^{(\tau)}\in\RR^{\tau K\times \tau K} \text{ and } \bphi_t^{(\tau)}\in\RR^{(t+1)K\times (t+1)K}.
\end{align*}
Due to the linear update rule for the first $\tau$ steps, we have $\bphi_{-t}^{(\tau)} = 0$ and
\begin{align}\label{eq: rec aux phi block}
    \bphi_{--}^{(\tau)} = \begin{pmatrix}
        0 & 0 & \cdots & 0 & 0\\
        S_u^{-1} & 0 & \cdots & 0 & 0\\
        0 & S_u^{-1} & \cdots & 0 & 0\\
        \vdots & \vdots & \ddots & \vdots & \vdots\\
        0 & 0 & \cdots & S_u^{-1} & 0
    \end{pmatrix}\in\RR^{\tau K \times \tau K},\  \bphi_{t-}^{(\tau)} = \begin{pmatrix}
        0 & \cdots & 0 & S_u^{-1}\\
        0 & \cdots & 0 & 0\\
        \vdots & \ddots & \vdots & \vdots\\
        0 & \cdots & 0 & 0
    \end{pmatrix}\in\RR^{(t+1)K\times\tau K}.
\end{align}
Similarly, we write
\begin{align*}
\bpsi_{\all,t}^{(\tau)} = \begin{pmatrix}
\bpsi_{--}^{(\tau)} & \bpsi_{-t}^{(\tau)}\\
\bpsi_{t-}^{(\tau)} & \bpsi_{t}^{(\tau)}
\end{pmatrix}\in\RR^{(\tau+t+1)K\times(\tau+t+1)K}
\end{align*}
while now by \eqref{eq: rec aux denoiser v}, the blocks are given by $\bpsi_{-t}^{(\tau)} = \bpsi_{t-}^{(\tau)}=0$ and
\begin{align}\label{eq: rec aux psi block}
\tbpsi_{--} = \begin{pmatrix}
S_v^{-1} & 0 & \cdots & 0\\
0 & S_v^{-1} & \cdots & 0\\
\vdots & \vdots & \ddots & \vdots\\
0 & 0 & \cdots & S_v^{-1}
\end{pmatrix}\in\RR^{\tau K\times \tau K},\ 
\tbpsi_{t} =
\begin{pmatrix}
S_v^{-1} & 0 & \cdots &  0\\
0 & \langle\partial_1\bV_1\rangle& \cdots &  0\\
\vdots & \vdots & \ddots & \vdots\\
0 & \langle\partial_1\bV_t\rangle & \cdots & \langle\partial_t\bV_t\rangle
\end{pmatrix}\in\RR^{(t+1)K\times(t+1)K}.
\end{align}

We first establish some important properties of $\bpsi_{\all,t}^{(\tau)}$ and $\bphi_{\all,t}^{(\tau)}$ that will be useful in the analysis. By the block decomposition, we have
\begin{align*}
    \bpsi_{\all,t}^{(\tau)}\bphi_{\all,t}^{(\tau)} = \begin{pmatrix}
    \bpsi_{--}^{(\tau)}\bphi_{--}^{(\tau)} & 0\\
    \bpsi_t^{(\tau)}\bphi_{t-}^{(\tau)} & \bpsi_t^{(\tau)}\bphi_t^{(\tau)}
    \end{pmatrix},\quad \bphi_{\all,t}^{(\tau)}\bpsi_{\all,t}^{(\tau)} = \begin{pmatrix}
    \bphi_{--}^{(\tau)}\bpsi_{--}^{(\tau)} & 0\\
    \bphi_{t-}^{(\tau)}\bpsi_{--}^{(\tau)} & \bphi_t^{(\tau)}\bpsi_t^{(\tau)}
    \end{pmatrix}
\end{align*}
where, recalling $S_uS_v=S^2$,
\begin{align*}
    \bpsi_{--}^{(\tau)}\bphi_{--}^{(\tau)} = \bphi_{--}^{(\tau)}\bpsi_{--}^{(\tau)} = {\footnotesize \begin{pmatrix}
        0 & \cdots & 0 & 0\\
        S^{-2} & \cdots & 0 & 0\\
        \vdots & \ddots & \vdots & \vdots\\
        0 & \cdots & S^{-2} & 0
    \end{pmatrix}},\quad
    \bpsi_t^{(\tau)}\bphi_{t-}^{(\tau)} = \bphi_{t-}^{(\tau)}\bpsi_{--}^{(\tau)} =  {\footnotesize \begin{pmatrix}
        0 & \cdots & 0 & S^{-2}\\
        0 & \cdots & 0 & 0\\
        \vdots & \ddots & \vdots & \vdots\\
        0 & \cdots & 0 & 0
    \end{pmatrix}}.
\end{align*}
Thus applying Lemma \ref{lemma: aux algebra fact} with
$\Ab=\bpsi_{\all,t}^{(\tau)}\bphi_{\all,t}^{(\tau)}$ and $\Bb=S^{-2}$ yields,
for any $r\in\{0,\ldots,t\}$ and $c\in\{1,\ldots,\tau\}$,
\begin{align}\label{eq: rec psiphi identity}
    (\bpsi_{\all,t}^{(\tau)}\bphi_{\all,t}^{(\tau)})^j[r,-c] =
(\bpsi_t^{(\tau)}\bphi_t^{(\tau)})^{j-c}[r,0]S^{-2c} & \ind\{1\leq c\leq j\}.
\end{align}
Similarly, we also have
\begin{align}\label{eq: rec phipsi identity}
    (\bphi_{\all,t}^{(\tau)}\bpsi_{\all,t}^{(\tau)})^j[r,-c] = (\bphi_t^{(\tau)}\bpsi_t^{(\tau)})^{j-c}[r,0]S^{-2c} & \ind\{1\leq c\leq j\}
\end{align}
Moreover, we can also establish analogous properties of
$(\bpsi_{\all,t}^{(\tau)}\bphi_{\all,t}^{(\tau)})^j\bpsi_{\all,t}^{(\tau)}$ and
$(\bphi_{\all,t}^{(\tau)}\bpsi_{\all,t}^{(\tau)})^j\bphi_{\all,t}^{(\tau)}$:
\begin{align}\label{eq: rec psiphipsi identity}
    (\bpsi_{\all,t}^{(\tau)}\bphi_{\all,t}^{(\tau)})^j\bpsi_{\all,t}^{(\tau)}[r,-c] &= \sum_{i=-\tau}^t (\bpsi_{\all,t}^{(\tau)}\bphi_{\all,t}^{(\tau)})^j[r,i]\bpsi_{\all,t}^{(\tau)}[i,-c]\notag\\
    &= (\bpsi_{\all,t}^{(\tau)}\bphi_{\all,t}^{(\tau)})^j[r,-c]\bpsi_{\all,t}^{(\tau)}[-c,-c]\notag\\
    &= (\bpsi_t^{(\tau)}\bphi_t^{(\tau)})^{j-c}[r,0]S^{-2c}S_v^{-1}  \ind\{1\leq c\leq j\}\notag\\
    &= (\bpsi_t^{(\tau)}\bphi_t^{(\tau)})^{j-c}\bpsi_t^{(\tau)}[r,0]S^{-2c} \ind\{1\leq c\leq j\}
\end{align}
where the second and fourth equalities follow from \eqref{eq: rec aux psi
block}, and the third equality is due to \eqref{eq: rec psiphi identity}.
Similarly, 
\begin{align}\label{eq: rec phipsiphi identity}
    (\bphi_{\all,t}^{(\tau)}\bpsi_{\all,t}^{(\tau)})^j\bphi_{\all,t}^{(\tau)}[r,-c] &= \sum_{i=-\tau}^t (\bphi_{\all,t}^{(\tau)}\bpsi_{\all,t}^{(\tau)})^j[r,i]\bphi_{\all,t}^{(\tau)}[i,-c]\notag\\
    &= (\bphi_{\all,t}^{(\tau)}\bpsi_{\all,t}^{(\tau)})^j[r,-c+1]\bphi_{\all,t}^{(\tau)}[-c+1,c]\notag\\
    &=(\bphi_t^{(\tau)}\bpsi_t^{(\tau)})^{j-c+1}[r,0]S^{-2(c-1)} S_u^{-1} \ind\{1\leq c\leq j\}\notag\\
    &=(\bphi_t^{(\tau)}\bpsi_t^{(\tau)})^{j-c}\bphi_t^{(\tau)}[r,0]S_v^{-1}S^{-2(c-1)}S_u^{-1} \ind\{1\leq c\leq j\}\notag\\
    &= (\bphi_t^{(\tau)}\bpsi_t^{(\tau)})^{j-c}\bphi_t^{(\tau)}[r,0]S^{-2c} \ind\{1\leq c\leq j\}
\end{align}
where the second and fourth equalities follow from \eqref{eq: rec aux phi block}
and \eqref{eq: rec aux psi block}, and the third equality is due to \eqref{eq: rec phipsi identity}.
All the above identities hold for $\taubPsi_{\all,t}$ and $\taubPhi_{\all,t}$ as well.

Finally, for the state evolution, we decompose
\begin{align*}
&\bnu_{\all,t}^{(\tau)} = \begin{pmatrix} \bnu_-^{(\tau)} \\ \bnu_t^{(\tau)}
\end{pmatrix}, \qquad\bmu_{\all,t}^{(\tau)} = \begin{pmatrix} \bmu_-^{(\tau)} \\
\bmu_t^{(\tau)} \end{pmatrix},\\
&\bDelta_{\all,t}^{(\tau)} = \begin{pmatrix}
\bDelta_{--}^{(\tau)} & \bDelta_{-t}^{(\tau)}\\
\bDelta_{t-}^{(\tau)} & \bDelta_t^{(\tau)}
\end{pmatrix},\quad \bPhi_{\all,t}^{(\tau)} = \begin{pmatrix}
\bPhi_{--}^{(\tau)} & \bPhi_{-t}^{(\tau)}\\
\bPhi_{t-}^{(\tau)} & \bPhi_t^{(\tau)}
\end{pmatrix}, \quad
\bSigma_{\all,t}^{(\tau)} = \begin{pmatrix}
\bSigma_{--}^{(\tau)} & \bSigma_{-t}^{(\tau)}\\
\bSigma_{t-}^{(\tau)} & \bSigma_t^{(\tau)}
\end{pmatrix},\\
&\bGamma_{\all,t}^{(\tau)} = \begin{pmatrix}
\bGamma_{--}^{(\tau)} & \bGamma_{-t}^{(\tau)}\\
\bGamma_{t-}^{(\tau)} & \bGamma_t^{(\tau)}
\end{pmatrix}, \quad \bPsi_{\all,t}^{(\tau)} = \begin{pmatrix}
\bPsi_{--}^{(\tau)} & \bPsi_{-t}^{(\tau)}\\
\bPsi_{t-}^{(\tau)} & \bPsi_t^{(\tau)}
\end{pmatrix}, \quad \bOmega_{\all,t}^{(\tau)} = \begin{pmatrix}
\bOmega_{--}^{(\tau)} & \bOmega_{-t}^{(\tau)}\\
\bOmega_{t-}^{(\tau)} & \bOmega_t^{(\tau)}
\end{pmatrix}.
\end{align*}

\subsection{Phase I: Linear AMP for rectangular matrices}
As in the symmetric case, we first establish the convergence of the iterates and the associated state evolution of the first $\tau$ steps of the auxiliary AMP algorithm, in the limit as $\tau\to\infty$. We again reindex these iterates as $1,2,\ldots,\tau$.

Specifically, let $\bU_1 = (\ub_1^1,\ldots,\ub_1^K)$ with each $\ub_1^k =
(\mu_{\pca,k} \bu_*^k + \sqrt{1 - \mu_{\pca,k}^2} \cdot \yb_k)/\theta_{u,k}$. We write $\Fb_0=\Ub_1 S_u$. Then the first $\tau$ iterates of the above auxiliary AMP algorithm have the structure of the following linear AMP:
\begin{align}\label{eq:rec:linearamp_algo}
\begin{aligned}
\bG_t &= \bX^\top \Fb_{t-1}S_u^{-1} - \gamma \sum_{j=1}^{t-1}  
\kappa_{2j} \bG_{t-j} S^{-2j},\\
\Fb_t &= \bX \bG_{t} S_v^{-1} 
- \sum_{j=1}^{t} \kappa_{2j}
\Fb_{t-j} S^{-2j}.
\end{aligned}
\end{align}
Up to iterate $\tau$, let $\bmu_\tau = (\mu_t)_{1\leq t\leq\tau}, \bnu_{\tau} = (\nu_t)_{1\leq t\leq\tau}, \bSigma_\tau = (\sigma_{st})_{1\leq s,t\leq\tau}$ and $\bOmega_\tau = (\omega_{st})_{1\leq s,t\leq\tau}$ be the parameters of the state evolution describing this linear AMP, where each $\mu_t,\nu_t\in\RR^{K\times K'}$ and $\sigma_{st},\omega_{st}\in\RR^{K\times K}$. Then it follows from Corollary \ref{cor: auxAMPrec} that
\begin{align}\label{eq: rec linear AMP state evolution}
\begin{aligned}
    (\bF_1, \ldots, \bF_\tau) &\toWtwo \bmu_\tau \cdot U_*' + (Y_1,\ldots,
Y_\tau) \textnormal{ with } (Y_1,\ldots,Y_\tau) \sim \calN(0,\bSigma_\tau)
\independent U_*' ,\\
    (\bG_1, \ldots, \bG_\tau) &\toWtwo \bnu_\tau \cdot V_*' + (Z_1,\ldots,
Z_\tau) \textnormal{ with } (Z_1,\ldots,Z_\tau) \sim \calN(0,\bOmega_\tau)
\independent V_*'.
\end{aligned}
\end{align}
Recall
\[\mu_\pca=(\mu_{\pca,1},\ldots,\mu_{\pca,K}) \in \RR^{K \times K'},
\qquad \nu_\pca=(\nu_{\pca,1},\ldots,\nu_{\pca,K}) \in \RR^{K \times K'}.\]

\begin{lemma}\label{lem:rec:linearamp}
Under Assumptions~\ref{assump:rectW} and \ref{assump:rec:AMP}(a) and (c), the following holds for the linear AMP algorithm~\eqref{eq:rec:linearamp_algo}:
\begin{enumerate}[label=(\alph*)]
\item $\lim_{t\to\infty}\uplim_{m,n\to\infty} \|\Fb_t - \Fb_\pca\|_\F/\sqrt{m} =
0$ and $\lim_{t\to\infty}\uplim_{m,n\to\infty} \|\Gb_t - \Gb_\pca\|_\F/\sqrt{n} = 0$ a.s.
\item The state evolution satisfies $\mu_t=\mu_\pca$ and $\nu_t=\nu_\pca$ for
every $t\geq 1$, and $\lim_{\min(s,t)\to\infty} \sigma_{st} =
\Id-\mu_\pca\mu_\pca^\top$ and $\lim_{\min(s,t)\to\infty} \omega_{st} =
\Id-\nu_\pca\nu_\pca^\top$.
\end{enumerate}
\end{lemma}

\begin{proof}
Recall the $K'$ largest sample singular values of $\Xb$ in \eqref{rec sigval}
and the associated singular vectors in \eqref{rec sigvec}. We denote the
remaining singular values and vectors as $\lambda_i(\Xb), \fb_\pca^i$ and
$\gb_\pca^i$ for $i=K'+1,\ldots,m$ in any order, with the same normalization
that $\|\fb_\pca^i\|=\sqrt{m}$ and $\|\gb_\pca^i\|=\sqrt{n}$. Thus
$\Xb\gb_\pca^i/\sqrt{n} = \lambda_i(\Xb)\fb_\pca^i/\sqrt{m}$ and
${\fb_\pca^i}^\top\Xb/\sqrt{m}=\lambda_i(\Xb)\gb_\pca^i/\sqrt{n}$ for
$i=1,\ldots,m$. Let
\[\cS=\Big\{i \in \{1,\ldots,K'\}:\theta_k>(D(\lambda_+))^{-1/2}\Big\}\]
be the set of ``super-critical'' signal values as characterized by Theorem
\ref{thm: recSpike}.
Denote $\|\Lambda\|_\infty = \lambda_+$, fix a small constant $\delta>0$, and define the event
\begin{align*}
    \cE_{m,n} = \{\lambda_i(\Xb)\leq \|\Lambda\|_\infty + \delta \text{
for all } i \notin \cS\}.
\end{align*}
Then $\cE_{m,n}$ occurs almost surely for all large $m$ and $n$, where this
bound for $i \in \{1,\ldots,K'\} \setminus \cS$ follows from Theorem \ref{thm: recSpike}, and that
for $i=K'+1,\ldots,n$ follows from Assumption \ref{assump:rectW}(c) and
Weyl's singular value interlacing inequality.

Let $\fb_t^k$ and $\gb_t^k$ be the $k^\text{th}$ columns of the linear AMP iterates $\Fb_t$ and $\Gb_t$. For part (a), it suffices to show that
\begin{align*}
    \lim_{t\to\infty}\uplim_{m,n\to\infty} \frac{\|\fb_t^k-\fb_\pca^k\|}{\sqrt{m}} + \frac{\|\gb_t^k-\gb_\pca^k\|}{\sqrt{n}} = 0 
\end{align*}
for each $k=1,\ldots,K$. Fixing any such $k$, by the definition of linear AMP in \eqref{eq:rec:linearamp_algo}, we have
\begin{align}\label{eq: rec fg_t linear AMP}
    \gb_t^k = \frac{1}{\theta_{u,k}}\cdot \Xb^\top \fb_{t-1}^k - \gamma\sum_{j=1}^{t-1} \frac{\kappa_{2j}}{\theta_k^{2j}}\gb_{t-j}^k,\qquad \fb_t^k = \frac{1}{\theta_{v,k}}\cdot \Xb\gb_t^k - \sum_{j=1}^t \frac{\kappa_{2j}}{\theta_k^{2j}}\fb_{t-j}^k.
\end{align}

We first show that the component of $\fb_t^k$ orthogonal to $\fb_\pca^k$ and
that of $\gb_t^k$ orthogonal to $\gb_\pca^k$ vanish a.s.\ in the limits
$t\to\infty$ and $m,n\to\infty$. 
Note that this linear AMP update ensures that $\bg_t^k$ is 
always in the span of $\{\bg_\pca^i\}_{i=1,\ldots,m}$.
For each $t\geq 0$ and $i\in\{1,\ldots,m\}$, define
$\ell_t^{k,i} = (\fb_\pca^i)^\top\fb_t^k/m$. For each $t \geq 1$ and
$i \in \{1,\ldots,m\}$, define
$r_t^{k,i} = {(\gb_\pca^i)}^\top\gb_t^k/n$. We further
extend these definitions by setting
$r_0^{k,i}=0$ for $t=0$ and all $i\in\{1,\ldots,m\}$.
Then applying \eqref{eq: rec fg_t linear AMP}, for all $i\in\{1,\ldots,m\}$,
\begin{align}\label{eq: rec linear recursion left}
\begin{aligned}
    \ell_t^{k,i} &= \frac{1}{\theta_{v,k}}\cdot \frac{{\fb_\pca^i}^\top\Xb\gb_t^k}{m} - \sum_{j=1}^t \frac{\kappa_{2j}}{\theta_k^{2j}} \cdot \frac{{\fb_\pca^i}^\top\fb_{t-j}^k}{m}\\
    &= \frac{\lambda_i(\Xb)}{\theta_{v,k}} \cdot \frac{{\gb_\pca^i}^\top\gb_t^k}{\sqrt{mn}} - \sum_{j=1}^t \frac{\kappa_{2j}}{\theta_k^{2j}} \ell_{t-j}^{k,i}\\
    &= \frac{\lambda_i(\Xb)}{\theta_{v,k}\sqrt{\gamma}} r_t^{k,i} -
\sum_{j=1}^{t} \frac{\kappa_{2j}}{\theta_k^{2j}} \ell_{t-j}^{k,i}.
\end{aligned}
\end{align}
Similarly, for all $i \in \{1,\ldots,m\}$,
\begin{align}\label{eq: rec linear recursion right}
\begin{aligned}
    r_t^{k,i} &= \frac{1}{\theta_{u,k}} \cdot \frac{{\gb_\pca^i}^\top\Xb^\top\fb_{t-1}^k}{n} - \gamma\sum_{j=1}^{t-1} \frac{\kappa_{2j}}{\theta_k^{2j}}\cdot \frac{{\gb_\pca^i}^\top \gb_{t-j}^k}{n}\\
    &= \frac{\lambda_i(\Xb)}{\theta_{u,k}} \cdot \frac{{\fb_\pca^i}^\top\fb_{t-1}^k}{\sqrt{mn}} - \gamma\sum_{j=1}^{t}\frac{\kappa_{2j}}{\theta_k^{2j}} r_{t-j}^{k,i}\\
    &= \frac{\lambda_i(\Xb)\sqrt{\gamma}}{\theta_{u,k}} \ell_{t-1}^{k,i} - \gamma\sum_{j=1}^{t}\frac{\kappa_{2j}}{\theta_k^{2j}} r_{t-j}^{k,i}
    \end{aligned}
\end{align}
where the second equality applies our convention $r_0^{k,i}\equiv0$. First, for any
$i\in \cS\setminus\{k\}$, by the initialization $\fb_0^k =
\theta_{u,k}\ub_1^k = \mu_{\pca,k}\ub_*^k + \sqrt{1-\mu_{\pca,k}^2}\cdot\yb_k$, we have
\begin{align*}
    \ell_0^{k,i} = \frac{\mu_{\pca,k}}{m} {\fb_\pca^i}^\top \ub_*^k +
\frac{\sqrt{1-\mu_{\pca,k}^2}}{m}{\fb_\pca^i}^\top \yb_k \to 0
\end{align*}
a.s.\ as $m,n\to\infty$, where the first term converges to 0 by Theorem \ref{thm:
recSpike}, and the second term converges to 0 since
${\fb_\pca^i}^\top\yb_k/m\sim\cN(0,1/m)$. Thus, it follows from the recursions \eqref{eq: rec linear recursion left} and \eqref{eq: rec linear recursion right} that
\begin{align}\label{eq: rec residual spike}
    \lim_{m,n\to\infty} \ell_t^{k,i} = \lim_{m,n\to\infty} r_t^{k,i} = 0 \text{
a.s. for each fixed } t\geq 0 \text{ and } i\in \cS\setminus\{k\}.
\end{align}

Next, for each $i \notin \cS$, consider a space $\cX$ of bounded
infinite-dimensional vectors with elements in $[0,\infty)$. For each $t\geq 1$, we define two elements $\brho_t^{k,i}, \bvarphi_t^{k,i}\in\cX$ as
\begin{align*}
    \brho_t^{k,i} = (\ell_t^{k,i},r_t^{k,i},\ldots,\ell_0^{k,i},r_0^{k,i},0,\ldots), \quad \bvarphi_t^{k,i} = (r_t^{k,i},\ell_{t-1}^{k,i},r_{t-1}^{k,i},\ldots,\ell_0^{k,i},r_0^{k,i},0,\ldots).
\end{align*}
Let $\iota\in(0,1)$ be chosen as in Assumption \ref{assump:rec:AMP}(c), and let us consider a norm $\|\cdot\|$ on $\cX$ defined by
\begin{align*}
    \|(x_0,x_{-1},x_{-2},\ldots)\| = \sup_{k\geq 0} |x_{-k}| \cdot \iota^k.
\end{align*}
Consider a map $g:\cX\to\cX$ defined as $g(\brho_{t-1}^{k,i})=\bvarphi_t^{k,i}$ and another map $h:\cX\to\cX$ given by $h(\bvarphi_t^{k,i}) = \brho_t^{k,i}$. Then we have $\brho_t^{k,i} = (h\circ g)(\brho_{t-1}^{k,i})$. We verify that both $g$ and $h$ are contractive with respect to the norm $\|\cdot\|$ on $\cX$. Let $\{(\brho_t,\bvarphi_t))\}_{t\geq 1}$ and $\{(\tilde\brho_t,\tilde\bvarphi_t)\}_{t\geq 1}$ be two sequences given by
\begin{align*}
    \begin{cases}
    \brho_t = (\ell_t,r_t,\ldots,\ell_0,r_0,0,\ldots)\\
    \bvarphi_t = (r_t,\ell_{t-1},r_{t-1},\ldots,\ell_0,r_0,0,\ldots)
    \end{cases},\qquad
    \begin{cases}
    \tilde\brho_t = (\tilde\ell_t,\tilde r_t,\ldots,\tilde\ell_0, \tilde r_0,0,\ldots)\\
    \tilde\bvarphi_t = (\tilde r_t,\tilde\ell_{t-1},\tilde r_{t-1},\ldots,\tilde\ell_0,\tilde r_0,0,\ldots)
    \end{cases}
\end{align*}
where both $\{(\ell_t,r_t)\}_{t\geq 1}$ and $\{(\tilde\ell_t,\tilde r_t)\}_{t\geq 1}$ satisfy the same recursions as in \eqref{eq: rec linear recursion left} and \eqref{eq: rec linear recursion right}. Note that
\begin{align}\label{eq: rec g contraction 1}
    \|g(\brho_{t-1}) - g(\tilde\brho_{t-1})\| = \|\bvarphi_t - \tilde\bvarphi_t\| &= \max\left\{|r_t-\tilde r_t|, \max_{1\leq j\leq t} |\ell_{t-j}-\tilde\ell_{t-j}|\iota^{2j-1}, \max_{1\leq j\leq t} |r_{t-j} - \tilde r_{t-j}|\iota^{2j}\right\}\notag\\
    &=\max\{|r_t-\tilde r_t|, \iota\cdot \|\brho_{t-1} - \tilde\brho_{t-1}\|\}.
\end{align}
Then we need to control $|r_t-\tilde r_t|$. It follows from \eqref{eq: rec linear recursion right} that
\begin{align}\label{eq: rec g contraction 2}
    |r_t-\tilde r_t| &= \left| \frac{\lambda_i(\Xb)\sqrt{\gamma}}{\theta_{u,k}}(\ell_{t-1}-\tilde\ell_{t-1}) - \gamma \sum_{j=1}^t\frac{\kappa_{2j}}{\theta_k^{2j}} (r_{t-j} - \tilde r_{t-j}) \right|\notag\\
    &\leq \frac{\lambda_i(\Xb)\sqrt{\gamma}}{|\theta_{u,k}|} |\ell_{t-1}-\tilde\ell_{t-1}| + \gamma\sum_{j=1}^t \frac{|\kappa_{2j}|}{\theta_k^{2j}\iota^{2j-1}} |r_{t-j}-\tilde r_{t-j}|\iota^{2j-1}\notag\\
    &\leq \left(\frac{\lambda_i(\Xb)\sqrt{\gamma}}{|\theta_{u,k}|} + \gamma\sum_{j=1}^t\frac{|\kappa_{2j}|}{\theta_k^{2j}\iota^{2j-1}}\right) \cdot \max\left\{|\ell_{t-1}-\tilde\ell_{t-1}|,\max_{1\leq j\leq t} |r_{t-j}-\tilde r_{t-j}|\iota^{2j-1} \right\}\notag\\
    &\leq \underbrace{\left(\frac{(\|\Lambda\|_\infty+\delta)\sqrt{\gamma}}{|\theta_{u,k}|} + \gamma\sum_{j=1}^{\infty} \frac{|\kappa_{2j}|}{\theta_k^{2j}\iota^{2j-1}}\right)}_{\eta_1} \cdot \|\brho_{t-1}-\tilde\brho_{t-1}\|
\end{align}
where the last inequality holds on the above event $\cE_{m,n}$. For sufficiently
small $\delta$, we have $\eta_1\in(0,1)$ by Assumption
\ref{assump:rec:AMP}(c). Similarly,
\begin{align}\label{eq: rec h contraction 1}
    \|h(\bvarphi_t) - h(\tilde\bvarphi_t)\| = \|\brho_t - \tilde\brho_t\| &= \max\left\{|\ell_t-\tilde\ell_t|, \max_{1\leq j\leq t}|\ell_{t-j}-\tilde\ell_{t-j}|\iota^{2j}, \max_{0\leq j\leq t} |r_{t-j}-\tilde r_{t-j}| \iota^{2j+1} \right\}\notag\\
    &= \max\{|\ell_t-\tilde\ell_t|, \iota \cdot
\|\bvarphi_t-\tilde\bvarphi_t\|\}.
\end{align}
By \eqref{eq: rec linear recursion left}, we have on the event $\cE_{m,n}$ that
\begin{align}\label{eq: rec h contraction 2}
    |\ell_t-\tilde\ell_t| &= \left| \frac{\lambda_i(\Xb)}{\theta_{v,k}\sqrt{\gamma}}(r_t-\tilde r_t) - \sum_{j=1}^{t}\frac{\kappa_{2j}}{\theta_k^{2j}}(\ell_{t-j}-\tilde\ell_{t-j}) \right|\notag\\
    &\leq \frac{\lambda_i(\Xb)}{|\theta_{v,k}|\sqrt{\gamma}}|r_t-\tilde r_t| + \sum_{j=1}^t \frac{|\kappa_{2j}|}{\theta_k^{2j}\iota^{2j-1}} |\ell_{t-j}-\tilde\ell_{t-j}|\iota^{2j-1}\notag\\
    &\leq \underbrace{\left(
\frac{\|\Lambda\|_\infty+\delta}{|\theta_{v,k}|\sqrt{\gamma}} +
\sum_{j=1}^\infty \frac{|\kappa_{2j}|}{\theta_k^{2j}\iota^{2j-1}}
\right)}_{\eta_2} \cdot \|\bvarphi_t-\tilde\bvarphi_t\|.
\end{align}
We also have
$\eta_2\in(0,1)$ for sufficiently small $\delta$ by Assumption
\ref{assump:rec:AMP}(c). Combining \eqref{eq: rec g contraction 1}, \eqref{eq: rec g contraction 2}, \eqref{eq: rec h contraction 1} and \eqref{eq: rec h contraction 2}, we get
\begin{align*}
    \|(h\circ g)(\brho_{t-1}) - (h\circ g)(\tilde \brho_{t-1})\| \leq \underbrace{\max\{\eta_1,\iota\}\cdot\max\{\eta_2,\iota\}}_{\rho} \cdot \|\brho_{t-1}-\tilde\brho_{t-1}\|.
\end{align*}
Therefore, $h\circ g$ is a $\rho$-contraction for some $\rho\in(0,1)$. Applying this property to $\{\brho_t^{k,i}\}_{t\geq 1}$ yields
\begin{align}\label{eq: rec residual contraction}
    |\ell_t^{k,i}| + |r_t^{k,i}| &\leq \frac{1}{\iota} \|(\ell_t^{k,i},r_t^{k,i},\ldots,\ell_0^{k,i},r_0^{k,i},0\ldots) - (0, 0,\ldots)\|\notag\\
    &\leq \frac{\rho^t}{\iota} \|(\ell_0^{k,i},r_0^{k,i},0,\ldots) - (0,0,\ldots)\|\notag\\
    &= \frac{\rho^t}{\iota} (|\ell_0^{k,i}| + |r_0^{k,i}|).
\end{align}
This holds simultaneously for all $i\notin \cS$ on the event $\cE_{m,n}$.

Now write $\fb_t^k = \xi_t^k\fb_\pca^k + \bell_t^k$ and $\gb_t^k =
\zeta_t^k\gb_\pca^k + \rb_t^k$ where $\bell_t^k\perp\fb_\pca^k$ and
$\rb_t^k\perp\gb_\pca^k$. Recalling that $\gb_t^k$ belongs to the span of
$\{\gb_\pca^i\}_{i=1,\ldots,m}$, we can expand $\bell_t^k$ and $\rb_t^k$ as
\begin{align*}
    \bell_t^k = \sum_{i=1,i\neq k}^m \frac{{\fb_\pca^i}^\top\fb_t^k}{m}\fb_\pca^i = \sum_{i=1,i\neq k}^m \ell_t^{k,i} \cdot \fb_\pca^i \quad \text{and}\quad \rb_t^k = \sum_{i=1,i\neq k}^{m} \frac{{\gb_\pca^i}^\top \gb_t^k}{n} \gb_\pca^i = \sum_{i=1,i\neq k}^{m} r_t^{k,i} \cdot \gb_\pca^i.
\end{align*}
Thus, we further have
\begin{align*}
    \frac{\|\bell_t^k\|^2}{m} + \frac{\|\rb_t^k\|^2}{n} = \underbrace{\sum_{i
\in \cS \setminus \{k\}} \left((\ell_t^{k,i})^2 + (r_t^{k,i})^2\right)}_{\cI_1}
+ \underbrace{\sum_{i \notin \cS} \left((\ell_t^{k,i})^2 + (r_t^{k,i})^2\right)}_{\cI_2}.
\end{align*}
By \eqref{eq: rec residual spike}, we have $\lim_{n\to\infty}\cI_1=0$. For
$\cI_2$, it follows from \eqref{eq: rec residual contraction} and the definition
$r_0^{k,i} \equiv 0$ that
\begin{align*}
    \cI_2 &\leq \frac{2\rho^{2t}}{\iota^2} \sum_{i \notin \cS} \left( (\ell_0^{k,i})^2
+  (r_0^{k,i})^2\right) \leq \frac{2\rho^{2t}}{\iota^2} \cdot
\frac{\|\bell_0^k\|^2}{m} \leq \frac{2\rho^{2t}}{\iota^2} \cdot \frac{\|\fb_0^k\|^2}{m}.
\end{align*}
By the initialization $\fb_0^k = \mu_{\pca,k}\ub_*^k +
\sqrt{1-\mu_{\pca,k}^2}\cdot\yb_k$, we have $\lim_{m,n\to\infty} \|\fb_0^k\|^2/m=1$. Therefore, it follows that
\begin{align}\label{eq: rec linear AMP residual converge}
    \lim_{t\to\infty}\uplim_{m,n\to\infty} \frac{\|\bell_t^k\|}{\sqrt{m}} + \frac{\|\rb_t^k\|}{\sqrt{n}} = 0.
\end{align}

Next we show that as $t\to\infty$, $\xi_t^k \defeq {\fb_t^k}^\top\fb_\pca^k/m\to
1$ and $\zeta_t^k \defeq {\gb_t^k}^\top\gb_\pca^k/n\to 1$ by using the state
evolution \eqref{eq: rec linear AMP state evolution} of the linear AMP. For any
$t \geq 1$, we have
\begin{align*}
\nu_{t+1} = \EE[U_{t+1}{U_*'}^\top] S'\sqrt{\gamma}= S_u^{-1}
\EE[F_t{U_*'}^\top] S'\sqrt{\gamma} = S_u^{-1} \EE[(\mu_tU_*' + Y_t){U_*'}^\top] S'\sqrt{\gamma} = S_u^{-1}\mu_tS'\sqrt{\gamma}
\end{align*}
where the last equality is due to $Y_t\independent U_*'$, $\EE[Y_t]=0$, and
$\EE[U_*'{U_*'}^\top]=\Id$. Setting $U_1=S_u^{-1}F_0$, this holds also for $t=0$
upon identifying $\mu_0=\mu_\pca$ to match the initialization of $\Fb_0$.
We also have, for $t \geq 1$,
\begin{align*}
\mu_t = \frac{1}{\sqrt{\gamma}}\EE[V_t{V_*'}^\top]S' =
\frac{1}{\sqrt{\gamma}}S_v^{-1} \EE[G_tV_*'] S'= \frac{1}{\sqrt{\gamma}}S_v^{-1}
\EE[(\nu_t V_*'+Z_t){V_*'}^\top]S' = \frac{1}{\sqrt{\gamma}}S_v^{-1}\nu_tS'
\end{align*}
where the last equality is due to $Z_t\independent V_*'$, $\EE[Z_t]=0$, and
$\EE[V_*'{V_*'}^\top]=\Id$. Combining the above two equalities yields
\begin{align*}
\mu_{t+1} = S_v^{-1}S_u^{-1}\mu_t(S')^2. 
\end{align*}
Since $\mu_0 = \mu_\pca$ and $S_uS_v = S^2$, we have $\mu_t\equiv \mu_\pca$ for
all $t \geq 1$. Thus also $\nu_t\equiv S_u^{-1}\mu_\pca S'\sqrt{\gamma} =
\nu_\pca$ for all $t\geq 1$, by (\ref{eq:rec:def:thetauv}) and
(\ref{eq:rec:pcaSolFZequi}). Hence, for each $k\in\{1,\ldots,K\}$, we have
\begin{align*}
    \mu_{\pca,k} = \lim_{m,n \to \infty} \frac{{\fb_\pca^k}^\top \ub_*^k}{m},
\qquad
\mu_{\pca,k} = 
\lim_{m,n\to\infty} \frac{{\fb_t^k}^\top\ub_*^k}{m} = \lim_{m,n\to\infty}
\xi_t^k\frac{(\fb_\pca^k)^\top\ub_*^k}{m} + \frac{(\bell_t^k)^\top\ub_*^k}{m},
\end{align*}
where the left equality follows from Theorem~\ref{thm: recSpike} and the right equality applies $\mu_t \equiv \mu_\pca$.
This further implies
\begin{align*}
    \uplim_{m,n\to\infty}\left|(\xi_t^k-1)
\frac{(\fb_\pca^k)^\top\ub_*^k}{m}\right| \leq \uplim_{m,n\to\infty} \frac{\|\bell_t^k\|}{\sqrt{m}}.
\end{align*}
Since $\lim_{m,n\to\infty}{\fb_\pca^k}^\top\ub_*^k/m = \mu_{\pca,k}\neq 0$ a.s. by Theorem \ref{thm: recSpike}, taking the limit as $t\to\infty$ on both sides, it follows from \eqref{eq: rec linear AMP residual converge} that
\begin{align*}
\lim_{t\to\infty} \uplim_{m,n\to\infty}|\xi_t^k-1| = 0.
\end{align*}
Combining with \eqref{eq: rec linear AMP residual converge}, we have
\begin{align*}
    \lim_{t\to\infty} \uplim_{m,n\to\infty} \frac{\|\fb_t^k - \fb_\pca^k\|}{\sqrt{m}} = 0.
\end{align*}
Applying the same argument, we also get $\lim_{t\to\infty}\uplim_{m,n\to\infty} \|\gb_t^k - \gb_\pca^k\|/\sqrt{n} = 0$.

Finally, the convergence of $\sigma_{st}$ and $\omega_{st}$ can be derived using the exact same argument as we have used for the linear AMP for symmetric matrices. This completes the proof.
\end{proof}

Returning to the auxiliary AMP iterations
$\bU_{-\tau}^{(\tau)},\bG_{-\tau}^{(\tau)},\bV_{-\tau}^{(\tau)},\bF_{-\tau}^{(\tau)},\ldots$
defined by (\ref{eq: rec aux denoiser v}) and (\ref{eq: rec aux denoiser u}),
this linear AMP algorithm characterizes these iterates up to $\bU_1^{(\tau)}$
and $\bG_1^{(\tau)}$. Thus, translating Lemma \ref{lem:rec:linearamp} back to
this indexing and recalling the spectral initializations $\bF_0=\bF_\pca$,
$\bG_0=\bG_1=\bG_\pca$, $\bU_0=\bU_1=\bF_\pca S_u^{-1}$, and $\bV_0=\bG_\pca
S_v^{-1}$, we have
\begin{align}
\begin{aligned}
&\lim_{\tau \to \infty} \uplim_{m,n \to \infty}
\frac{\|\bF_{i}^{(\tau)}-\bF_0\|_\F}{\sqrt{m}}=0,
\quad \lim_{\tau \to \infty} \uplim_{m,n \to \infty}
\frac{\|\bV_{i}^{(\tau)}-\bV_0\|_\F}{\sqrt{n}}=0,\\
&\hspace{1in}\mu_{i}^{(\tau)}=\mu_\pca,\quad \lim_{\tau \to \infty}
\bSigma_{\all,t}^{(\tau)}[i,j]=\Id-\mu_\pca\mu_\pca^\top \text{ for all fixed }
i,j \leq 0, 
\end{aligned}
\label{eq:rec:linearamp1}
\end{align}
and
\begin{align}
&\lim_{\tau \to \infty} \uplim_{m,n \to \infty}
\frac{\|\bG_{i}^{(\tau)}-\bG_0\|_\F}{\sqrt{n}}
=\lim_{\tau \to \infty} \uplim_{m,n \to \infty}
\frac{\|\bG_{i}^{(\tau)}-\bG_1\|_\F}{\sqrt{n}}=0,\nonumber\\
&\lim_{\tau \to \infty} \uplim_{m,n \to \infty}
\frac{\|\bU_{i}^{(\tau)}-\bU_0\|_\F}{\sqrt{m}}
=\lim_{\tau \to \infty} \uplim_{m,n \to \infty}
\frac{\|\bU_{i}^{(\tau)}-\bU_1\|_\F}{\sqrt{m}}=0,\nonumber\\
&\hspace{1in}\nu_{i}^{(\tau)}=\nu_\pca,\quad \lim_{\tau \to \infty}
\bOmega_{\all,t}^{(\tau)}[i,j]=\Id-\nu_\pca\nu_\pca^\top \text{ for all fixed }
i,j \leq 1.\label{eq:rec:linearamp2}
\end{align}

\subsection{Phase II: auxiliary AMP for rectangular matrices}

\begin{proof}[Proof of Theorem~\ref{thm:rec:AMPPCA}]
We show by induction that the following statements hold a.s.\ for each $t\geq 0$.
\begin{enumerate}[label=(G.{\alph*})]
\item $\limtau\uplim_{m,n\to\infty}\|\Ub_t^{(\tau)} - \Ub_t\|_\F / \sqrt{m} = 0$
and $\lim_{m,n\to\infty}\|\Ub_t\|_\F/\sqrt{m}<C_t$ for a constant $C_t>0$. \label{enum:G:u}

\item $(\Ub_*',\Ub_0,\Ub_1,\ldots,\Ub_t,\Fb_0,\ldots,\Fb_{t-1}) \toWtwo (U_*',U_0,U_1,\ldots,U_s,F_0,\ldots,F_{t-1})$, where the limiting distribution is defined in \eqref{eq:rec:limitDistPcaAmp}. \label{enum:G:W2convergence}

\item $\limtau\taubPhi_t = \bPhi_t$ and $\limtau\taubDelta_t = \bDelta_t$. \label{enum:G:PhiDelta}

\item $\limtau \lim_{m,n \to \infty} \|\taubphi_t - \bphi_t\| = 0$.\label{enum:G:phi}

\item $\limtau\uplim_{m,n\to\infty}\|\Gb_t^{(\tau)} - \Gb_t\|_\F / \sqrt{n} = 0$
and $\lim_{m,n\to\infty} \|\Gb_t\|_\F/\sqrt{n}<C_t$ for a constant $C_t>0$.\label{enum:G:g}

\item
$\limtau \taubOmega_t = \bOmega_t$ and $\limtau \taubnu_t = \bnu_t$.\label{enum:G:nuOmega}
\end{enumerate}
and
\begin{enumerate}[label=(F.{\alph*})]
\item $\limtau\uplim_{m,n\to\infty} \|\Vb_t^{(\tau)} - \Vb_t \|/ \sqrt{n} = 0$ and
$\lim_{m,n\to\infty}\|\Vb_t\|_\F/\sqrt{n}<C_t$ for a constant $C_t>0$.\label{enum:F:v}

\item $(\Vb_*',\Vb_0,\Vb_1,\ldots,\Vb_t,\Gb_0,\ldots,\Gb_t) \toWtwo (V_*',V_0,V_1,\ldots,V_t,G_0,\ldots,G_t)$, where the limiting distribution is defined in \eqref{eq:rec:limitDistPcaAmp}. \label{enum:F:W2convergence}

\item $\limtau \taubPsi_t = \bPsi_t$ and $\limtau \taubGamma_t = \bGamma_t$.\label{enum:F:PsiGamma}

\item $\limtau \lim_{m,n \to \infty} \|\taubpsi_t - \bpsi_t\| = 0$.\label{enum:F:psi}

\item
$\limtau\uplim_{m,n\to\infty} \|\Fb_t^{(\tau)} - \Fb_t\|/\sqrt{m} = 0$ and
$\lim_{m,n\to\infty} \|\Fb_t\|_\F/\sqrt{m}<C_t$ for a constant $C_t>0$. \label{enum:F:f}

\item $\limtau \taubSigma_t = \bSigma_t$ and $\limtau \taubmu_t = \bmu_t$.\label{enum:F:muSigma}
\end{enumerate}

Denote by $t^{(G)}, t^{(F)}$ the claims of parts G and F {at 
iteration $t$}. The base cases of $0^{(G)},0^{(F)}$ and $1^{(G)}$ may be verified
from the implications (\ref{eq:rec:linearamp1}) and (\ref{eq:rec:linearamp2})
of Lemma~\ref{lem:rec:linearamp}, similar to the proof of Theorem \ref{thm:sym}.

Let us take $t \geq 1$, suppose by induction that $s^{(F)}$ holds for all
$0\leq s\leq t-1$ and $s^{(G)}$ holds for all $0\leq s \leq t$, and show
$t^{(F)}$. The proofs of $t^{(F.a-F.d)}$ apply the same arguments as in the proof
of Theorem~\ref{thm:sym}, and we omit these for brevity. We provide below the
proofs of $t^{(F.e)}$ and $t^{(F.f)}$.

\paragraph{Part~\ref{enum:F:f}}
We control the difference between $\Fb_t^{(\tau)}$ and $\Fb_t$. By their definitions in \eqref{alg:auxAMPrec} and~$\eqref{eq:rec:AMPPCA}$, applying the triangle inequality yields
\begin{align*}
\frac{\|\Fb_t^{(\tau)} - \Fb_t\|_\F}{\sqrt{m}} 
\leq 
\frac{\|\bX\| \|\Vb_t^{(\tau)} - \Vb_t\|_\F}{\sqrt{m}} +
\frac{1}{\sqrt{m}}\left\|\sum_{i=-\tau}^{t} \Ub_i^{(\tau)}a_{t,i}^{(\tau)\top}  -
\sum_{i=0}^{t} \Ub_i a_{t,i}^\top \right\|_\F.
\end{align*}
The first term vanishes in the limit $m,n \to \infty$ by $t^{(F.a)}$, as
$\lim_{m,n\to\infty}\|\Xb\|$ is finite. 
To show the second term will also converge to zero, we may write analogously to 
(\ref{eq: step 2 f difference 2})
\[\frac{1}{\sqrt{m}}\left\|\sum_{i=-\tau}^{t} \Ub_i^{(\tau)}a_{t,i}^{(\tau)\top}  -\sum_{i=0}^{t} \Ub_i a_{t,i}^\top \right\|_\F
\leq \underbrace{\frac{1}{\sqrt{m}}\left\|\sum_{i=-\tau}^0
\Ub_i^{(\tau)}a_{t,i}^{(\tau)\top} - \Ub_0 a_{t,0}^\top\right\|_\F}_{\cI_1}
+\underbrace{\sum_{i=1}^t \frac{1}{\sqrt{m}}
\left\|\Ub_i^{(\tau)}a_{t,i}^{(\tau)\top}-\Ub_i a_{t,i}^\top
\right\|_\F}_{\cI_2}.\]
The term $\cI_2$ converges to 0 by $t^{(G.a)}$, $t^{(G.d)}$, $t^{(F.d)}$, and
the same argument as in the proof of Theorem \ref{thm:sym}. For $\cI_1$, let us
further write, analogously to (\ref{eq: f diff at 0}),
\[\frac{1}{\sqrt{m}}\left\|\sum_{i=-\tau}^0
\Ub_i^{(\tau)}a_{t,i}^{(\tau)\top} - \Ub_0 a_{t,0}^\top\right\|_\F
\leq \underbrace{\left\|\sum_{i=-\tau}^0 a_{t,i}^{(\tau)}-a_{t,0}\right\|}_{\cI_{1,1}}
\cdot \frac{\|\Ub_0\|_\F}{\sqrt{m}}
+\underbrace{\sum_{i=-\tau}^0 \|a_{t,i}^{(\tau)}\| \cdot
\frac{\|\Ub_i^{(\tau)}-\Ub_0\|_\F}{\sqrt{m}}}_{\cI_{1,2}}.\]
The term $\cI_{1,2}$ converges to 0 by (\ref{eq:rec:linearamp2}) and
Lemma \ref{lemma:ziplock}, following again the same argument as in the proof of
Theorem \ref{thm:sym}. Thus it remains to verify the convergence of $\cI_{1,1}$.

Recall our definition of debiasing coefficients in~\eqref{eq:rec:debiaseCoef} and~\eqref{eq: rec debiascoef aux}:
\begin{align*}
a_{t,0} &= \widetilde{\Ab}_t[t,0] = \sum_{j=0}^\infty \bpsi_t(\bphi_t \bpsi_t)^j[t,0] \odot \tilde\kappa_{2(j+1)}, \\
\taua_{t,i} &= \taubA_{\all, t}[t,i] = \sum_{j=0}^\infty \kappa_{2(j+1)} \bpsi_{\all,t}^{(\tau)} (\bphi_{\all,t}^{(\tau)} \bpsi_{\all,t}^{(\tau)})^j [t,i] 
= \sum_{j=-i}^{-i+t} \kappa_{2(j+1)}
\bpsi_t^{(\tau)}(\bphi_t^{(\tau)} \bpsi_t^{(\tau)})^{j+i} [t,0] S^{2i}
\end{align*}
where the last equality comes from~\eqref{eq: rec psiphipsi identity}.
Therefore, we have
\begin{align*}
\limtau\limmn \sum_{i=-\tau}^0 \taua_{t,i}  - a_{t,0} &= \limtau\limmn
\sum_{i=-\tau}^{0}\sum_{j=-i}^{-i+t} \kappa_{2(j+1)}
\bpsi_t^{(\tau)}(\bphi_t^{(\tau)} \bpsi_t^{(\tau)})^{j+i}[t,0]S^{2i} - a_{t,0}\\
&= \sum_{j=0}^t \bpsi_t(\bphi_t\bpsi_t)^j[t,0] \left(\sum_{i=0}^\infty \kappa_{2(j+i+1)} S^{2i}\right) - a_{t,0} = 0
\end{align*}
by the expression of $a_{t,0}$ and the definition
of $\tilde\kappa_{2(j+1)}$. This shows $t^{(F.e)}$.

\paragraph{Part~\ref{enum:F:muSigma}}
By the definition of $\taubSigma_{\all, t}$ in~\eqref{eq: rec aux Sigma} and its block decomposition, we have
\begin{align*}
\taubSigma_t &= \underbrace{
\sum_{j= 0}^\infty 
\kappa_{2(j+1)} 
\sum_{i=0}^j \left[(\bPsi_{\all,t}^{(\tau)} \bPhi_{\all,t}^{(\tau)})^i \bGamma_{\all,t}^{(\tau)} (\bPhi_{\all,t}^{(\tau)\top} \bPsi_{\all,t}^{(\tau)\top})^{j-i}\right]_t 
}_{\cI_1^{(\tau)}}\\
&\qquad + \underbrace{\sum_{j=0}^\infty \kappa_{2(j+1)} \sum_{i=0}^{j-1} \left[(\bPsi_{\all,t}^{(\tau)}\bPhi_{\all,t}^{(\tau)})^i \bPsi_{\all,t}^{(\tau)} \bDelta_{\all,t}^{(\tau)} \bPsi_{\all,t}^{(\tau)\top} (\bPhi_{\all,t}^{(\tau)\top} \bPsi_{\all,t}^{(\tau)\top})^{j-i}\right]_t}_{\cI_2^{(\tau)}}
\end{align*}
where $[\cdot]_t$ means taking the lower-right $(t+1)K\times(t+1)K$ submatrix.
Similarly, we may decompose $\bSigma_t$ defined in \eqref{def:recSigma} as
$\bSigma_t = \cI_1 + \cI_2$.
To show $\limtau\taubSigma_t = \bSigma_t$, it suffices to show $\limtau
\cI_1^{(\tau)} = \cI_1$ and $\limtau \cI_2^{(\tau)} = \cI_2$.  

For $\cI_1^{(\tau)}$,
by the block decompositions of $\taubPhi_{\all,t}$, $\taubPsi_{\all,t}$, and $\taubDelta_{\all,t}$, we obtain
\begin{align}\label{eq: rec I1tau decompose}
\cI_1^{(\tau)} &= \underbrace{\sum_{j=0}^\infty \kappa_{2(j+1)} \sum_{i=0}^{j}
(\bPsi_{\all,t}^{(\tau)}\bPhi_{\all,t}^{(\tau)})^i_{t-} \bGamma_{--}^{(\tau)}
(\bPsi^\top_{\all,t}\bPhi^\top_{\all,t})^{j-i}_{-t}}_{\widehat\bSigma_t^{(\tau)}}
+ \underbrace{\sum_{j=0}^\infty \kappa_{2(j+1)} \sum_{i=0}^{j} (\bPsi_{t}^{(\tau)}\bPhi_{t}^{(\tau)})^i \bGamma_{t-}^{(\tau)} (\bPsi^{(\tau)\top}_{\all,t}\bPhi^{(\tau)\top}_{\all,t})^{j-i}_{-t}}_{\widetilde\bSigma_t^{(\tau)}} \notag\\
&+ \underbrace{\sum_{j=0}^\infty \kappa_{2(j+1)} \sum_{i=0}^{j} (\bPsi_{\all,t}^{(\tau)} \bPhi_{\all,t}^{(\tau)})^i_{t-} \bGamma_{-t}^{(\tau)} (\bPsi^{(\tau)\top}_{t}\bPhi^{(\tau)\top}_t)^{j-i}}_{\widetilde\bSigma_t^{(\tau)\top}}
+ \underbrace{\sum_{j=0}^\infty \kappa_{2(j+1)} \sum_{i=0}^{j} (\bPsi_{t}^{(\tau)} \bPhi_{t}^{(\tau)})^i \bGamma_{t}^{(\tau)} (\bPsi^{(\tau)\top}_{t}\bPhi^{(\tau)\top}_{t})^{j-i}}_{\widebar\bSigma_t^{(\tau)}}
\end{align}
where we observe that the third term is the transpose of the second term.
Recalling
$\bGamma_t=\widebar\bGamma_t+\widetilde\bGamma_t+\widetilde\bGamma_t^\top
+\widehat\bGamma_t$, we correspondingly decompose
\begin{align}\label{eq: rec I1 decompose}
\cI_1 &= \underbrace{\sum_{j=0}^\infty  (\bPsi_{t}\bPhi_{t})^i \left(\hat\kappa_{2(j+1)}\odot \widehat\bGamma_t - \tilde\kappa_{2(j+1)}\odot \widehat\bGamma_t - \widehat\bGamma_t^\top \odot\tilde\kappa_{2(j+1)} + \kappa_{2(j+1)} \widehat\bGamma_t\right) (\bPsi^\top_{t}\bPhi^\top_{t})^{j-i}}_{\widehat\bSigma_t}\notag\\
&\quad+ \underbrace{\sum_{j=0}^\infty  (\bPsi_{t}\bPhi_{t})^i \left[ (\tilde\kappa_{2(j+1)}-\kappa_{2(j+1)}\Id)\odot (\widehat\bGamma_t + \widetilde\bGamma_t) + (\widehat\bGamma_t + \widetilde\bGamma_t)^\top\odot (\tilde\kappa_{2(j+1)} - \kappa_{2(j+1)}\Id) \right] (\bPsi^\top_{t}\bPhi^\top_{t})^{j-i}}_{\widetilde\bSigma_t} \notag\\
&\quad + \underbrace{\sum_{j=0}^\infty \kappa_{2(j+1)} (\bPsi_{t}\bPhi_{t})^i 
\bGamma_t (\bPsi^\top_{t}\bPhi^\top_{t})^{j-i}}_{\widebar\bSigma_t}.
\end{align}
Let us show that for any $r,c \in \{0,1,\ldots,t\}$,
\begin{align*}
\lim_{\tau\to\infty}\widehat\bSigma_t^{(\tau)}[r,c] =
\widehat\bSigma_t[r,c],\quad
\lim_{\tau\to\infty}\left(\widetilde\bSigma_t^{(\tau)} +
\widetilde\bSigma_t^{(\tau)^\top}\right)[r,c] = \widetilde\bSigma_t[r,c],
\quad \lim_{\tau\to\infty} \widebar\bSigma_t^{(\tau)}[r,c] = \widebar\bSigma_t [r,c].
\end{align*}

\subparagraph{Convergence of $\widehat\bSigma_t^{(\tau)}$.}
First, for $\widehat\bSigma_t^{(\tau)}$, we have
\begin{align*}
\widehat\bSigma_t^{(\tau)}[r,c] &= 
\sum_{j=0}^{\infty}\kappa_{2(j+1)}
\sum_{i=0}^{j}\sum_{\alpha, \beta  =1}^\tau (\taubPsi_{\all,t} \taubPhi_{\all,t})^i[r,-\alpha] \taubGamma_{--}[-\alpha,-\beta] (\bPhi_{\all,t}^{(\tau)\top} \bPsi_{\all,t}^{(\tau)\top})^{j-i}[-\beta,c]\\
&= \sum_{j=0}^\infty \kappa_{2(j+1)} \sum_{i=0}^j \sum_{\alpha=1}^{i \wedge \tau}
\sum_{\beta=1}^{(j-i) \wedge \tau} (\bPsi_t^{(\tau)}\bPhi_t^{(\tau)})^{i-\alpha}[r,0] S^{-2\alpha}\bGamma_{--}^{(\tau)}[-\alpha,-\beta]S^{-2\beta} (\bPhi_t^{(\tau)\top}\bPsi_t^{(\tau)\top})^{j-i-\beta}[0,c]\\
&= \sum_{p,q=0}^\infty (\bPsi_t^{(\tau)}\bPhi_t^{(\tau)})^p[r,0] \left(\sum_{\alpha,\beta=1}^\tau \kappa_{2(p+q+\alpha+\beta+1)} S^{-2\alpha} \bGamma_{--}^{(\tau)}[-\alpha,-\beta] S^{-2\beta}\right) (\bPhi_t^{(\tau)\top}\bPsi_t^{(\tau)\top})^q[0,c] 
\end{align*}
where the second equality follows from \eqref{eq: rec psiphi
identity} and the third equality re-indexes the summation by setting
$p=i-\alpha$ and $q=j-i-\beta$. By following the same argument for the convergence of the corresponding $\widehat\bSigma_t^{(\tau)}$ in the proof
of Theorem \ref{thm:sym}, we obtain that 
\begin{align*}
\lim_{\tau\to\infty} \sum_{\alpha,\beta=1}^\tau \kappa_{2(p+q+\alpha+\beta+1)}S^{-2\alpha}\bGamma_{--}^{(\tau)}[-\alpha,-\beta]S^{-2\beta} &= \sum_{\alpha,\beta=1}^\infty \kappa_{2(p+q+\alpha+\beta+1)}S^{-2\alpha}\bGamma_t[0,0]S^{-2\beta}\\
&= \sum_{\alpha,\beta=1}^\infty \kappa_{2(p+q+\alpha+\beta+1)}S^{-2(\alpha+\beta)}\bGamma_t[0,0]\\
&= (\hat\kappa_{2(p+q+1)}-2\tilde\kappa_{2(p+q+1)} + \kappa_{2(p+q+1)}\Id)\bGamma_t[0,0]
\end{align*}
where the second equality is due to the fact that 
$\bGamma_t[0,0] = \EE[V_0V_0^\top] = S_v^{-2}$
is diagonal, and the third equality follows from
the definitions of $\hat\kappa$ and $\tilde\kappa$ in
\eqref{eq:sym:kappaseries}. Combining the above limit with the convergence of
$\bPsi_t^{(\tau)}$ and $\bPhi_t^{(\tau)}$ in $t^{(G.c)}$ and $t^{(F.c)}$, since $\bGamma_t[0,0] = \widehat\bGamma_t[0,0]$ by definition, we get
\begin{align*}
\lim_{\tau\to\infty} \widehat\bSigma_t^{(\tau)}[r,c] = \sum_{p,q=0}^\infty (\bPsi_t\bPhi_t)^p[r,0] \left( \hat\kappa_{2(j+1)} - 2\tilde\kappa_{2(j+2)} + \kappa_{2(j+1)}\Id \right) \widehat\bGamma_t[0,0] (\bPhi_t^\top \bPsi_t^\top)^q[0,c] = \widehat\bSigma_t[r,c].
\end{align*}

\subparagraph{Convergence of $\widetilde\bSigma_t^{(\tau)}$.}
Applying \eqref{eq: rec psiphi identity} and re-indexing the summation similarly as above,
\begin{align*}
\widetilde\bSigma_t^{(\tau)}[r,c] &= \sum_{j=0}^{\infty}\kappa_{2(j+1)}
\sum_{i=0}^{j} \sum_{\alpha = 0}^t\sum_{\beta  =1}^\tau (\taubPsi_t \taubPhi_t)^i [r,\alpha] \taubGamma_{t-}[\alpha,-\beta] (\bPhi_{\all,t}^{(\tau)\top} \bPsi_{\all,t}^{(\tau)\top})^{j-i}[-\beta,c]\\
&= \sum_{j=0}^{\infty}\kappa_{2(j+1)}
\sum_{i=0}^{j}\sum_{\alpha = 0}^t\sum_{\beta  =1}^{(j-i) \wedge \tau}
(\taubPsi_t \taubPhi_t)^i[r,\alpha] \taubGamma_{t-}[\alpha,-\beta]S^{-2\beta}
(\bPhi_{t}^{(\tau)\top} \bPsi_{t}^{(\tau)\top})^{j-i-\beta}[0,c]\\
&= \sum_{p,q=0}^\infty\sum_{\alpha=0}^t (\bPsi_{t}^{(\tau)}
\bPhi_{t}^{(\tau)})^p[r,\alpha] \left( \sum_{\beta=1}^\tau
\kappa_{2(p+q+\beta+1)} \bGamma_{t-}^{(\tau)}[\alpha,-\beta] S^{-2\beta}\right)
(\bPhi_{t}^{(\tau)\top} \bPsi_{t}^{(\tau)\top})^q [0,c].
\end{align*}
Applying the same argument for the convergence of the corresponding $\widetilde\bSigma_t^{(\tau)}$ in Theorem \ref{thm:sym},
\begin{align*}
    \lim_{\tau\to\infty} \widetilde\bSigma_t^{(\tau)}[r,c] &= \sum_{p,q=0}^\infty \sum_{\alpha=0}^t (\bPsi_t\bPhi_t)^p[r,\alpha] \left(\sum_{\beta=1}^\infty \kappa_{2(p+q+\alpha+\beta+1)} \bGamma_t[\alpha,0] S^{-2\beta}\right) (\bPhi_t^{\top} \bPsi_t^{\top})^q [0,c]\\
    &= \sum_{p,q=0}^\infty \sum_{\alpha=0}^t (\bPsi_t\bPhi_t)^p[r,\alpha] \bGamma_t[\alpha,0] (\tilde\kappa_{2(p+q+1)} - \kappa_{2(p+q+1)}\Id) (\bPhi_t^{\top} \bPsi_t^{\top})^q [0,c]\\
    &= \sum_{p,q=0}^\infty \sum_{\alpha=0}^t \left[(\bPsi_t\bPhi_t)^p (\widehat\bGamma_t + \widetilde\bGamma_t)^\top \odot (\tilde\kappa_{2(p+q+1)} - \kappa_{2(p+q+1)}\Id) (\bPhi_t^{\top} \bPsi_t^\top)^q \right] [r,c]
\end{align*}
where the second equality follows from the definition of $\tilde\kappa$ in \eqref{eq:sym:kappaseries}. Similarly, we also have
\begin{align*}
\lim_{\tau\to\infty} \widetilde\bSigma_t^{(\tau)\top}[r,c]
= \sum_{p,q=0}^\infty \sum_{\alpha=0}^t \left[(\bPsi_t\bPhi_t)^p (\tilde\kappa_{2(p+q+1)}-\kappa_{2(p+q+q)}\Id) \odot (\widehat\bGamma_t + \widetilde\bGamma_t) (\bPhi_t^\top\bPsi_t^\top)^q\right][r,c] 
\end{align*}

\subparagraph{Convergence of $\overline\bSigma_t^{(\tau)}$.}
It directly follows from the convergence of $\bPhi_t^{(\tau)}$,
$\bPsi_t^{(\tau)}$, and $\bGamma_t^{(\tau)}$ in $t^{(G.c)}$ and $t^{(F.c)}$ that
$\lim_{\tau\to\infty} \widebar\bSigma_t^{(\tau)} = \widebar\bSigma_t$.

Collecting the decomposition of $\cI_1^{(\tau)}$ in \eqref{eq: rec I1tau
decompose} and that of $\cI_1$ in \eqref{eq: rec I1 decompose}, and combining
all of the above, we have shown
$\lim_{\tau\to\infty} \cI_1^{(\tau)} = \cI_1$. The convergence of
$\cI_2^{(\tau)}$ can be established similarly by replacing
$(\bPsi_{\all,t}^{(\tau)}\bPhi_{\all,t}^{(\tau)})^i$ and
$\bGamma_{\all,t}^{(\tau)}$ with
$(\bPsi_{\all,t}^{(\tau)}\bPhi_{\all,t}^{(\tau)})^i\bPsi_{\all,t}^{(\tau)}$ and
$\bDelta_{\all,t}^{(\tau)}$, and applying instead the identity 
\eqref{eq: rec psiphipsi identity} in place of \eqref{eq: rec psiphi identity}.
Hence, we conclude that
$\lim_{\tau\to\infty} \bSigma_t^{(\tau)} = \bSigma_t$, which shows $t^{(F.f)}$.\\

Now supposing by induction that $s^{(F)}$ and $s^{(G)}$ hold for
$0 \leq s \leq t$, we may establish similarly $t+1^{(G.a-G.f)}$.
The proof is symmetric to the above, using the identities
\eqref{eq: rec phipsi identity} and \eqref{eq: rec phipsiphi identity} in place
of \eqref{eq: rec psiphi identity} and \eqref{eq: rec psiphipsi identity}, and
we omit this for brevity. This completes the induction.
\end{proof}

\section{Proof for independent initialization}\label{appendix:indInitproof}

The proof of Theorem \ref{thm:symindinit} follows closely the arguments
of \cite{fan2020approximate}, with minor modifications needed to extend the
results of \cite[Theorem 4.3 and Corollary 4.4]{fan2020approximate}
from vector-valued iterates in $\RR^n$ to matrix-valued iterates in
$\RR^{n \times K}$. As in \cite{fan2020approximate}, the theorem is first shown
under an additional nondegeneracy condition of Assumption
\ref{assump:nondegenerate} below, which ensures that the state evolution
covariance matrices $\bSigma_t$ are invertible for all $t \geq 1$. This enables
an inductive argument of partial conditioning on the randomness of the
Haar-orthogonal matrix $\bO$, to establish state evolution
characterizations for the original AMP iterates
$(\bU_1,\ldots,\bU_{t+1},\bZ_1,\ldots,\bZ_t,\bE)$ and also
for the above auxiliary iterates $(\bR_1,\ldots,\bR_t)$ and
\[\blambda=\diag(\bLambda) \in \RR^n.\]
The combinatorial arguments needed to close this induction
are the same as in \cite{fan2020approximate}.
The theorem without Assumption \ref{assump:nondegenerate} is then obtained by
applying this result to a slightly perturbed AMP sequence, and taking the limit
of the perturbation to 0. We describe these arguments in further detail below.

Write the AMP iterations (\ref{eq:AMPz}--\ref{eq:AMPu}) as 
\[\bR_t=\bO\bU_t, \quad \bS_t=\bO^\top \bLambda \bR_t,
\quad \bZ_t=\bS_t-\bU_1 b_{t1}^\top-\ldots-\bU_t b_{tt}^\top,
\quad \bU_{t+1}=u_{t+1}(\bZ_1,\ldots,\bZ_t,\bE).\]
For each $t \geq 1$ and $k \geq 0$, define the $tK \times tK$ matrix
\[\bL_t^{(k)}=\sum_{j=0}^\infty c_{k,j}\bTheta_t^{(j)},
\qquad \bTheta_t^{(j)}=\bTheta^{(j)}[\bPhi_t,\kappa_{j+2}\bDelta_t]\]
where $\{c_{k,j}\}_{k,j \geq 0}$ are the partial moment coefficients defined
in \cite[Section A.1]{fan2020approximate}, and $\bTheta^{(j)}[\cdot,\cdot]$
is as defined in (\ref{eq:Sigma}). We write
\[\bB_t=\sum_{j=0}^\infty \kappa_{j+1}\bPhi_t^j\]
as the large-$n$ limit of $\bb_t$, replacing $\bphi_t$ by its large-$n$ limit
$\bPhi_t$. (To simplify notation, we are using $\bphi_t,\bb_t^\top$ and
$\bPhi_t,\bB_t^\top$ in place of $\bPhi_t,\bB_t$ and
$\bPhi_t^\infty,\bB_t^\infty$ respectively in \cite{fan2020approximate}.
We have also defined
$\bSigma_t,\bTheta_t^{(j)},\bL_t^{(k)}$ directly in the large-$n$ limit,
instead of using the notation
$\bSigma_t^\infty,\bTheta_t^{(j,\infty)},\bL_t^{(k,\infty)}$.)
The identities and proofs of
\cite[Lemmas A.2 and A.3]{fan2020approximate} on the relations between
$\bB_t^\top,\bSigma_t,\bTheta_t^{(j)},\bL_t^{(k)}$ then hold without any modification,
as they rely only on the definitions of these matrices in terms of
$\bDelta_t,\bPhi_t$ and on the combinatorial relations between
$\{c_{k,j}\}_{k,j \geq 0}$.

We first show an extended form of
Theorem \ref{thm:symindinit} under the following additional assumption.

\begin{assumption}\label{assump:nondegenerate}
The random variables $\Lambda$ and $U_1$ satisfy
$\Var[\Lambda]>0$ and $\EE[U_1U_1^\top] \succ 0$ strictly.
Defining $(U_1,\ldots,U_{t+1},Z_1,\ldots,Z_t)$ by the state evolution
(\ref{eq:SEindinit}), for each $t \geq 1$ and any deterministic vector $v \in
\RR^K$, there do not exist constants
$\alpha_1,\ldots,\alpha_t,\beta_1,\ldots,\beta_t$ for which
$v^\top U_{t+1}$ is almost surely equal
to $\alpha_1 \cdot v^\top U_1+\ldots+\alpha_t \cdot v^\top U_t
+\beta_1 \cdot v^\top Z_1+\ldots+\beta_t \cdot v^\top Z_t$.
\end{assumption}

\begin{lemma}\label{lemma:symindinitextended}
Suppose Assumptions \ref{assump:symW}, \ref{assump:symindinit},
and \ref{assump:nondegenerate} hold.
Then almost surely for each fixed $t \geq 1$,
\begin{enumerate}[label=(\alph*)]
\item $\lim_{n \to \infty} (\langle \bU_r^\top \bU_s \rangle)_{r,s=1}^t
=\bDelta_t$ and $\lim_{n \to \infty} \bphi_t=\bPhi_t$.
\item For some random vectors $R_1,\ldots,R_t \in \RR^K$ having finite second
moment, $(\bR_1,\ldots,\bR_t,\blambda) \toWtwo (R_1,\ldots,R_t,\Lambda)$.
Furthermore, for each $k \geq 0$, $\EE[(R_1,\ldots,R_t)^\top \Lambda^k
(R_1,\ldots,R_t)]=\bL_t^{(k)}$.
\item $(\bU_1,\ldots,\bU_{t+1},\bZ_1,\ldots,\bZ_t,\bE) \toWtwo
(U_1,\ldots,U_{t+1},Z_1,\ldots,Z_t,E)$ as described in Theorem
\ref{thm:symindinit}.
\item The matrix
\[\begin{pmatrix} \bDelta_t & \bPhi_t \bSigma_t \\ \bSigma_t\bPhi_t^\top &
\bSigma_t \end{pmatrix}\]
is non-singular.
\end{enumerate}
\end{lemma}
\begin{proof}
The proof is an extension of that of
\cite[Lemma A.4]{fan2020approximate}. We denote $t^{(a),(b),(c),(d)}$ for the
claims of (a--d) up to iteration $t$.

Applying Lemma \ref{lemma:orthogproj}(a) to $\bR_1=\bO\bU_1$, we have
$(\blambda,\bR_1) \toWtwo (\Lambda,R_1)$ where $R_1 \sim
\cN(0,\EE[U_1U_1^\top])$ is independent of $\Lambda$. Then applying Lemma
\ref{lemma:productW2}, $n^{-1}\bR_1^\top \bLambda^k \bR_1 \to m_k \cdot
\EE[U_1U_1^\top]=m_k \cdot \bDelta_1$
where $m_k=\EE[\Lambda^k]$. Then $1^{(a)}$ and $1^{(b)}$
follow as in \cite[Lemma A.4]{fan2020approximate}. Conditional on
$\bR_1=\bO\bU_1$, we may represent the law of $\bS_1$ as
\[\bS_1=\bS_\parallel+\bS_\perp,
\qquad \bS_\parallel=\bU_1(\bR_1^\top \bR_1)^{-1} \bR_1^\top \bLambda \bR_1,
\qquad \bS_\perp=\Pi_{\bU_1^\perp} \tilde{\bO} \Pi_{\bR_1^\perp}^\top \bLambda \bR_1\]
where $\Pi_{\bU_1^\perp},\Pi_{\bR_1^\perp} \in \RR^{n \times (n-K)}$ are
projections orthogonal to the column spans of $\bU_1,\bR_1$, and $\tilde{\bO}$
is a Haar-orthogonal matrix of dimension $n-K$.
We have $n^{-1}\bR_1^\top \bR_1 \to \EE[U_1U_1^\top]=\bDelta_1$,
which is invertible by Assumption \ref{assump:nondegenerate}, and
\[n^{-1}(\Pi_{\bR_1^\perp}^\top \bLambda \bR_1)^\top
(\Pi_{\bR_1^\perp}^\top \bLambda \bR_1)=n^{-1} \bR_1^\top \bLambda^2 \bR_1
-n^{-1}\bR_1^\top \bLambda \bR_1 \cdot
(n^{-1} \bR_1^\top \bR_1)^{-1} \cdot n^{-1} \bR_1^\top \bLambda \bR_1\]
Then applying the same arguments as
\cite[Lemma A.4]{fan2020approximate}, using Lemma \ref{lemma:orthogproj}(a) to
analyze $\bS_\perp$, we obtain
\[(\bZ_1,\bU_1,\bE) \toWtwo (Z_1,U_1,E), \qquad (U_1,E) \independent
Z_1 \sim \cN(0,\kappa_2\bDelta_1).\]
Here, $\kappa_2\bDelta_1=\bSigma_1$.
Since $u_2(Z_1,E)$ is Lipschitz, Lemma \ref{lemma:LipschitzW2} then implies
$1^{(c)}$. Assumption \ref{assump:nondegenerate} and the identities
$\bSigma_1=\kappa_2\bDelta_1$ and $\bPhi_1=0$ then imply $1^{(d)}$.

Now suppose $t^{(a),(b),(c),(d)}$ hold. For $t+1^{(a)}$,
convergence to $\bDelta_{t+1}$ in
$t+1^{(a)}$ is immediate from $W_2$-convergence of $(\bU_1,\ldots,\bU_{t+1})$ in
$t^{(c)}$ and Lemma \ref{lemma:productW2} (applied with $k=0$). Since each
$u_s(\cdot)$ is Lipschitz, its derivatives are bounded and also continuous with
probability 1 with respect to the limit law of $(Z_1,\ldots,Z_t,E)$ in
$t^{(c)}$, by Assumption \ref{assump:symindinit}. Then the convergence
$\bphi_t \to \bPhi_t$ follows also from weak convergence of
$(\bZ_1,\ldots,\bZ_t,\bE)$ in $t^{(c)}$, which shows $t+1^{(a)}$.

For $t+1^{(b)}$, observe that by the $W_2$-convergence in $t^{(c)}$,
weak-differentiability of the Lipschitz function $u_s(\cdot)$, and
\cite[Proposition E.5]{fan2020approximate},
\[n^{-1}(\bZ_1,\ldots,\bZ_t)^\top (\bU_1,\ldots,\bU_t)
\to \bSigma_t \cdot \bPhi_t^\top.\]
Then conditioning $\bO$ on
\[(\bR_1,\ldots,\bR_t)=\bO(\bU_1,\ldots,\bU_t) \text{ and }
\bO(\bZ_1,\ldots,\bZ_t)=\bLambda(\bR_1,\ldots,\bR_t)-(\bR_1,\ldots,\bR_t)\bB_t^\top,\]
the same argument that shows \cite[Eq.\ (A.23)]{fan2020approximate}
yields $(\bR_1,\ldots,\bR_{t+1},\blambda) \toWtwo (R_1,\ldots,R_{t+1},\Lambda)$
where
\begin{equation*}
R_{t+1}=\begin{pmatrix} R_1 & \cdots & R_t & \Lambda R_1 & \cdots & \Lambda
R_t \end{pmatrix}\bUpsilon_t^{-1}\begin{pmatrix}
\bDelta_{t+1}[1:t,\;t+1] \\ \bPhi_{t+1}^\top[1:t,\;t+1] \end{pmatrix}+R_\perp.
\end{equation*}
Here
\[\bUpsilon_t=\begin{pmatrix} \bDelta_t & \bDelta_t \bB_t^\top+\bPhi_t\bSigma_t \\
\bPhi_t^\top & \bPhi_t^\top \bB_t^\top+\Id \end{pmatrix} \in \RR^{2tK \times
2tK},\]
and in the decomposition into blocks of size $K \times K$, 
$\bDelta_{t+1}[1:t,\;t+1]$ and $\bPhi_{t+1}^\top[1:t,\;t+1]$
denote the first $t$ row blocks of the last column block $t+1$ of
$\bDelta_{t+1}$ and $\bPhi_{t+1}^\top$. The random vector $R_\perp
\in \RR^K$ has the Gaussian law
\[R_\perp \sim \cN\left(0,\,\EE[U_{t+1}U_{t+1}^\top]
-\begin{pmatrix} \bDelta_{t+1}[1:t,\;t+1] \\ \bSigma_t
\bPhi_{t+1}^\top[1:t,\;t+1] \end{pmatrix}^\top
\begin{pmatrix} \bDelta_t & \bPhi_t \bSigma_t \\ \bSigma_t \bPhi_t^\top &
\bSigma_t \end{pmatrix}^{-1}
\begin{pmatrix} \bDelta_{t+1}[1:t,\;t+1] \\ \bSigma_t
\bPhi_{t+1}^\top[1:t,\;t+1] \end{pmatrix}\right).\]
For any fixed $v \in \RR^K$, we may check that $v^\top \Cov[R_\perp]v$
is the residual variance of the $L_2$-projection of $v^\top U_{t+1}$ onto the
linear span of $(v^\top Z_1,\ldots,v^\top Z_t,v^\top U_1,\ldots,v^\top U_t)$.
Thus, by Assumption \ref{assump:nondegenerate}, we have analogously to
\cite[Eq.\ (A.24)]{fan2020approximate} that
\begin{equation}\label{eq:Rperpnonsingular}
\Cov[R_\perp] \succ 0
\end{equation}
Note that Lemma \ref{lemma:productW2} implies
\[n^{-1}(\bR_1,\ldots,\bR_{t+1})^\top
\bLambda^k(\bR_1,\ldots,\bR_{t+1}) \to \EE[(R_1,\ldots,R_{t+1})^\top
\Lambda^k(R_1,\ldots,R_{t+1})].\]
Using \cite[Lemmas A.1 and A.2]{fan2020approximate} and the same arguments as in
\cite[Lemma A.4]{fan2020approximate} to analyze
$\EE[(R_1,\ldots,R_t)^\top
\Lambda^k R_{t+1}]$ and $\EE[R_{t+1}^\top \Lambda^k R_{t+1}]$,
we obtain $\EE[(R_1,\ldots,R_{t+1})^\top
\Lambda^k (R_1,\ldots,R_{t+1})]=\bL_{t+1}^{(k)}$ and hence conclude $t+1^{(b)}$.

For $t+1^{(c)}$, we define $\tilde{\bPhi}_t=\bPhi_{t+1}[1:t+1,\,1:t]$ and
$\tilde{\bB}_t^\top=\bB_{t+1}^\top[1:t+1,\,1:t]$ as the first $t$ column blocks of
$\bPhi_{t+1}$ and $\bB_{t+1}^\top$, and condition $\bO$ on
\[(\bR_1,\ldots,\bR_{t+1})=\bO(\bU_1,\ldots,\bU_{t+1}) \text{ and }
\bO(\bZ_1,\ldots,\bZ_t)=\bLambda(\bR_1,\ldots,\bR_{t+1})-(\bR_1,\ldots,\bR_{t+1})\tilde{\bB}_t^\top.\]
Applying again the $W_2$-convergence in $t^{(c)}$ and \cite[Proposition
E.5]{fan2020approximate}, we have
\[n^{-1}(\bZ_1,\ldots,\bZ_t)^\top(\bU_1,\ldots,\bU_{t+1}) \to
\bSigma_t \cdot \tilde{\bPhi}_t^\top.\]
Observe that
\[\begin{pmatrix} \bDelta_{t+1} & \tilde{\bPhi}_t\bSigma_t \\
\bSigma_t \tilde{\bPhi}_t^\top & \bSigma_t \end{pmatrix}\]
is invertible, as its submatrix removing row block $t+1$ and column block $t+1$
is invertible by
$t^{(d)}$, and the Schur-complement of the $(t+1,t+1)$ block is invertible by
(\ref{eq:Rperpnonsingular}). Then applying the same arguments as leading to
\cite[Eq.\ (A.33)]{fan2020approximate}, we get
$(\bU_1,\ldots,\bU_{t+1},\bZ_1,\ldots,\bZ_t,\bE,\bS_{t+1}) \toWtwo
(U_1,\ldots,U_{t+1},Z_1,\ldots,Z_t,E,S_{t+1})$ where
\[S_{t+1}=\begin{pmatrix} U_1 & \cdots & U_{t+1} \end{pmatrix}
\bB_{t+1}^\top[1:t,\;t+1]+\begin{pmatrix} Z_1 & \cdots & Z_t \end{pmatrix}
\bSigma_t^{-1} \cdot \bSigma_{t+1}[1:t,\,t+1]+S_\perp\]
and
\begin{align*}
S_\perp &\sim \cN\Bigg(0,\;\Bigg(\bL_{t+1}^{(2)}[t+1,t+1]
-\begin{pmatrix} \bL_{t+1}^{(1)}[1:t+1,\,t+1] \\
\bL_{t+1}^{(2)}[1:t,\,t+1] \end{pmatrix}^\top \times \\
&\hspace{1in}\begin{pmatrix} \bL_{t+1}^{(0)} & \bL_{t+1}^{(1)}[1:t,\,1:t+1] \\
\bL_{t+1}^{(1)}[1:t,\,1:t+1]^\top & \bL_t^{(2)} \end{pmatrix}^{-1}
\begin{pmatrix} \bL_{t+1}^{(1)}[1:t+1,\,t+1] \\
\bL_{t+1}^{(2)}[1:t,\,t+1] \end{pmatrix}\Bigg)\Bigg).
\end{align*}
Here again, we use $[1:t,\,t+1]$ to denote the first $t$ row blocks of the last
column block $t+1$, and similarly for $[t+1,\,t+1]$ and $[1:t+1,\,t+1]$. Since
$\bZ_{t+1}=\bS_{t+1}-\begin{pmatrix} \bU_1 & \cdots & \bU_{t+1} \end{pmatrix}
\bB_{t+1}^\top[1:t,\,t+1]$, the same computations as in \cite[Lemma
A.4]{fan2020approximate} show convergence of
$(\bU_1,\ldots,\bU_{t+1},\bZ_1,\ldots,\bZ_{t+1},\bE)$ as in $t+1^{(c)}$, where
the covariance of $(Z_1,\ldots,Z_{t+1})$ is exactly $\bSigma_{t+1}$. 
Since $\bU_{t+2}=u_{t+2}(\bZ_1,\ldots,\bZ_{t+1},\bE)$ and $u_{t+2}(\cdot)$ is
Lipschitz, $t+1^{(c)}$ follows from Lemma \ref{lemma:LipschitzW2}.

Finally, for $t+1^{(d)}$, observe from the definition of $\bL_t^{(k)}$ that for
any fixed vector $v \in \RR^K$, $v^\top \Cov[S_\perp]v$ is the residual variance
of the $L_2$-projection of $v^\top \Lambda R_{t+1}$ onto the linear span of
$(v^\top R_1,\ldots,v^\top R_{t+1},v^\top \Lambda R_1,\ldots,v^\top \Lambda
R_t)$. If $v^\top \Cov[S_\perp]v=0$, then this would imply
\[v^\top \Lambda R_{t+1}=\alpha_1 \cdot v^\top R_1+\ldots+\alpha_{t+1} \cdot
v^\top R_{t+1}+\beta_1 \cdot v^\top \Lambda R_1+\ldots+\beta_t \cdot v^\top
\Lambda R_t,\]
and hence $(\Lambda-\alpha_{t+1})v^\top R_\perp=f(v^\top R_1,\ldots,v^\top
R_t,\Lambda)$ for some function $f$. This implies that $R_\perp$ is constant
conditional on $(R_1,\ldots,R_t,\Lambda)$ and the positive-probability event
$\Lambda \neq \alpha_{t+1}$. This contradicts that $R_\perp$ is independent of
$(R_1,\ldots,R_t,\Lambda)$ with a Gaussian law of positive variance. So $v^\top
\Cov[S_\perp]v>0$ implying $\Cov[S_\perp] \succ 0$ strictly. Then $t+1^{(d)}$
follows as in \cite[Lemma A.4]{fan2020approximate}, concluding the proof.
\end{proof}

\begin{proof}[Proof of Theorem \ref{thm:symindinit}]
We may remove Assumption \ref{assump:nondegenerate} by studying a perturbed AMP
sequence and applying a continuity argument, as in \cite[Appendix
D]{fan2020approximate} and \cite{berthier2020state}. Consider the perturbed AMP
iterations
\[\bZ_t^\eps=\bW^\eps \bU_t^\eps-\bU_1^\eps b_{t1}^{\eps\top}-\ldots-\bU_t^\eps
b_{tt}^{\eps\top}, \quad \bU_{t+1}^\eps=u_{t+1}(\bZ_1^\eps,\ldots,\bZ_t^\eps,\bE)
+\eps \bG_{t+1}\]
where $\bW^\eps=\bO^\top \diag(\blambda+\eps \bgamma)\bO$, the vector $\bgamma
\in \RR^n$ has i.i.d.\
$\operatorname{Uniform}(-1,1)$ entries, each $b_{ts}^\eps$ is the version of
$b_{ts}$ defined with the limit free cumulants of $\bW^\eps$ in place of $\bW$,
and $\bG_{t+1} \in \RR^{n \times K}$ is an independent
matrix with i.i.d.\ $\cN(0,1)$ entries in each iteration. Then the same argument
as in \cite[Appendix D]{fan2020approximate} shows
\[\lim_{\eps \to 0} \limsup_{n \to \infty} n^{-1}\|\bZ_t^\eps-\bZ_t\|_F^2=0,
\qquad \lim_{\eps \to 0} \limsup_{n \to \infty}
n^{-1}\|\bU_t^\eps-\bU_t\|_F^2=0,\]
\[\lim_{\eps \to 0}
(\bDelta_t^\eps,\bPhi_t^\eps,\bB_t^{\eps\top},\bSigma_t^\eps)
=(\bDelta_t,\bPhi_t,\bB_t^\top,\bSigma_t)\]
where $\bDelta_t^\eps,\bPhi_t^\eps,\bB_t^\eps,\bSigma_t^\eps$ are the
matrices corresponding to this perturbed AMP sequence. Note that in 
\cite[Appendix D]{fan2020approximate}, the continuous-differentiability of each
function $u_s(\cdot)$ is used only to show the convergence $\bPhi_t \to
\bPhi_t^\infty$ at the end of the proof (i.e.\ $\bphi_t \to \bPhi_t$ in our
notation), and the relaxed condition of Assumption \ref{assump:symindinit}(b) is
sufficient for this statement. On the other hand,
Assumption \ref{assump:nondegenerate} holds for this perturbed AMP sequence, so 
this perturbed sequence is characterized by
Lemma \ref{lemma:symindinitextended}.
Combining these shows Theorem \ref{thm:symindinit}.
\end{proof}

\section{Auxiliary lemmas}\label{sec:aux}

\begin{lemma}\label{lemma:ziplock}
For each $\tau>0$, let $\{x_i^{(\tau)}\}_{0\leq i\leq \tau}$ and
$\{y_i^{(\tau)}\}_{0\leq i\leq \tau}$ be two sequences where $|x_i^{(\tau)}|$
and $|y_i^{(\tau)}|$ uniformly bounded by a constant $C>0$. Suppose, for any $\epsilon>0$, there exists some $\cT>0$ such that for all $\tau>\cT$,
\begin{align*}
    \sum_{i=\cT}^\tau |x_i^{(\tau)}| \leq \epsilon \quad \text{and} \quad
\lim_{\tau\to\infty}|y_i^{(\tau)}| = 0 \text{ for all } 0\leq i\leq \cT.
\end{align*}
Then $\limtau \sum_{i=0}^\tau x_i^{(\tau)} y_i^{(\tau)} = 0$.
\end{lemma}
\begin{proof}
For any $\epsilon>0$, let $\cT$ be as given. Then for all $\tau>\cT$, we have
\begin{align*}
    \bigg|\sum_{i=0}^{\tau} x_i^{(\tau)} y_i^{(\tau)}\bigg| &\leq \max_{0\leq
i\leq\cT} |x_i^{(\tau)}| \cdot \sum_{i=0}^{\cT} |y_i^{(\tau)}| + \max_{\cT<i\leq \tau}|y_i^{(\tau)}| \cdot \sum_{i=\cT}^\tau |x_i^{(\tau)}|\\
    &\leq C\cdot \sum_{i=0}^{\cT} |y_i^{(\tau)}| + C\cdot \sum_{i=\cT}^\tau |x_i^{(\tau)}|
\end{align*}
where $C$ is the given uniform bound. Taking the limit $\tau\to\infty$ on both sides, we obtain
\begin{align*}
    \limsup_{\tau\to\infty} \bigg|\sum_{i=0}^{\tau} x_i^{(\tau)} y_i^{(\tau)}\bigg| \leq C\epsilon.
\end{align*}
Then since $\epsilon$ is arbitrary, the desired result follows.
\end{proof}

\begin{lemma}\label{lemma: aux algebra fact}
For any integers $\tau,t,K\geq 1$, let $\Ab\in\RR^{(\tau+t+1)K\times
(\tau+t+1)K}$ have the block decomposition
\begin{align*}
    \Ab = \begin{pmatrix}
    \Ab_{--} & 0\\
    \Ab_{+-} & \Ab_{++}
    \end{pmatrix}
\end{align*}
where $\Ab_{--}\in\RR^{\tau K\times\tau K}$ and $\Ab_{++} \in \RR^{(t+1)K \times
(t+1)K}$. Suppose $\Ab_{--}$ and $\Ab_{+-}$ have the further block decompositions
\begin{align*}
    \Ab_{--} = \begin{pmatrix}
        0 & 0 & \cdots & 0 & 0\\
        \Bb & 0 & \cdots & 0 & 0\\
        0 & \Bb & \cdots & 0 & 0\\
        \vdots & \vdots & \ddots & \vdots & \vdots\\
        0 & 0 & \cdots & \Bb & 0
    \end{pmatrix}, \quad \Ab_{+-} = \begin{pmatrix}
        0 & \cdots & 0 & \Bb\\
        0 & \cdots & 0 & 0\\
        \vdots & \ddots & \vdots & \vdots\\
        0 & \cdots & 0 & 0
    \end{pmatrix}
\end{align*}
for some $\Bb\in\RR^{K\times K}$. Then, indexing the row and column blocks by
$\{-\tau,-\tau+1,\ldots,t\}$, for all $r=0,\ldots,t$ and $c=1,\ldots,\tau$,
\begin{align*}
    \Ab^j[r,-c] = \begin{cases}
    \Ab_{++}^{j-c}[r,0] \cdot \Bb^{c} & 1\leq c\leq j,\\
    0 & j<c
    \end{cases}
\end{align*}
where $[r,-c]$ denotes the $K\times K$ submatrix corresponding to row block $r$ and column block $-c$.
\end{lemma}
\begin{proof}
Due to the special structure of $\Ab$, it is straightforward to verify (by induction) that for any $j\geq 1$, we have
\begin{align*}
    \Ab^j = \begin{pmatrix}
        \Ab_{--}^j & 0\\
        \sum_{i=0}^{j-1} \Ab_{++}^i \Ab_{+-} \Ab_{--}^{j-i-1} & \Ab_{++}^j
    \end{pmatrix}.
\end{align*}
Moreover, observe that for each $j-i-1 \geq 0$,
\begin{align*}
    \Ab_{+-}\Ab_{--}^{j-i-1} = \begin{pmatrix}
        0 & \cdots & 0 & \Bb^{j-i} & 0 & \cdots & 0\\
        0 & \cdots & 0 & 0 & 0 & \cdots & 0\\
        \vdots & \ddots & \vdots & \vdots & \vdots & \ddots & \vdots\\
        0 & \cdots & 0 & 0 & 0 & \cdots & 0
        \end{pmatrix}.
\end{align*}
Following the convention that $\Ab_{++}^0 = \Id$, we further have
\begin{align*}
    \sum_{i=0}^{j-1} \Ab_{++}^i\Ab_{+-}\Ab_{--}^{j-i-1} = \begin{pmatrix}
        0 & \cdots & 0 & \Ab_{++}^0[:,0]\Bb^j & \cdots & \Ab_{++}^{j-1}[:,0] \Bb
    \end{pmatrix},
\end{align*}
or equivalently, for all $r=0,\ldots,t$ and $c=1,\ldots,\tau$,
\begin{align*}
    \Ab^j[r,-c] = \begin{cases}
    \Ab_{++}^{j-c}[r,0] \cdot \Bb^{c} & 1\leq c\leq j,\\
    0 & j<c.
    \end{cases}
\end{align*}
\end{proof}

\begin{lemma}\label{lemma:productW2}
If $\bLambda=\diag(\blambda) \in \RR^{n \times n}$
and the random variable $\Lambda$ satisfy Assumption \ref{assump:symW},
and $\bR \in \RR^{n \times j}$ is such that $(\bR,\blambda) \toWtwo (R,\Lambda)$
as $n \to \infty$ for a random vector $R \in \RR^j$, then for any $k \geq 0$,
\[\lim_{n \to \infty} \frac{1}{n}\bR^\top \bLambda^k \bR=\EE[\Lambda^k \cdot
RR^\top] \in \RR^{j \times j}.\]
\end{lemma}
\begin{proof}
Fix any $a,b \in \{1,\ldots,j\}$, let $C>\max(|\lambda_-|,|\lambda_+|)$, and
define the function
$f_C(\lambda,r_a,r_b)=\min(\max(\lambda,-C),C)^k \cdot r_ar_b$. Then
$|f_C(\lambda,r_a,r_b)| \leq C^k(r_a^2+r_b^2)$, so by the given Wasserstein-2
convergence, a.s.
\[\frac{1}{n}\sum_{i=1}^n f_C(\lambda_i,r_{a,i},r_{b,i}) \to
\EE[f_C(\Lambda,R_a,R_b)].\]
The left side
coincides with the $(a,b)$ entry of $n^{-1}\bR^\top \bLambda^k \bR$ a.s.\ for
all large $n$, while the right side coincides with that of $\EE[\Lambda^k
\cdot RR^\top]$.
\end{proof}

\begin{lemma}\label{lemma:LipschitzW2}
Suppose $\bZ \in \RR^{n \times j}$ satisfies $\bZ \toWtwo Z$, and $u:\RR^j \to
\RR^k$ is Lipschitz. Then $(\bZ,u(\bZ)) \toWtwo (Z,u(Z))$.
\end{lemma}
\begin{proof}
Let $L$ be the Lipschitz constant of $u$.
For any continuous function $f:\RR^{j+k} \to \RR$ satisfying $|f(z,u)| \leq
C(1+\|z\|^2+\|u\|^2)$, we have
\[|f(z,u(z))| \leq C(1+\|z\|^2+(L\|z\|+u(0))^2)
\leq C'(1+\|z\|^2)\]
for a different constant $C'>0$.
Then $n^{-1}\sum_i f(z_i,u(z_i)) \to \EE[f(Z,u(Z))]$, so the result follows.
\end{proof}

\begin{lemma}\label{lemma:orthogproj}
Fix $J,L \geq 0$ and $K \geq 1$.
Let $\bO \in \RR^{(n-L) \times (n-L)}$ be a random Haar-uniform orthogonal
matrix. Let $\bE \in \RR^{n \times J}$ and $\bV \in \RR^{(n-L) \times K}$ satisfy
$\bE \toWtwo E$ and $n^{-1}\bV^\top \bV \to \Sigma \in \RR^{K \times K}$.
Let $\Pi \in \RR^{n \times (n-L)}$ be any
deterministic matrix with orthonormal columns. Then
almost surely as $n \to \infty$,
\[(\Pi \bO\bV,\bE) \toWtwo (Z,E)\]
where $Z \sim \cN(0,\Sigma) \in \RR^K$ is independent of $E \in \RR^J$.
\end{lemma}
\begin{proof}
Let $\tilde{\bO}$ be the first $K$ columns of $\bO$. Writing the
singular value decomposition $\bV=\bQ\bD\bU^\top$ and applying the equality in law
$\bO\bQ\overset{L}{=}\tilde{\bO}\bU$, we have the equality in law
$\bO\bV \overset{L}{=} \tilde{\bO}(\bV^\top \bV)^{1/2}$,
where $(\bV^\top \bV)^{1/2}=\bU\bD\bU^\top$ is the positive-semidefinite matrix
square-root. Then
introducing $\bZ \in \RR^{n \times K}$ with i.i.d.\ $\cN(0,1)$ entries, 
and applying $\tilde{\bO}\overset{L}{=}\Pi^\top \bZ(\bZ^\top \Pi\Pi^\top
\bZ)^{-1/2}$, also
\[\bO\bV \overset{L}{=} \Pi^\top \bZ(n^{-1}\bZ^\top \Pi\Pi^\top \bZ)^{-1/2}(n^{-1}\bV^\top \bV)^{1/2}.\]
So
\begin{equation}\label{eq:PiOV}
\Pi \bO\bV \overset{L}{=} (\Id-\Pi^\perp)\bZ(n^{-1}\bZ^\top \Pi\Pi^\top
\bZ)^{-1/2}(n^{-1}\bV^\top \bV)^{1/2},
\end{equation}
where $\Pi^\perp=\Id-\Pi\Pi^\top \in \RR^{n \times n}$
is a projection onto a subspace of fixed dimension $L$.
Since $n^{-1}\bZ^\top \Pi\Pi^\top \bZ \to \Id_{K \times K}$ and $n^{-1} \bV^\top
\bV \to \Sigma$, and the matrix square-root is continuous, we obtain
\begin{equation}\label{eq:ZEconvergence}
\Big(\bZ(n^{-1}\bZ^\top \Pi\Pi^\top
\bZ)^{-1/2}(n^{-1}\bV^\top \bV)^{1/2},\;\bE\Big) \toWtwo (Z,E)
\end{equation}
by \cite[Propositions E.1 and E.4]{fan2020approximate}. Now write
$\Pi^\perp=\bU\bU^\top$ where $\bU \in \RR^{n \times L}$ has orthonormal
columns. Then $\Pi^\perp \bZ \overset{L}{=} \bU G$
where $G \in \RR^{L \times K}$ has i.i.d.\ $\cN(0,1)$ entries. Let $u_i$ be the
$i^\text{th}$ row of $\bU$, and $g_j$ be the $j^\text{th}$ column of $G$. Then
a.s.
\[\frac{1}{n}\sum_{i=1}^n (u_i^\top g_j)^2
\leq \frac{1}{n}\sum_{i=1}^n \|u_i\|^2 \cdot \|g_j\|^2
=\frac{L}{n} \cdot \|g_j\|^2 \to 0.\]
This holds for each column $j$ of $\Pi^\perp \bZ\overset{L}{=}\bU G$, so
$\Pi^\perp \bZ \toWtwo 0$. Then combining with (\ref{eq:ZEconvergence})
and applying this to (\ref{eq:PiOV}), we obtain
$(\Pi \bO\bV,\bE) \toWtwo (Z,E)$ by \cite[Proposition E.4]{fan2020approximate}.
\end{proof}

\section*{Acknowledgments}

We would like to thank Marco Mondelli for a helpful discussion about
\cite{mondelli2021pca}. This research was supported in part by NSF DMS-1916198
and DMS-2142476.

\bibliographystyle{alpha}
\bibliography{OAMP}

\end{document}